\documentclass[11pt,reqno]{amsart}
\usepackage{amssymb}
\usepackage{amsfonts}
\usepackage{stmaryrd}
\usepackage{color}
\hoffset=- 2cm \voffset=0cm 
\theoremstyle{plain}
\newtheorem{Thm}{Theorem}

\newtheorem{Prop}[Thm]{Proposition}
\newtheorem{Rem}[Thm]{Remark}
\newtheorem{Lem}[Thm]{Lemma}

\newtheorem{Def}[Thm]{Definition}
\newtheorem{Prob}[Thm]{Problem}
\newtheorem{Ques}[Thm]{Question}

\newcommand {\p}{\partial}
\newcommand{\q}{\quad}
\newcommand{\qq}{\qquad}

\newcommand{\eq}{\begin{equation}}
\newcommand{\eeq}{\end{equation}}

\def\a{\alpha}
\def\b{\beta}

\def\g{\gamma}
\def\k{\kappa}
\def\lam{\lambda}
\def\p{\partial}

\def\O{\Omega}

\def\var{\varepsilon}

\def\lam{\lambda}
\def\var{\varepsilon}

\def\A{\bold A}
\def\B{\bold B}
\def\bv{\bold v}
\def\D{\bold D}

\def\H{\bold H}

\def\u{\bold u}

\def\w{\bold w}

\def\0{\bold 0}

\def\mA{\mathcal A}
\def\mB{\mathcal B}
\def\C{\mathbb C}

\def\mE{\mathcal E}
\def\mF{\mathcal F}
\def\mG{\mathcal G}

\def\mH{\mathcal H}

\def\mK{\mathcal K}
\def\mL{\mathcal L}

\def\mS{\mathcal S}
\def\mU{\mathcal U}
\def\mV{\mathcal V}
\def\mW{\mathcal W}
\def\mZ{\mathcal Z}

\def\q{\quad}
\def\qq{\qquad}
\def\qqq{\qq\qq}
\def\qqqq{\qq\qq\qq\qq}

\def\v{\vskip}

\def\curl{\text{\rm curl\,}}
\def\div{\text{\rm div\,}}

\def\max{\text{\rm max}}
\def\loc{\text{\rm loc}}
\def\dist{\text{dist}}

\errorcontextlines=0 \numberwithin{equation}{section}
\numberwithin{Thm}{section}

\begin{document}
\large

\title[Meissner States]
{Meissner States of Type I\!I Superconductors}

\author{Xing-Bin Pan}

\address{School of Mathematics,
East China Normal University,  and NYU-ECNU Institute of Mathematical Sciences at NYU Shanghai,
Shanghai 200062, P.R. China. xbpan@math.ecnu.edu.cn}

\thanks{ }

\keywords{Ginzburg-Landau system, superconductivity, Meissner state,
superheating field, elliptic equations, quasilinear system, asymptotic behavior}

\subjclass[2010]{82D55, 35B25, 35B40, 35B45, 35J47, 35J57, 35Q55, 35Q56}

\begin{abstract} This paper concerns mathematical theory of Meissner states of a bulk superconductor of type $I\!I$, which occupies a bounded domain $\O$ in $\Bbb R^3$ and is subjected to an applied magnetic field below the critical field $H_{\text{\rm S}}$. A Meissner state is described by a solution $(f,\A)$ of a nonlinear partial differential system called Meissner system, where $f$ is a positive function on  $\O$ which is equal to the modulus of the order parameter, and $\A$ is the magnetic potential defined on the entire space such that the inner trace of the normal component on the domain boundary $\p\O$ vanishes. Such a solution is called a Meissner solution. Various properties of the Meissner solutions are examined, including regularity, classification and asymptotic behavior for large value of the Ginzburg-Landau parameter $\k$. It is shown that the Meissner solution is smooth in $\O$, however the regularity of the magnetic potential outside $\O$ can be rather poor. This observation leads to the ides of decomposition of the Meissner system into two problems, a boundary value problem in $\O$ and an exterior problem outside of $\O$. We show that the solutions of the boundary value problem with fixed boundary data converges uniformly on $\O$ as $\k$ tends to $\infty$, where the limit field of the magnetic potential is a solution of  a nonlinear curl system. This indicates that, the magnetic potential part $\A$ of the solution $(f,\A)$ of the Meissner system, which has same tangential component of $\curl\A$ on $\p\O$, converges to a solution of the curl system  as $\k$ increases to infinity, which verifies that the curl system is indeed the correct limit of the Meissner system in the case of three dimensions.

\

Published in: 

J. Elliptic and Parabolic Equations, {\bf 4} (2) (2018), 441-523.

    DOI 10.1007/s41808-018-0027-0.

    https://link.springer.com/article/10.1007/s41808-018-0027-0

    Springer Nature Sharedit initiative link https://rdcu.be/bbM8F
\end{abstract}
\maketitle

%\NoBlackBoxes

\tableofcontents

\section{Introduction}
\label{Section1}

\subsection{Problems and motivations}

\subsubsection{Mathematical model of Meissner states}\

Below the critical temperature, a type I\!I superconductor undergoes phase transitions as the
applied magnetic field increases. This phenomenon is described by Ginzburg-Landau theory of superconductivity \cite{GL}. In this theory, superconductivity is described
by a complex-valued function $\psi$ called order parameter and a real
vector field $\mA$ called magnetic potential, and $(\psi,\mathcal
A)$ is a critical point of the Ginzburg-Landau functional
$$\mG[\psi,\mA]
=\int_{\O}\Bigl\{\Bigl|{\lam\over\k}\nabla\psi-i\mathcal
A\psi\Bigr|^2+{1\over 2}(1-|\psi|^2)^2\Bigr\}dx+\int_{\mathbb
R^3}|\lam\,\curl\mA-\mH^e|^2dx,
$$
namely, $(\psi,\mA)$ is a solution of the Euler-Lagrange
equations in $\Bbb R^3$ for the functional $\mG$, which is called Ginzburg-Landau system:
\begin{equation}\label{GL}
\left\{\aligned
-&\nabla_{\k\lam^{-1}\mA}^2\psi
=\k^2\lam^{-2}(1-|\psi|^2)\psi\q&\text{ \rm in } \O,\\
&\lam^2\curl^2\mathcal
A=\lam\k^{-1}\Im(\bar{\psi}\nabla_{\k\lam^{-1}\mA}\psi)
 &\text{ \rm in }\O,\\
&\curl^2\mA=\bold 0\qqqq\q&\text{\rm in }\O^c,\\
&(\nabla_{\k\lam^{-1}\mA}\psi)\cdot\nu=0\qqq&\text{\rm on }\p\O,\\
&[\mA_T]=\0,\qq [(\curl\mA)_T]=\bold 0 \q&\text{\rm on } \p\O,\\
&\lam\,\curl\mA-\mH^e\to\0\qq\text{\rm as } |x|\to
\infty,&
\endaligned \right.
\end{equation}
where $\O$ is a domain in $\mathbb R^3$ occupied by the
superconductor, $\O^c=\mathbb R^3\setminus\overline{\O}$, $\nu$ is the unit
outer normal vector to the domain boundary $\p\O$ pointing into $\O^c$, $\mH^e$
is the applied magnetic field satisfying
\footnote{Without the assumption $\curl\mH^e\equiv \0$,
the second
and the third lines in the Euler-Lagrange equations will take the form
$$\lam^2\curl^2\mathcal
A=\lam\k^{-1}\Im(\bar{\psi}\nabla_{\k\lam^{-1}\mA}\psi)+\lam^2\curl\mH^e\
\q\text{\rm in }\O,\q
\curl^2\mA=\lam^{-1}\curl\mH^e\q\text{\rm in }\O^c.
$$}
\eq\label{1.2}
\div\mH^e=0\q\text{and}\q\curl\mH^e=\0\q\text{in }
\mathbb R^3.
\eeq
$\lam$ is the penetration length, $\k$ is the Ginzburg-Landau
parameter and $\k=\lam/\xi$, where $\xi$ is the coherence length.
$\mA_T$,  $(\curl \mA)_T$ and $\mH^e_T$ denote
the tangential component of $\mA$, $\curl \mA$ and
$\mH^e$ on $\p\O$, respectively. $[\,\cdot\,]$ denotes the
jump in the enclosed quantity across $\p\O$, that is,
$$[\B]=\B^+-\B^-,
$$
where $\B^+$, $\B^-$ are the outer and inner trace of $\B$ at
$\p\O$, see section \ref{Section2} for more precise
definition. We use the notation
$$\aligned
\nabla_{\A}\psi=\nabla\psi-i\A\psi,\q
\nabla_\A^2\psi=(\nabla-i\A)^2\psi=
\Delta\psi-i(2\A\cdot\nabla\psi+\psi \div \A)-|\A|^2\psi.
\endaligned
$$
Note that $|\psi|^2$ is proportional to the density of
superconducting electron pairs. $\psi=0$ if the sample is in the
normal state, and $\psi\neq 0$ if the sample is in the Meissner state.
A superconductor in a weak magnetic field will be in the Meissner
state, namely $\psi$ does not vanish. If the applied magnetic field is below the first critical field $H_{C_1}$ then the Meissner state is the global minimizer of the Ginzburg-Landau functional (see \cite{GL, dG2, SST, T} for earlier physical literature and \cite{SS} and the references therein for mathematical study). When the applied magnetic field is higher then $H_{C_1}$ but still below the second critical field $H_{\text{\rm S}}$, the Meissner state is locally stable (here we omit the precise statement of the meaning of the local stability).  If the applied
field increases further and reaches the third critical field $H_{\text{\rm sh}}$, then vortices (the zero points of
$\psi$) nucleate in the sample and the sample turns into the mixed
state (see \cite{dG1, MS, Fn, Kra, FP1, FP2}).

Now we recall the mathematical model of Meissner states derived in \cite{C1, C2, C3},  see also \cite{Mon} and \cite[Section 2]{P3}.
Let us start with a solution  $(\psi,\A)$ of \eqref{GL} which describes a Meissner state of a superconductor occupying a bounded and simply-connected domain $\O$ in $\Bbb R^3$, so $\psi\neq 0$, and we can write
$$\psi=fe^{i\chi},\q
\mA=\A+{\lam\over\k}\nabla\chi,
$$
where $f>0$, and $\chi$ is a smooth function. Under the assumption $\curl\mH^e\equiv\0$,
from \eqref{GL} we derive equations for $(f,\A)$ in $\Bbb R^3$:
\begin{equation}\label{eqfA}
\left\{\aligned
-&{\lam^2\over\k^2}\Delta f=(1-f^2-|\A|^2)f\qq\q&\text{\rm in }\O,\\
&\lam^2\curl^2\A+f^2\A=\bold 0\qqq\q\,&\text{\rm in }\O,\\
&\curl^2\A={\bold 0}\qqq\qqq\q&\text{ \rm in }\O^c,\\
&{\p f\over\p\nu}=0,\q [\A_T]={\bold 0},\q
[(\curl\A)_T]={\bold 0}\;&\text{ \rm on }\p\O,\\
\endaligned\right.
\end{equation}
with a condition at infinity
\begin{equation}\lam\,\curl\A-\mH^e\to {\bold 0} \q\text{\rm as }
|x|\to\infty, \label{1.4}
\end{equation}
and a condition on $\p\O$
\eq\label{1.5}
\nu\cdot\A^-=0\q\text{\rm on}\;\;\p\O,
\eeq
where $\nu\cdot\A^-$ denotes the inner normal trace of $\A$ at $\p\O$, see the definition given in Section 2. One feature of this problem is that, besides a boundary condition for $f$, it includes a two-side continuity condition for the tangential component $\A_T$ and for $(\curl\A)_T$, and an one-side condition for the inner normal trace  $\nu\cdot\A^-$ on $\p\O$. We call a solution $(f,\A)$ of \eqref{eqfA} with $f>0$ on $\O$ and satisfying \eqref{1.5} a  {\it Meissner solution}, see Definition \ref{Def3.1} below.\footnote{Let us mention that in our paper the name ``Meissner solution" is used only for some solutions of \eqref{eqfA} and \eqref{dec-eqfA-1}, and their equivalent systems. Please note that in literature the name ``Meissner solution" is used for all solutions $(\psi,\A)$ of the Ginzburg-Landau system \eqref{GL} such that $|\psi(x)|>0$ on $\O$, see for instance \cite[p.1376]{BCM}.}

\subsubsection{Asymptotic limit as $\k\to\infty$}\

We shall examine behavior of Meissner solutions when $\k$ is large.
To get some information of the limiting behavior of the Meissner solutions, we begin with a formal analysis as in \cite{C2}. We fix
$\lam$ and let $\k\to\infty$. If we ignore the boundary condition of $f$ on $\p\O$, then formally we should have
${\lam^2\over\k^2}\Delta f(x)\sim 0$ for $x\in\O$. From this and the
first equation in \eqref{eqfA} we have
$f^2(x)\sim 1-|\A(x)|^2$, then replacing $f^2(x)$ by $1-|\A(x)|^2$ in the second equation in
\eqref{eqfA}, we reach a semilinear curl system on $\Bbb R^3$:
\begin{equation}\label{eqA}
\left\{\aligned
-&\lam^2\curl^2\A=(1-|\A|^2)\A\q\;\;&\text{\rm in }\O,\\
&\curl^2\A={\bold 0}\qqq&\text{ \rm in }\O^c,\\
&[\A_T]={\bold 0},\q [(\curl\A)_T]={\bold
0}\q &\text{ \rm on }\p\O,
\endaligned\right.
\end{equation}
which also includes the two-side continuity condition for the tangential components $\A_T$ and $(\curl\A)_T$ on $\p\O$.
It is natural to ask
\begin{Ques}\label{Ques1.1} Is Eq. \eqref{eqA} indeed the correct  limit of Eq. \eqref{eqfA} for Meissner solutions? More precisely, if $(f_\k,\A_\k)$ is a Meissner solution of \eqref{eqfA}-\eqref{1.4}-\eqref{1.5}, is it true that $\A_\k$ sub-converges to a solution of \eqref{eqA} as $\k\to\infty$?
\end{Ques}

The 2-dimensional version of Question \ref{Ques1.1} has been solved by Bonnet, Chapman and Monneau in \cite{BCM}.
In this paper  we work on the 3-dimensional problem.
Before going to study this question, let us look at a boundary value problem (BVP for short) in $\O$:
\begin{equation}\label{eqA-Omega}
\left\{\aligned
-&\lam^2\curl^2\A=(1-|\A|^2)\A\q&\text{\rm in }\O,\\
&\lam(\curl\A)_T=\mH_T\q&\text{\rm on }\p\O.
\endaligned\right.
\end{equation}
BVP \eqref{eqA-Omega} is deduced from \eqref{eqA} in the following sense: If $\A$ is a solution of \eqref{eqA}-\eqref{1.4}, then the restriction of $\A$ on $\overline{\O}$ solves \eqref{eqA-Omega} with the boundary data $\mH_T=(\lam\,\curl\A)_T$.
BVP \eqref{eqA-Omega} has been studied by
Chapman \cite{C1, C2, C3}, Berestycki, Bonnet and Chapman \cite{BBC},
Bonnet, Chapman and Monneau \cite{BCM}, Bolley and Helffer \cite{BH}, Pan and Kwek \cite{PK} in
the two dimensional case, and by Monneau \cite{Mon}, Bates and Pan
\cite{BaP}, Lieberman and Pan \cite{LiP}, Xiang \cite{X} in the three dimensional
case. Also see surveys \cite{P1, P3, P8} on \eqref{eqA-Omega} and related problems, and \cite{P2, P5, P7} for the Meissner model of anisotropic superconductors. Note that in the two dimensional case, if $\O$ is bounded and simply-connected, then the problem \eqref{eqA} in $\Bbb R^2$ is equivalent to BVP \eqref{eqA-Omega} with $\mathcal
H_T=\mH^e_T$, and in this case it has been proved in
\cite{BCM} that the Meissner solutions converge to a solution of
\eqref{eqA-Omega} as $\k\to\infty$.
 In the three dimensional case, it has been proved
in \cite[p.575, Theorem $1'$]{BaP}\footnote{Also see \cite[Theorem 1]{BaP}, which is stated for the equivalent system for $\H=\lam\,\curl\A$.} that, if $\O$ is a bounded and simply-connected domain
without holes and with a $C^4$ boundary, and if the following condition holds:\footnote{If $\A$ is a solution of \eqref{eqA} and $\mH=\lam\,\curl\A$, then the condition $\nu\cdot\curl\mH_T=0$ is natural, see the explanation in \cite[Remark 1.4]{BaP} and also see Lemma \ref{Lem-equiv-dec-eqfA} in this paper. Nevertheless, existence and regularity of solutions to \eqref{eqA-Omega} without this extra condition has been obtained in \cite{LiP}.

The condition $\|\mathcal
H^e_T\|_{C^0(\p\O)}<\sqrt{5/18}$  is optimal for existence of stable solutions for all small $\lam$ (see \cite{BaP}), and it has been shown that
$$\aligned
H_{S}(\mathbb R^2_+)=\sqrt{5\over
18}={\sqrt{5}\over 3}H_C(\mathbb R^2_+),\\
\endaligned
$$
see V. Galaiko \cite{Ga}, L. Kramer \cite{Kra} and Chapman
\cite{C2}. Note that in \cite{PK} we wrote $H_{\text{\rm S}}(\mathbb R^2_+)$ by
$H_{\text{\rm sh}}(\mathbb R^2_+)$.}
$$
\mH_T\in C^{2+\a}(\p\O,\Bbb R^3),\q
\nu\cdot\mH_T=0,\,\q \nu\cdot\curl\mathcal
H_T=0\q\text{\rm on}\;\;\p\O,\q \|\mathcal
H_T\|_{C^0(\p\O)}<\sqrt{5\over 18},
$$
then for all small $\lam$, \eqref{eqA-Omega} has a unique solution $\A\in
C^3_{\loc}(\O,\mathbb R^3)\cap C^{2+\a}(\overline{\O},\mathbb R^3)$, and it
satisfies
\begin{equation}\label{cond-1.8}
\|\A\|_{L^\infty(\O)}\ <  {1\over \sqrt{3}}.
\end{equation}

From these results, it is natural to expect that,
if the Meissner solutions of problem \eqref{eqfA}-\eqref{1.4} satisfy
some conditions which are comparable with \eqref{cond-1.8}, then their
restriction on $\overline{\O}$ converges to a solution of \eqref{eqA} as
$\k\to\infty$. In this paper we are able to verify a weaker version of this observation, see Theorem \ref{Thm4.9}.

\subsection{Mathematical challenges of the Meissner system}\

Eq.\eqref{eqfA} and \eqref{eqA} are derived from the Ginzburg-Landau system \eqref{GL}, however their mathematical structures are different to that of \eqref{GL}, which causes new difficulties in the study of solvability.

\subsubsection{Existence of solutions}\

The Ginzburg-Landau functional $\mG$ enjoys gauge invariance and one may always resume  compactness by working in the spaces where the magnetic potentials are divergence-free, and obtain solutions of \eqref{GL} by applying the standard variational methods. On the other hand, the energy functional associated with a solution $(f,\A)$ of problem \eqref{eqfA}-\eqref{1.4} is
$$\aligned &\mE[f,\A]\equiv \mathcal
G[fe^{i\chi},\A+{\lam\over\k}\nabla\chi]\\
=&\int_{\O}\Bigl\{{\lam^2\over\k^2}|\nabla f|^2+|f \A|^2+{1\over
2}(1-|f|^2)^2\Bigr\}dx+\int_{\mathbb R^3}|\lam\,\curl\A-\mathcal
H^e|^2dx.
\endaligned
$$
$\mE$ is not convex and does not enjoy the gauge invariance, and we are not able to find solutions by directly applying variational methods to $\mE$. The same difficulty exists for \eqref{eqA}. See more discussions in \cite[Subsection 2.1]{BCM} on the difficulties of \eqref{eqfA}. We also refer to \cite[p.576]{BaP} for the discussion on the mathematical difficulty of BVP \eqref{eqA-Omega}.

\subsubsection{The continuity requirements at $\p\O$}\

The requirements of continuity of the tangential components  $\A_T$ and $(\curl\A)_T$ is a key feature of problems \eqref{eqfA} and \eqref{eqA}. If one of the two continuity requirements is dropped, then the question of existence of solutions becomes much easier.

To see this, let us drop the requirement $[(\curl\A)_T]=\0$, then we can find solutions of \eqref{eqfA} as follows:

Step 1. Given a tangential vector field $\mB_T$ on $\p\O$ which satisfies some necessary conditions for solvability (for instance $\mB_T=\mH^e_T$), one can find $(f,\A^i)$ on $\O$ which satisfies the first two equations in \eqref{eqfA} in $\O$ and satisfies the boundary conditions
$${\p f\over\p\nu}=0,\q \lam(\curl\A^i)_T^-=\mB_T\q\text{on }\p\O,
$$
see \eqref{dec-eqfA-1} below.

Step 2. Then we solve the following exterior problem in $\O^c$:
$$\curl^2\A^o=\0\q\text{in }\O^c,\q (\A^o)^+_T=(\A^i)_T\q\text{on }\p\O,\q \lam\,\curl\A^o\to \mH^e\q\text{as }|x|\to\infty.
$$

We can show that both Steps 1 and 2 can be solved. Then we define a vector field $\A$ on $\Bbb R^3$ by letting $\A=\A^i$ in $\O$ and $\A=\A^o$ in $\O^c$. $(f,\A)$ satisfies \eqref{eqfA}-\eqref{1.4} except the continuity requirement $[(\curl\A)_T]=\0$.

On the other hand, if we drop the requirement $[\A_T]=\0$ from \eqref{eqfA}, then Step 1 is the same as above, and Step 2 is to solve
$$\curl^2\A^o=\0\q\text{in }\O^c,\q \lam(\curl\A^o)^+_T=\mB_T\q\text{on }\p\O,\q \lam\,\curl\A^o\to \mH^e\q\text{as }|x|\to\infty.
$$
Again we can solve these two steps and get $(f,\A)$ which satisfies \eqref{eqfA}-\eqref{1.4} except the requirement $[\A_T]=\0$.

In contrast, with both the two continuity conditions required, solvability of systems \eqref{eqfA} and \eqref{eqA} is much harder.

\subsubsection{Lack of control on divergence}\

One of the difficulties of \eqref{eqfA}, \eqref{eqA} and \eqref{eqA-Omega} is lack of control
on $\div\A$. Without control on divergence of the magnetic potential, we are not able to control derivatives of the solutions, hence not able to get higher regularity and a priori estimates of the weak solutions. Recall that, when studying \eqref{eqA-Omega}, to overcome this difficulty,  Chapman \cite{C2} introduce a
system for $\H=\lam\,\curl \A$, and solutions of the new system satisfy automatically the divergence-free condition.
Following this idea, we proved in \cite{P3} that, if $\A$ is a
solution of \eqref{eqA}-\eqref{cond-1.8}, then
$\H=\lam\,\curl \A$ solves a quasilinear system
\begin{equation}\label{eqH}
\left\{\aligned
-&\lam^2\curl \bigl[F(\lam^2|\curl \H|^2)\curl \H\bigr]=\H\q&\text{ \rm in } \O,\\
&\curl\H=\0,\qq \div\H=0\q&\text{ \rm in }\O^c,\\
&[\H_T]=\0\qqq\qqq\;\,&\text{ \rm on }\p\O,\\
\endaligned\right.
\end{equation}
and the following estimate holds:
\begin{equation}\label{cond-1.10}
\lam\|\curl \H\|_{L^\infty(\O)}< \sqrt{4\over 27}.
\end{equation}
Here the function $F$ is determined by
$$v=F(t^2)t\q\text{if and only if}\q t=(1-v^2)v,\q F(0)=1.
$$
$F$ is uniquely defined for $0\leq t\leq \sqrt{4/27}$, i.e., for $0\leq v\leq 1/\sqrt{3}$.\footnote{More precise description of regularity of the function $F$ is given in \cite[Lemma 2.2]{P3}.}
Similarly if $(f,\A)$ is a solution of \eqref{eqfA}-\eqref{cond-1.8}, and if we let
$\H=\lam\,\curl\A$,  from the second equation of \eqref{eqfA} we get
$$\A=-\lam\,f^{-2}\curl\H
$$
in $\O$, hence
$(f,\H)$ is a solution of the following
system
\begin{equation}\label{eqfH}
\left\{\aligned
-&{\lam^2\over\k^2}\Delta f=(1-f^2-\lam^2f^{-4}|\curl\H|^2)f\q&\text{\rm in }\O,\\
&\lam^2\curl(f^{-2}\curl\H)+\H=\bold 0\q\;\;\,&\text{\rm in }\O,\\
&\curl\H={\bold 0},\qq \div\H=0\qq&\text{\rm in }\O^c,\\
&{\p f\over\p\nu}=0,\qq [\H_T]={\bold 0}\qqq&\text{\rm on }\p\O,
\endaligned\right.
\end{equation}
and condition \eqref{1.4} is written as
\begin{equation}\label{cond-1.12}
\H-\mH^e\to {\bold 0}
\q\text{\rm as } |x|\to\infty.
\end{equation}
On the other hand, a solution of  \eqref{eqfH}-\eqref{cond-1.10} (resp. of \eqref{eqH}-\eqref{cond-1.10}) satisfying certain continuity conditions yields a solution of \eqref{eqfA}-\eqref{cond-1.8} (resp. of \eqref{eqA}-\eqref{cond-1.8}). For more details see \cite[Lemma 2.1, Lemma 3.3]{P3}. Therefore in the following, for our convenience, we shall call \eqref{eqfH} (resp. \eqref{eqH}) an equivalent system with \eqref{eqfA} (resp. with \eqref{eqA}), although the meaning of ``equivalence" needs to be understand carefully.
As \eqref{eqfH} and \eqref{eqH} have better structure than \eqref{eqfA} and \eqref{eqA} in the sense that solutions of \eqref{eqfH} and of \eqref{eqH} satisfy the divergence-free condition $\div\H=0$ both in $\O$ and in $\O^c$, which provides possibility to control derivatives of the solutions, so we study first \eqref{eqfH} and \eqref{eqH}.

\subsection{Outlines} \label{Subsection1.3}\

In Section 2 we collect some preliminary materials which will be used frequently in this paper, including spaces of vector fields, the div-curl-gradient inequalities, and a priori estimates of solutions of a linear Maxwell's system.

In Section 3 we study properties of solutions of \eqref{eqfA}. The main result in this section is Theorem \ref{Thm-reg-fA}, which gives regularity and a priori estimates of the weak Meissner solutions $(f,\A)$ of \eqref{eqfA}. We will see that $(f,\A)$ is smooth on $\overline{\O}$, however the regularity of $\A$ in $\O^c$ can be rather poor, and in general we only have $\A\in H^1_{\loc}(\O^c,\Bbb R^3)$. This is partially due to the fact that the definition of weak solutions to \eqref{eqfA} only requires continuity of the tangential components $\A_T$ and $(\curl\A)_T$ on $\p\O$ but allows the normal components $\nu\cdot\A$ and $\nu\cdot\curl\A$ be discontinuous across $\p\O$. This observation leads to the unusual-looking definition  of the classical solutions to \eqref{eqfA} in Definitions \ref{Def3.7} and \ref{Def3.8}, and leads to the idea of decomposition of \eqref{eqfA} into two problems: a BVP \eqref{dec-eqfA-1} for $(f,\A)$ in $\O$, and an exterior problem \eqref{dec-eqfA-2} for $\A$ on $\O^c$.

In Section 4 we study BVP \eqref{dec-eqfA-1}. Regularity of weak solutions of \eqref{dec-eqfA-1} is stated in Proposition \ref{Prop4.3}.
Existence of solutions $(f,\A)$ is proved in Proposition \ref{Prop4.6}, where we work on an equivalent BVP \eqref{eq4.7} for $(f,\H)$ with $\H=\lam\,\curl\A$, as \eqref{dec-eqfA-1} does not provide control on divergence of $\A$. The main result in this section is Theorem \ref{Thm4.9}, which verifies that \eqref{eqA-Omega} is the correct limit of \eqref{dec-eqfA-1} for the Meissner solutions. More precisely, Theorem \ref{Thm4.9} shows that, for each value of $\k$,  \eqref{dec-eqfA-1} has a classical Meissner solution $(f_\k,\A_\k)$, and $(f_\k,\A_\k)$ uniformly converges on $\overline{\O}$ to $(f_\infty,\A_\infty)$ as $\k\to\infty$, where $\A_\infty$ is a solution of \eqref{eqA-Omega} and
$f_\infty(x)=(1-|\A_\infty(x)|^2)^{1/2}.$
Recall that \eqref{dec-eqfA-1} is the restriction on $\O$ of the full Meissner system \eqref{eqfA}, and \eqref{eqA-Omega} is the restriction on $\O$ of \eqref{eqA}, Theorem \ref{Thm4.9} actually says that, the magnetic potential part $\A$ of the Meissner solution $(f,\A)$ of \eqref{eqfA}, of which the tangential component $(\curl\A)_T$ has same value on $\p\O$, converge uniformly on $\overline{\O}$ to a solution of the semilinear curl system \eqref{eqA}, hence it gives an answer to Question \ref{Ques1.1} in the three-dimensional case positively.

In Section 5 we study the exterior problem \eqref{dec-eqfA-2}. Existence and classification of weak solutions are given in Theorem \ref{Thm5.3}.

In Section 6 we study the limiting system \eqref{eqA}. We first examine the equivalent system \eqref{eqH}, and derive existence and classification of solutions (see Lemmas \ref{Lem6.2}, \ref{Lem6.3}). Then we discuss existence of classical Meissner solutions of \eqref{eqA} in Theorem \ref{Thm6.7}, where a solvability condition is given in \eqref{cond-6.15}, which can also be represented in terms of the Dirichlet-to-Neumann type operators $\Gamma$ and $\Sigma$, see Definition \ref{DefE.4} and \eqref{solvability-1} in Appendix E.

In section 7 we examine existence of solutions to the full Meissner system \eqref{eqfA}. We first consider the equivalent system \eqref{eqfH}, for which the  precise meaning of equivalence is carefully stated in Lemma \ref{Lem7.1}, and existence of solutions to \eqref{eqfH} is given in Lemma \ref{Lem7.2}.
Then we go back to \eqref{eqfA}, and in Theorem \ref{Thm7.5} we discuss existence of classical Meissner solutions of \eqref{eqfA}, where a solvability condition is given in \eqref{cond-7.8}, which can also be represented using a Dirichlet-to-Neumann type operator $\Pi$, see Definition \ref{Def7.6} and \eqref{cond-7.8}.
Combining Theorem \ref{Thm7.5} in this paper with \cite[Theorem 1]{BaP} we have a better understanding
on the Meissner solutions of problem \eqref{eqfA}-\eqref{1.4} for small $\lam$ and large $\k$.

In this paper we use frequently the results and techniques developed for Maxwell's equations and div-curl systems, in particular the div-curl-gradients inequalities, which can be found in various references including \cite{DaL1, DaL3, Ce, GR, Sc,  BW, W, KY, MMT, MP1, MP2, Pi, AuA, CD, Co, ABD, AS}. We also use frequently the results on exterior problems from \cite{NW}.
Finally we mention that nonlinear systems involving operator $\curl$ have been studied by many authors in the recent years, see for instance \cite{BF, Jo, Y1, Y2, P5, P6} and the references therein.

\v0.05in

{\bf Acknowledgements.}  This work was partially supported
by the National Natural Science Foundation of China grants no.  11671143 and no. 11431005.

\section{Preliminaries}\label{Section2}

\subsection{Spaces of vector fields}\

Let $\O$ be a bounded domain in $\Bbb R^3$ with a $C^1$ boundary. We use $\nu$ to denote the unit outer normal vector of $\p\O$ which
points to the outside of $\O$, and denote $\nu_{\p\O^c}=-\nu$.
For a function $u(x)$ defined in a neighborhood of $\p\O$, let
$u|_\O$ and $u|_{\O^c}$ denote the restrictions of $u$ on
$\O$ and on $\O^c$ respectively. We define the inner trace $u^-$ and
outer trace $u^+$   on $\p\O$ by
$u^-=\text{trace of $u\bigr|_\O$ on $\p\O$}$ and $u^+=\text{trace of
$u\bigr|_{\O^c}$ on $\p\O$}$
if they exist, and define the jump of $u$ by $[u]=u^+-u^-$.

For a vector field $\A$ defined in $\O$, the trace, tangential trace and normal trace
of $\A$ on $\p\O$, if  exist, are denoted by $\A$, $\A_T$   and $\nu\cdot\A$
respectively.\footnote{$\A_T$ is also denoted by $(\nu\times\A)\times\nu$.} These traces are also called as {\it inner trace}, {\it inner tangential trace}
and {\it inner normal trace}, and also denoted by $\A^-$, $\A_T^-$ and $\nu\cdot\A^-=(\nu\cdot\A)^-$
respectively. For a vector field $\A$ defined in $\O^c$, the trace, tangential trace and normal trace
of $\A$ on $\p\O$, if exist, are called {\it outer trace}, {\it outer tangential trace}
and {\it outer normal trace}, and denoted by $\A^+$, $\A_T^+$ and $\nu\cdot\A^+=(\nu\cdot\A)^+$
respectively. We write
$$[\nu\cdot\A]=\nu\cdot\A^+-\nu\cdot\A^-=(\nu\cdot\A)^+-(\nu\cdot\A)^-,\q [\A_T]=\A_T^+-A_T^-.
$$

We use $C^{k+\a}(\overline{\O})$, $L^p(\O)$ and $H^k(\O)$ to denote the H\"older spaces, Lebesgue spaces and Sobolev spaces for real valued functions,
$C^{k+\a}(\overline{\O},\Bbb C)$, $L^p(\O,\Bbb C)$ and $H^k(\O,\Bbb C)$ to denote the corresponding spaces of complex-valued functions,
$C^{k+\a}(\overline{\O},\Bbb R^3)$, $L^p(\O,\Bbb R^3)$ and $H^k(\O,\Bbb R^3)$ to denote the spaces of vectors fields. However the norms both for scalar functions and vector fields will be denoted by $\|\cdot\|_{C^{k+\a}(\overline{\O})}$, $\|\cdot\|_{L^p(\O)}$ and $\|\cdot\|_{H^k(\O)}$.
We write
$$\aligned
\mH(\O,\text{\rm div})=&\{\B\in L^2(\O,\mathbb
R^3):~\div\B\in L^2(\O)\},\\
\mH(\O,\text{\rm curl})=&\{\B\in L^2(\O,\mathbb
R^3):~\curl\B\in L^2(\O,\mathbb
R^3)\}.
\endaligned
$$
If $D$ is an
unbounded domain (for example $D=\O^c$ or $D=\mathbb R^3$), we define
\begin{equation}\label{sp-Hloc}
\aligned
\mH_{\loc}(D,\text{\rm div})=&\{\B\in L^2_{\loc}(D,\mathbb
R^3),\;
\div\B\in L^2_{\loc}(D)\},\\
\mH_{\loc}(D,\text{\rm curl})=&\{\B\in L^2_{\loc}(D,\mathbb
R^3),\;\curl\B\in L^2_{\loc}(D,\mathbb R^3)\}.\\
\endaligned
\end{equation}
Recall the following decomposition (see \cite[section 4.1]{DaL3}):
\begin{equation}\label{dec-L2}
L^2(\O,\mathbb R^3)=\text{\rm grad} H_0^1(\O)\oplus \mH(\O,\div0)
=\text{\rm grad} H^1(\O)\oplus \mH_{n0}(\O,\div0).
\end{equation}
For a vector field $\A\in\mathcal
H_{\loc}(\mathbb R^3,\div)$,
$[\nu\cdot\A]=\nu\cdot\A^+-\nu\cdot\A^-$ belongs to
$H^{-1/2}(\p\O)$. For a vector field $\A\in\mH_{\loc}(\mathbb
R^3,\curl)$, $[\A_T]=\A_T^+-\A_T^-$ belongs to
$H^{-1/2}(\p\O,\mathbb R^3)$.

We denote the spaces of tangential vector fields on $\p\O$ by
\eq\label{TanS}
\aligned
T\!C^{k+\a}(\p\O,\Bbb R^3)=&\{\w\in C^{k+\a}(\p\O,\Bbb R^3):~ \nu\cdot\w=0\;\text{on }\p\O\},\\
T\!H^{s}(\p\O,\Bbb R^3)=&\{\w\in H^s(\p\O,\Bbb R^3):~ \nu\cdot\w=0\;\text{on }\p\O\;\;\text{in the sense of trace}\}.
\endaligned
\eeq
$T\!H^{-s}(\p\O,\Bbb R^3)$ denotes the dual space of $T\!H^s(\p\O,\Bbb R^3)$.\footnote{When $\p\O$ is Lipschitz, see \cite{BCS} for the definition of $T\!H^s(\p\O,\Bbb R^3)$ and $T\!S^{-s}(\p\O,\Bbb R^3)$.}
If $F(\O)$ denote a space of scalar functions, then we set
$$\dot F(\O)=\{\phi\in F(\O): \int_\O \phi(x) dx=0\}.
$$
We also use the following notation: If $X(\O)$ denotes a space of vector fields, then we set
$$\aligned
X(\O,\div0)=&\{\u\in X(\O):~ \div\u=0\;\text{in }\O\},\\
X(\O,\curl0)=&\{\u\in X(\O):~ \curl\u=\0\;\text{in }\O\},\\
X_{t0}(\O)=&\{\u\in X(\O):~ \u_T=\0\;\text{on }\p\O\},\\
X_{n0}(\O)=&\{\u\in X(\O):~ \nu\cdot\u=0\;\text{on }\p\O\}.
\endaligned
$$

We need the following div-curl-gradient inequalities.
\begin{Lem}\label{Lem2.1} Assume $\O$ is a bounded domain in $\Bbb R^3$ with a $C^2$ boundary, $k$ is a non-negative integer, and $1<p<\infty$.
\begin{itemize}
\item[(i)] If $\O$ is simply-connected, then
\eq\label{dcg1}
\|\u\|_{W^{k+1,p}(\O)}\leq C(\O,k,p)\{\|\div\u\|_{W^{k,p}(\O)}+\|\curl\u\|_{W^{k,p}(\O)}+\|\nu\cdot\u\|_{W^{k+1-1/p,p}(\p\O)}\}.
\eeq
\item[(ii)] If $\O$ has no holes, then
\eq\label{dcg2}
\|\u\|_{W^{k+1,p}(\O)}\leq C(\O,k,p)\{\|\div\u\|_{W^{k,p}(\O)}+\|\curl\u\|_{W^{k,p}(\O)}+\|\nu\times\u\|_{W^{k+1-1/p,p}(\p\O)}\}.
\eeq
\end{itemize}
\end{Lem}

\begin{Lem}\label{Lem2.2} Assume $\O$ is a bounded domain in $\Bbb R^3$ with a $C^{k+2+\a}$ boundary, $k$ is a non-negative integer, and $0<\a<1$.
\begin{itemize}
\item[(ii)]  If $\O$ is simply-connected, then
\eq\label{dcg3}
\|\u\|_{C^{k+1+\a}(\overline{\O})}\leq C(\O,k,\a)\{\|\div\u\|_{C^{k+\a}(\overline{\O})}+\|\curl\u\|_{C^{k+\a}(\overline{\O})}+\|\nu\cdot\u\|_{C^{k+1+\a}(\p\O)}\}.
\eeq
\item[(ii)] If $\O$ has no holes, then
\eq\label{dcg4}
\|\u\|_{C^{k+1+\a}(\overline{\O})}\leq C(\O,k,\a)\{\|\div\u\|_{C^{k+\a}(\overline{\O})}+\|\curl\u\|_{C^{k+\a}(\overline{\O})}+\|\nu\times\u\|_{C^{k+1+\a}(\p\O)}\}.
\eeq
\end{itemize}
\end{Lem}
These inequalities and more general versions can be found in literature. For instance, \eqref{dcg1} and \eqref{dcg2} with $p=2$
can be found in Theorem 3 on p.209, and Proposition $6'$ on p.237 in \cite{DaL3}, also se \cite{Ce, GR, Sc, MMT}. \eqref{dcg1} and \eqref{dcg2} with $1<p<\infty$ can be found in \cite{W, AS, KY}. \eqref{dcg3} and \eqref{dcg4} can be found in \cite{BW}. For a domain with Lipschitz boundary, see for instance \cite{Ne, Sa, Co, ABD, CD} and the references therein.

For a smooth tangential vector field $\mB_T$
defined on $\p\O$, $\nu\cdot\curl\mB_T$ is well-defined and
it depends only on $\mB_T$. From \cite[Lemma 2.5]{NW} (also see \cite[Lemma 2.3]{BaP}) we have

\begin{Lem}\label{Lem-extension} Let $\O$ be a bounded and simply-connected domain in $\mathbb R^3$
with a $C^{k+1}$ boundary, $k\geq 1$, $0\leq \a<1$, and
$$\mB_T\in T\!C^{k+\a}(\p\O,\mathbb R^3),\q \nu\cdot\curl\mB_T=\0\q\text{\rm on }\p\O.
$$
Then $\mB_T$ can be extended to a curl-free $C^{k+\a}$ vector field on $\O$, namely, there exists $\tilde{\mB}\in
C^{k+\a}(\overline{\O},\curl0)$ such that $\tilde{\mB}_T=\mB_T$ on $\p\O$.
Furthermore, there
exists a harmonic function $\phi\in C^{k+1+\a}(\overline{\O})$ such that
$(\nabla\phi)_T=\mB_T$ on $\p\O$.
\end{Lem}

\subsection{Estimates for linear Maxwell's system}\

By the analysis in \cite[Lemmas B.2, B.3]{P6}, but with more careful control of computations, we can get the following estimates of solutions of a linear Maxwell's system.

\begin{Lem}\label{Lem-linear-curl}
 Let $\O$ be a bounded domain in $\mathbb R^3$ with a $C^3$ boundary, $a\in C^1(\overline{\O})$ and $a(x)>0$ on $\overline{\O}$. Define an operator $\mL$ by
$$
\mL\u=\curl\bigl(a(x)\curl\u\bigr)+\u.
$$
We have the following conclusions:
\begin{itemize}
\item[(i)] $\mL:~ H^2_{t0}(\O,\div0)\to \mH(\O,\div0)$ is an
isomorphism with
\begin{equation}\label{L-est1}
\aligned
&\|\u\|_{H^1(\O)}\leq {C(\O)\over\sqrt{m}}\|\mL\u\|_{L^2(\O)},\\
& \|\u\|_{H^2(\O)}\leq
C(\O)m^{-3/2}\|a\|_{C^1(\overline{\O})}\|\mL\u\|_{L^2(\O)},
\endaligned
\end{equation}
where $m=\min\{1,\min_{x\in\overline{\O}}|a(x)|\}$.
\item[(ii)] If furthermore $\O$ is simply-connected, without holes and with a $C^{3+\a}$ boundary, and $a\in
C^{1+\a}(\overline{\O})$, $0<\a<1$,
then $\mL:~C^{2+\a}_{t0}(\overline{\O},\div0)\to C^\a(\overline{\O},\div0)$  is an
isomorphism with
$$
\aligned
\|\u\|_{C^{2+\a}(\overline{\O})} \leq&
C(\O,\a)m^{-3/2}M(a)\|a^{-1}\|_{C^{1+\a}(\overline{\O})}\|\mL\u\|_{C^\a(\overline{\O})},
\endaligned
$$
where
$$M(a)=\Big(1+\|a\|_{C^1(\overline{\O})}\Big)^2\Big(1+\|\nabla(\log a)\|_{C^\a(\overline{\O})}^{2+\a/2}
+\|a^{-1}\|_{C^{1+\a}(\overline{\O})}^{\a/2}\Big).
$$
If $0<\a<1/2$, then
$$
\|\u\|_{C^{1+\a}(\overline{\O})}\leq
C(\O,\a)m^{-3/2}\|a^{-1}\|_{C^\a(\overline{\O})}\Big(1+\|a\|_{C^1(\overline{\O})}^2\Big)\|\mL\u\|_{H^1(\O)}.
$$
\end{itemize}
\end{Lem}

\subsection{Assumptions}\label{Subsection2.3}\

Some of the following assumptions with $0<\a<1$ will be needed in various places in this paper.

\begin{itemize}
\item[$(O)$] $\O$ is a bounded and simply-connected domain in $\mathbb R^3$
with a $C^{r+\a}$ boundary and without holes, $r\geq 3$.

\item[$(H)$] $\mH^e\in C^{2+\a}_{\loc}(\O^c,\curl0,\div0)
\cap C^{1+\a}_{\loc}(\overline{\O^c},\mathbb R^3)$ such that
\begin{equation} \label{cond-vanish}
\int_{\p\O}\nu\cdot\mH^e dS=0.
\end{equation}

\item[$(F)$] There exists $\mF^e\in C^{2+\a}_{\loc}(\overline{\O^c},\div0)$ such that $\curl\mF^e=\mH^e$ in $\O^c$.
\end{itemize}

Under condition $(H)$ it holds that $\nu\cdot\curl\mH^e_T=\nu\cdot\curl\mH^e=0$ on $\p\O$.
Note that when we consider solutions of \eqref{eqH} with continuous tangential component we only need the following condition
\begin{itemize}
\item[$(H_0)$] $\mH^e\in C^{2+\a}_{\loc}(\O^c,\curl0,\div0)
\cap C^{1+\a}_{\loc}(\overline{\O^c},\mathbb R^3)$ with $\mH^e_T\in
C^{2+\a}(\p\O,\Bbb R^3)$.
\end{itemize}
Condition $(H_0)$ is weaker than $(H)$ where the integral condition \eqref{cond-vanish} is dropped. Later on when we look for  solutions of \eqref{eqH} and \eqref{eqfH} with continuous normal component $\nu\cdot\H$ we need \eqref{cond-vanish}, see for instance \eqref{cond-6.9} and the discussions in Section \ref{Section7}.

\v0.1in

\section{The Meissner System: Basic Properties of Solutions}\label{Section3}

\subsection{Definition and basic properties of weak solutions}\

The following sets of vector fields will be needed in order to define weak
solutions of \eqref{eqfA}. Given a vector field $\mH^e$ and
$\lam>0$ we define
\begin{equation}\label{sp-AB}
\aligned
&\mA(\O,\mathbb R^3)=\{\A\in\mH_{\loc}(\mathbb
R^3,\text{\rm curl}):~ \|\A\|_{L^\infty(\O)}<\infty\},\\
&\mA(\O,\mathbb R^3,\lam^{-1}\mH^e)=\{\A\in \mathcal
A(\O,\mathbb
R^3):~\curl\A-\lam^{-1}\mH^e\in L^2(\mathbb R^3,\mathbb R^3)\},\\
&\mB(\O,\mathbb R^3)=\{\B\in \mA(\O,\Bbb R^3):
~\curl\B\in L^2(\mathbb R^3,\mathbb R^3),\;\; \B_T^+\in T\!H^{1/2}(\p\O,\Bbb R^3)\}.
\endaligned
\end{equation}
Let
$\mE$ be the functional defined in section \ref{Section1}. If
$(f,\A)$ is a critical point of $\mE$ on
$H^1(\O)\times\mA(\O,\Bbb R^3,\lam^{-1}\mH^e)$,
then for any $(g,\B)\in H^1(\O)\times\mB(\O,\mathbb R^3)$ we have
\begin{equation}\label{wkeqfA}
\int_\O\Big\{{\lam^2\over\k^2}\nabla f\cdot\nabla
g-(1-|f|^2-|\A|^2)fg+f^2\A\cdot\B\Big\}dx+\int_{\mathbb
R^3}\lam(\lam\,\curl\A-\mH^e)\cdot\curl\B\, dx=0.
\end{equation}
If $\mH^e$ satisfies \eqref{1.2} and if $\B$ has bounded support, then  \eqref{wkeqfA} is reduced to
\begin{equation}\label{wkeqfA-2}
\int_\O\Big\{{\lam^2\over\k^2}\nabla f\cdot\nabla
g-(1-|f|^2-|\A|^2)fg+f^2\A\cdot\B\Big\}dx+\lam^2\int_{\mathbb
R^3}\curl\A\cdot\curl\B\, dx=0.
\end{equation}
Under condition \eqref{1.2}, from either \eqref{wkeqfA} or \eqref{wkeqfA-2}
we find that the Euler-Lagrange equations of $(f,\A)$ is exactly
\eqref{eqfA}.
Note that the second integral in \eqref{wkeqfA} makes sense
if  $(f,\A)\in H^1(\O)\times\mA(\O,\mathbb R^3,\lam^{-1}\mH^e)$
and $(g,\B)\in H^1(\O)\times \mB(\O,\Bbb R^3)$, however the second integral in \eqref{wkeqfA-2}
may not make sense for such $(f,\A)$ and $(g,\B)$. This observation leads to the
following different definitions of weak solutions to \eqref{eqfA} and to \eqref{eqfA}-\eqref{1.4}.

\begin{Def}\label{Def3.1}
{\rm (i)} $(f,\A)$ is called a weak
solution of \eqref{eqfA} if $(f,\A)\in H^1(\O)\times\mA(\O,\mathbb R^3)$
such that \eqref{wkeqfA-2} holds for all $(g,\B)\in H^1(\O)\times\mB(\O,\mathbb R^3)$ with $\B$ having bounded support.

{\rm (ii)} A weak solution $(f,\A)$ of \eqref{eqfA}  is called a weak
Meissner solution  if $f>0$ in $\overline{\O}$, and $\nu\cdot\A^-=0$ holds in
the sense of trace in $H^{-1/2}(\p\O)$.

{\rm (iii)} Assume $\mH^e$ satisfies \eqref{1.2}.
$(f,\A)$ is called a weak solution of problem \eqref{eqfA}-\eqref{1.4}
if $(f,\A)\in H^1(\O)\times\mA(\O,\mathbb R^3,\lam^{-1}\mH^e)$ such that \eqref{wkeqfA} holds for all
$(g,\B)\in H^1(\O)\times\mB(\O,\mathbb R^3)$.
\end{Def}

In Definition \ref{Def3.1},
the requirement $[\A_T]=\0$ on $\p\O$ is included in the condition $\A\in\mA(\O,\mathbb R^3)$
(see the trace theorem \cite[P.204, Theorem 2]{DaL3}), and it holds in $H^{-1/2}(\p\O,\Bbb R^3)$.
The requirement $[(\curl\A)_T]=\0$ on $\p\O$ is included in \eqref{wkeqfA} or \eqref{wkeqfA-2}, see Lemma \ref{Lem3.5} below, and the equality holds also in  $H^{-1/2}(\p\O,\Bbb R^3)$.
The requirement \eqref{1.4} is replaced by requiring
$(f,\A)\in H^1(\O)\times\mA(\O,\mathbb R^3,\lam^{-1}\mH^e)$ and requiring \eqref{wkeqfA} to hold for all $(g,\B)\in H^1(\O)\times\mB(\O,\Bbb R^3)$. The requirement
$\B^+_T\in H^{1/2}(\p\O,\Bbb R^3)$ for the test fields $\B\in\mB(\O,\Bbb R^3)$ is not needed in \eqref{wkeqfA} and \eqref{wkeqfA-2},
however, it is needed in Lemma \ref{Lem3.5} to derive \eqref{wkeqA-exterior}.

If $(f,\A)$ is a
weak solution of problem \eqref{eqfA}-\eqref{1.4}, then $f$ is a weak
solution of a Neumann problem in $\O$
\begin{equation}\label{3.4}
-{\lam^2\over\k^2}\Delta f=(1-f^2-|\A|^2)f\q\text{\rm in }\O,\q {\p
f\over\p\nu}=0\q\text{\rm on }\p\O,
\end{equation}
and $\A$ satisfies
\begin{equation}\label{3.5}
\int_{\mathbb R^3}\{\lam(\lam\,\curl\A-\mathcal
H^e)\cdot\curl\B+\chi_\O f^2\A\cdot\B\}dx=0,\q\forall \B\in \mB(\O,\Bbb R^3),
\end{equation}
where $\chi_\O$ is the characteristic function of $\O$.
Under condition \eqref{1.2} if $\B$ has bounded support then \eqref{3.5}
can be written as
\eq\label{3.6}
\int_{\mathbb R^3}\{\lam^2\curl\A\cdot\curl\B+\chi_\O f^2\A\cdot\B\}dx=0.
\eeq

\begin{Def}\label{Def3.2} {\rm (i)} $\A$ is called a weak solution of \eqref{eqA} if
$\A\in \mA(\O,\Bbb R^3)$ such that for all $\B\in \mB(\O,\Bbb R^3)$
having bounded support it holds that
\begin{equation}\label{3.7}
\int_\O(1-|\A|^2)\A\cdot\B\, dx+\lam^2\int_{\mathbb
R^3}\curl\A\cdot\curl\B\, dx=0.
\end{equation}
A weak solution of \eqref{eqA} is called a {\it Meissner solution} if $\nu\cdot\A^-=0$ on $\p\O$.

{\rm (ii)} Assume $\mH^e$ satisfies \eqref{1.2}. $\A$ is called a weak solution
of \eqref{eqA}-\eqref{1.4} if $\A\in\mA(\O,\mathbb R^3,\lam^{-1}\mathcal
H^e)$ such that
\begin{equation}\label{3.8}
\int_\O(1-|\A|^2)\A\cdot\B\, dx+\int_{\mathbb
R^3}\lam(\lam\,\curl\A-\mH^e)\cdot\curl\B\, dx=0,\q\forall \B\in\mB(\O,\mathbb R^3).
\end{equation}
\end{Def}

To define the weak solutions of \eqref{eqH} and \eqref{eqfH} we need the following sets:
\begin{equation}\label{sp-HHU}
\aligned
&\mH(\O,\mathbb R^3)=\{\H\in
\mH_{\loc}(\mathbb R^3,\curl):~\curl(\H|_{\O})\in L^\infty(\O,\Bbb R^3),\;\;
\div(\H|_{\O})=0,\\
&\qqq\qq \curl(\H|_{\O^c})=\0,\;\; \div(\H|_{\O^c})=0\},\\
&\mH(\O,\mathbb R^3,\mH^e)=\{\H\in
\mH(\O,\mathbb R^3):~\H-\mH^e\in L^2(\mathbb R^3,\mathbb R^3)\},\\
&\mU(\O,\mathbb R^3)
=\{\B\in \mH(\O,\mathbb R^3)\cap L^2(\Bbb R^3,\Bbb R^3):~\B_T^-\in T\!H^{1/2}(\p\O,\Bbb R^3)\}.
\endaligned
\end{equation}

\begin{Def}\label{Def3.3} {\rm (i)}
$\H$ is called a weak solution of \eqref{eqH} if $\H\in
\mH(\O,\mathbb R^3)$ such that
$$
\aligned
&\int_\O\{\lam^2F(\lam^2|\curl\H|^2)\curl\H\cdot\curl\B +\H\cdot\B\} dx\\
&\q +\int_{\p\O}\lam^2F(\lam^2|\curl\H|^2)((\curl\H)_T^-\times\B_T^-)\cdot\nu
dS=0,\q\forall \B\in\mU(\O,\mathbb R^3).
\endaligned
$$

{\rm (ii)} $\H$ is called a weak solution of problem \eqref{eqH}-\eqref{cond-1.12}
if $\H$ is a weak solution of \eqref{eqH} and $\H\in
\mH(\O,\mathbb R^3,\mH^e)$.
\end{Def}

\begin{Def}\label{Def3.4}
{\rm (i)} $(f,\H)$ is called a weak
solution of \eqref{eqfH} if $(f,\H)\in H^1(\O)\times\mH(\O,\mathbb R^3)$
such that $f>0$ on $\O$ and
$$
\aligned
&\int_\O\Big\{{\lam^2\over\k^2}\nabla f\cdot\nabla
g-(1-|f|^2-\lam^2f^{-4}|\curl\H|^2)fg+\lam^2f^{-2}\curl\H\cdot\curl\B+ \H\cdot\B\Big\}dx\\
+&\int_{\p\O}\lam^2f^{-2}((\curl\H)_T^-\times\B^-_T)\cdot\nu
dS=0,\qq\forall (g,\B)\in H^1(\O)\times\mU(\O,\mathbb R^3).
\endaligned
$$

{\rm (ii)}  $(f,\H)$ is called a weak solution of problem \eqref{eqfH}-\eqref{cond-1.12}
if $(f,\H)$ is a weak solution of \eqref{eqfH} and $\H\in
\mH(\O,\mathbb R^3,\mH^e)$.
\end{Def}

\begin{Lem}\label{Lem3.5}
Let $\O$ be a bounded domain in $\Bbb R^3$ with a Lipschitz boundary, and let
$(f,\A)$ be a weak solution of \eqref{eqfA}.
\begin{itemize}
\item[(i)] We have
\eq\label{curl2A0}
\curl^2\A(x)=\0\q\text{\rm a.e. in }\O^c,
\eeq
and the outer tangential trace $(\curl\A)_T^+$ exists in $T\!H^{-1/2}(\p\O,\Bbb R^3)$.
For any $\B\in\mB(\O,\Bbb R^3)$ with bounded support it holds that
\eq\label{wkeqA-exterior}
\int_{\O^c}\curl\A\cdot\curl\B\, dx=\int_{\p\O}((\curl\A)_T^+\times\B_T)\cdot\nu dS,
\eeq
where $\nu$ is the unit normal vector field of $\p\O$ pointing into $\O^c$.
\item[(ii)] $\curl^2\A(x)$ exists for a.e. $x\in\O$,
$\curl^2\A\in L^2(\O,\Bbb R^3)$, and the second equation in \eqref{eqfA} holds for a.e. $x\in\O$. The equality
\eq\label{t-cont-curlA}
[(\curl\A)_T]=\0\q\text{\rm on }\p\O
\eeq
holds in $T\!H^{-1/2}(\p\O,\Bbb R^3)$.
\item[(iii)] If in addition $\p\O$ is of class $C^2$, then
$\curl\A\in H^1_{\loc}(\Bbb R^3,\Bbb R^3)$ and
$$[\curl\A]=\0\q\text{\rm on }\p\O
$$
holds in $H^{1/2}(\p\O,\Bbb R^3)$.
\end{itemize}
\end{Lem}

The proof will be given in Appendix \ref{AppendixA}.

\subsection{Regularity of weak Meissner solutions}\

In this subsection we examine regularity in $\O$ of the weak solutions of \eqref{eqfA}. We shall use $C(\O)$ to denote a generic constant which depends only on $\O$ but may vary from line to line.

\begin{Thm}\label{Thm-reg-fA} Assume that $\O$ is a bounded domain in $\mathbb R^3$ with a $C^2$ boundary, and
$\mH^e\in C^1_{\loc}(\mathbb R^3,\mathbb R^3)$
satisfying \eqref{1.2}. Let $(f,\A)$ be a weak Meissner solution of
\eqref{eqfA}. Then there exist constants $c$ and $M$ such that
$$
0<c=\min_{\overline{\O}}f\leq \max_{\overline{\O}}f\leq 1,\q M=\|\A\|_{L^\infty(\O)}<\infty.
$$
Let $\H=\lam\,\curl\A$, $0<\a<1$, $0<\b<1/2$ and $\k\geq \max\{1,\lam\}$. Then we have:
\begin{itemize}
\item[(a)] $f\in H^2(\O)\cap C^{1+\a}(\overline{\O})$, $\A\in H^1_{\loc}(\mathbb R^3,\mathbb R^3)$,
$\H\in H^1_{\loc}(\mathbb R^3,\mathbb R^3)$, $\curl\H\in
H^1_{\loc}(\mathbb R^3,\mathbb R^3)\cap L^\infty(\Bbb R^3,\Bbb R^3)$,
$$[\A_T]=\0,\q
[\H_T]=\0\q\text{\rm on }\p\O,
$$
and
\eq\label{est-3.13}
\aligned
&\|\H\|_{L^2(\O)}\leq C(\O)(1+\lam)M,\\
&\|\H\|_{H^1(\O)}\leq C(\O)(\lam+\lam^{-1})M,\\
\endaligned
\eeq
\eq\label{est-3.14}
\|\A\|_{H^1(\O)}\leq C(\O)d_1 M,
\eeq
where $d_1=c^{-1}\lam^{-1}\k$.

\item[(b)] If $\p\O$ is of class $C^3$, then
$f\in H^3(\O)\cap C^{2+\b}(\overline{\O})$, $\A\in H^2(\O,\Bbb R^3)$, $\curl\H\in H^2(\O,\Bbb R^3)$, and
\begin{equation}\label{est-3.15}
\aligned
&\|\A\|_{H^2(\O)}\leq
C(\O,M)(c^{-1}d_1)^2M.
\endaligned
\end{equation}
If furthermore $\H_T\in T\!H^{3/2}(\p\O,\Bbb R^3)$, then $\H\in H^2(\O,\Bbb R^3)$, and
\eq\label{H-H2}
\|\H\|_{H^2(\O)}\leq
C(\O)\lam^{-1}d_1M+C(\O)\|\H_T\|_{H^{3/2}(\p\O)}.
\eeq

\item[(c)] If $\p\O$ is of class $C^{3+\b}$, then
$f\in C^{3+\b}(\overline{\O})$, $\A\in C^{1+\b}(\overline{\O},\mathbb R^3)$,
$\curl\H\in C^{1+\b}(\overline{\O},\mathbb R^3)$, and
\eq\label{est-3.17}
\sum_{n=0}^3\Big({\lam\over\k}\Big)^n\|D^n f\|_{C^0(\overline{\O})}+\Big({\lam\over\k}\Big)^{3+\b}[D^3f]_\b\leq
C(\O,M,\b).
\eeq
\end{itemize}
\end{Thm}

The proof of Theorem \ref{Thm-reg-fA} is similar to that of Proposition \ref{Prop4.3}. We shall give a complete proof of Proposition \ref{Prop4.3} in Appendix \ref{AppendixB} and give a brief proof of Theorem \ref{Thm-reg-fA} in Appendix~\ref{AppendixC}.

From Theorem \ref{Thm-reg-fA} we see that a weak Meissner solution
$(f,\A)$ is smooth on $\overline{\O}$. However, the regularity of the
magnetic potential part $\A$ in $\O^c$ can be rather poor. In general we only know $\A$ is in
$H^1_{\loc}(\O^c,\mathbb R^3)$. To see this, take a
function $\phi\in H^2_0(\O^c)$ that does not belong to
$H^3_{\loc}(\O)$. Then $(f,\A+\nabla\phi)$ is also a weak Meissner
solution of \eqref{eqfA}-\eqref{1.4} but $\A+\nabla\phi$ does not belong to $H^2_{\loc}(\O^c,\mathbb R^3)$.
The same reasoning applies to the weak solutions of \eqref{eqA}, and
we see that in general a weak solution $\A$ of \eqref{eqA} belongs
to $H^1_{\loc}(\O^c,\mathbb R^3)$ but does not belong to
$H^2_{\loc}(\O^c,\mathbb R^3)$. Moreover, the solutions of
\eqref{eqfA} or of \eqref{eqA} may have discontinuity in normal
components $\nu\cdot\A$ or $\nu\cdot\curl\A$ at $\p\O$. It was also proved in \cite[Lemma 3.1]{P3} that if $\A$ is a non-zero solution of \eqref{eqA} and satisfies
\eqref{cond-1.8}, then $\curl^2\A$ is not continuous on the boundary
$\p\O$.

To describe the precise regularity of such solutions of \eqref{eqfA} we need  some spaces of functions, which were introduced in \cite[Subsection 3.1]{P3} to study \eqref{eqA}. If $\p\O$ is of class $C^k$ and if $\A\in C^k(\overline{\O},\mathbb R^3)$, then the tangential
derivatives of $\A$ at $\p\O$ of order up to $k$ are well-defined,
and we may write them as $(D_t^l\A)^-$, $0\leq |l|\leq k$, and we
may call them the {\it inner trace of the tangential derivatives} of
$\A$. Similarly, if $\A\in C^k_{\loc}(\overline{\O^c},\mathbb R^3)$, then
we may write the tangential derivatives at $\p\O$ as $(D_t^l\A)^+$
and call them the {\it outer trace of the tangential derivatives} of
$\A$.

\begin{Def}\label{Def3.7} Let $k$ be a non-negative integer and $0<\b<1$.

If there exists a neighborhood $U$ of $\p\O$ such that
$\A\in C^k(\overline{\O}\cap \overline{U},\mathbb R^3)\cap C^k(\overline{\O^c}\cap \overline{U},\mathbb R^3)$
and if
$$(D_t^l\A)^-=(D_t^l\A)^+\q\text{holds on $\p\O$ for all $|l|\leq k$},
$$
then we say that $\A\in \C^k(\p\O,\mathbb R^3)$.

If $\A\in C^{k+\b}(\overline{\O}\cap \overline{U})\cap C^{k+\b}(\overline{\O^c}\cap \overline{U})$
and if
$$(D_t^l\A)^-=(D_t^l\A)^+\in C^\b(\p\O,\mathbb
R^3)\q\text{for all $|l|=k$},
$$
then we say that $\A\in
\C^{k+\b}(\p\O,\mathbb R^3)$.
\end{Def}

In the following we shall always assume that the unit outer normal
vector field $\nu$ has been extended onto a small neighborhood $U$
of $\p\O$. Then $\nu\cdot\A$ and $\A_T$ can be defined on $U$, and
the statement ``$\A_T\in \C^{m+\b}(\p\O,\mathbb R^3)$" is
meaningful. Then we can define, for integers $k, m\geq 0$ and real
numbers $0\leq\a,\b<1$,
\eq\label{sp-Ckalpha}
\aligned
&\C^{k+\a,m+\b}(\overline{\O},\overline{\O^c},\mathbb R^3) =\{\A\in
C^{k+\a}(\overline{\O},\mathbb R^3)\cap C^{k+\a}_{\loc}(\overline{\O^c},\mathbb
R^3):~\A\in \C^{m+\b}(\p\O,\mathbb R^3)\},\\
&\C_t^{k+\a,m+\b}(\overline{\O},\overline{\O^c},\mathbb R^3) =\{\A\in
C^{k+\a}(\overline{\O},\mathbb R^3)\cap C^{k+\a}_{\loc}(\overline{\O^c},\mathbb
R^3):~ \A_T\in \C^{m+\b}(\p\O,\mathbb R^3)\}.
 \endaligned
\eeq

In \cite[Definition 3.1]{P3} we define $\A$ to be a classical solution of \eqref{eqA} if $\A$ satisfies
\begin{equation}\label{cond-A-ct}
\A\in\C_t^{2,0}(\overline{\O},\overline{\O^c},\mathbb R^3),\q (\curl\A)_T\in
\C^0(\p\O,\mathbb R^3),
\end{equation}
and $\A$ satisfies \eqref{eqA} pointwise. We define
$\H$ to be a classical solution of \eqref{eqH} if
$\H\in \C_t^{2,0}(\overline{\O},\overline{\O^c},\mathbb R^3)$ and if $\H$
satisfies \eqref{eqH} pointwise.

\begin{Def}\label{Def3.8} $(f,\A)$ is called a classical
solution of \eqref{eqfA} if $f\in C^2_{\loc}(\O)\cap C^1(\overline{\O})$, $\A$ satisfies \eqref{cond-A-ct}, and $(f,\A)$ satisfies \eqref{eqfA} pointwise.
\end{Def}

\subsection{Locally $L^\infty$-stable solutions}\

Recall that for any solution $(f,\A)$ of \eqref{eqfA}, the field $(\psi,\mA)=(fe^{i\chi},\A+{\lam\over\k}\nabla\chi)$, as a solution of the original Ginzburg-Landau system \eqref{GL}, is not globally stable for large $\k$ (see \cite{S}). Local stability of Meissner solutions to \eqref{eqfA} has been discussed in \cite[Subsection 2.6]{P3}.\footnote{Chapman \cite{C2} shows in the 2-dimensional case that the solution of \eqref{eqA-Omega} is stable if it satisfies \eqref{cond-1.8}.}
Here we consider $L^\infty$-stability of $(f,\A)$ with respect to \eqref{eqfA}, but not to \eqref{GL}.

\begin{Def}
Assume $\mH^e$ satisfies \eqref{1.2}.
A weak solution of problem \eqref{eqfA}-\eqref{1.4} is said to be
locally $L^\infty$-stable if there exists $\delta>0$ such that for any
$(g,\B)\in H^1(\O)\times\mB(\O,\mathbb R^3)$ satisfying
\begin{equation}\label{g-B-small}
\|g\|_{L^\infty(\O)}+\|\B\|_{L^\infty(\O)}<\delta
\end{equation}
it holds that
$\mE[f+g,\A+\B]\geq \mE[f,\A].
$
\end{Def}

Similarly we can define local stability of solutions of \eqref{eqA}-\eqref{1.4}.\footnote{For the definition of weak solutions of \eqref{eqA-Omega} one may see \cite[Definition 3.1]{BaP}.}
\eqref{g-B-small} does not put any restriction on $\B$ in
$\O^c$. This is related to the invariance property
$$\mE[f,\A+\nabla\phi]=\mE[f,\A]\q\text{if $\nabla\phi=\0$ in $\O$}.
$$

\begin{Lem}\label{Lem-loc-stable} Assume $\mH^e$ satisfies \eqref{1.2}.

{\rm (i)} A weak Meissner solution  $(f,\A)$ of \eqref{eqfA}-\eqref{1.4} is locally $L^\infty$-stable provided
\begin{equation}\label{f-A-stable}
\inf_{x\in\overline{\O}}\;\{f^2(x)-|\A(x)|^2\}>{1\over 3}.
\end{equation}

{\rm (ii)} A weak solution $\A$ of \eqref{eqA}-\eqref{1.4} (resp. a solution of \eqref{eqA-Omega}) is
locally $L^\infty$-stable provided it satisfies \eqref{cond-1.8}.
\end{Lem}

Lemma \ref{Lem-loc-stable} is proved by direct computations. It can also be derived using the convexity of $\mathcal
E_\O$ in $\mK(\O)$, see Lemma \ref{Lem4.2} (ii).

\begin{Lem}\label{Lem-Uniq}
Let $\O$ be a bounded domain in $\Bbb R^3$ with a $C^2$ boundary, and $\mH^e$ satisfy $(H)$. If
problem \eqref{eqfA}-\eqref{1.4} has two weak Meissner solutions $(f_0,\A_0)$ and $(f_1,\A_1)$ which satisfy \eqref{f-A-stable}, then
$$(f_0,\A_0)=(f_1,\A_1)\q\text{\rm in }\O,\q\curl\A_0=\curl\A_1\q\text{\rm in }\O^c.
$$
\end{Lem}

The proof of Lemma \ref{Lem-Uniq} will be given in Appendix \ref{AppendixA}.

\subsection{Decomposition  of the Meissner system}\

%\subsection{Decomposition}\

We shall show that \eqref{eqfA} can be decomposed into two systems: a BVP in $\O$
\begin{equation}\label{dec-eqfA-1}
\left\{\aligned
-&{\lam^2\over\k^2}\Delta f=(1-f^2-|\A|^2)f\q&\text{\rm in }\O,\\
&\lam^2\curl^2\A+f^2\A=\bold 0\q\,&\text{\rm in }\O,\\
&{\p f\over\p\nu}=0,\q  \lam(\curl\A)^-_T=\mB_T\q&\text{\rm
on }\p\O,
\endaligned\right.
\end{equation}
and an exterior problem in $\O^c$:
\begin{equation}\label{dec-eqfA-2}
\left\{\aligned
&\curl^2\A={\bold 0}\qqq\qqq\q&\text{ \rm in }\O^c,\\
&\A_T^+=\mA_T,\q\lam(\curl\A)^+_T=\mB_T\q\;& \text{\rm
on }\p\O,
\endaligned\right.
\end{equation}
with $\mB_T$ being a suitably chosen vector field.
To define weak solutions of BVP \eqref{dec-eqfA-1} we need the following spaces:
\begin{equation}\label{sp-WWVD}
\aligned
&\mW(\O)=\{(f,\A):~f\in H^1(\O),\;\A\in\mathcal
H(\O,\curl)\cap
L^\infty(\O,\mathbb R^3)\},\\
&\mW_{t0}(\O)=\{(g,\B)\in\mW(\O):~  \B^-_T=\0\text{ \rm on
}\p\O\},\\
&\mV(\O)=\{(g,\B)\in \mW(\O):~ \B_T^-\in T\!H^{1/2}(\p\O,\Bbb R^3) \}.
\endaligned
\end{equation}

\begin{Def}
{\rm (i)} Giving $\mB_T\in T\!H^{-1/2}(\p\O,\Bbb R^3)$, we say that $(f,\A)$ is a weak
solution of \eqref{dec-eqfA-1} if $(f,\A)\in \mW(\O)$,
$$\lam\,(\curl\A)^-_T=\mB_T\q\text{\rm on }\p\O,
$$
holds in $T\!H^{-1/2}(\p\O,\mathbb R^3)$, and
\begin{equation}\label{wkdec-eqfA-1}
\aligned
&\int_\O\Big\{{\lam^2\over\k^2}\nabla f\cdot\nabla
g-(1-|f|^2-|\A|^2)fg+\lam^2\curl\A\cdot\curl\D+f^2\A\cdot\D\Big\} dx\\
+&\lam\int_{\p\O}(\mB^-_T\times\D^-_T)\cdot\nu dS=0,\qq\forall (g,\D)\in \mV(\O).
\endaligned
\end{equation}

{\rm (ii)} We say that $(f,\A)$ is a weak Meissner solution of
\eqref{dec-eqfA-1} if $(f,\A)$ is a weak solution of \eqref{dec-eqfA-1} and if in addition
$$f>0\q\text{\rm in }\overline{\O},\q
\nu\cdot\A=0\q\text{\rm on }\p\O.
$$
\end{Def}

To define weak solutions of \eqref{dec-eqfA-2} we need the following space:
\eq\label{sp-Z}
\mZ(\O^c)=\{\bold Z\in L^2_{\loc}(\O^c,\Bbb R^3):~\curl\bold Z\in L^2(\O^c,\Bbb R^3),\;\; \bold Z_T^+\in T\!H^{1/2}(\p\O,\Bbb R^3)\}.
\eeq

\begin{Def}
Giving $\mA_T, \mB_T\in T\!H^{-1/2}(\p\O,\Bbb R^3)$,
we say that $\A$ is a weak solution of \eqref{dec-eqfA-2} if $\A\in
\mH_{\loc}(\O^c,\curl)$,
$$\A^+_T=\mA_T\q\text{and}\q \lam(\curl\A)^+_T=\mB_T\q\text{\rm on } \p\O,
$$
which hold in $T\!H^{-1/2}(\p\O,\mathbb R^3)$, and for any $\bold Z\in \mZ(\O^c)$
with bounded support it holds that
$$
\int_{\O^c}\curl\A\cdot\curl\bold Z\, dx=\int_{\p\O}\lam^{-1}(\mB_T\times\bold Z_T^+)\cdot\nu dS,
$$
where $\nu$ is the unit normal to $\p\O$ pointing into $\O^c$.
\end{Def}

\begin{Lem}\label{Lem-equiv-dec-eqfA}
Let $\O$ be a bounded domain in $\mathbb R^3$ with a Lipschitz boundary and $(f,\A)\in H^1(\O)\times\mA(\O,\mathbb R^3)$ be a weak solution of \eqref{eqfA}. Set
$$
\A^i=\A\bigr|_{\overline{\O}},\q
\A^o=\A\bigr|_{\overline{\O^c}},\q \mA_T=(\A^i)_T^-,\q \mB_T=\lam(\curl\A^i)^-_T.
$$
Then we have the following conclusions:
\begin{itemize}
\item[(i)] $(f,\A^i)$
is a weak solution of \eqref{dec-eqfA-1} for the boundary datum $\mB_T$, and $\A^o$ is a weak solution of \eqref{dec-eqfA-2} for the boundary
data $\mA_T$ and $\mB_T$.

\item[(ii)] If furthermore $(f,\A)$ is a weak Meissner solution of \eqref{eqfA}, then $(f,\A^i)$ is a weak Meissner solution of \eqref{dec-eqfA-1} and
\begin{equation}\label{cond-3.27}
\nu\cdot\curl\mB_T=0\q
\text{\rm on }\p\O.
\end{equation}
\end{itemize}
\end{Lem}

\begin{Lem}\label{Lem-equiv-dec-eqfA-2} Let $(f,\A^i)$ be a weak solution of
\eqref{dec-eqfA-1} with $\mB_T\in T\!H^{-1/2}(\p\O,\Bbb R^3)$ and set $\mathcal
A_T=(\A^i)^-_T$. Then
$\mA_T\in T\!H^{-1/2}(\p\O,\mathbb R^3)$.
Let $\A^o$ be a weak solution of
\eqref{dec-eqfA-2} for the boundary data $\mA_T$ and $\mB_T$, and
define $\A$ on $\Bbb R^3$ by letting
$\A(x)=\A^i(x)$ if $x\in \O$ and $\A(x)=\A^o(x)$ if $x\in\O^c$.
Then $(f,\A)$ is a weak solution of \eqref{eqfA}.
If furthermore $(f,\A^i)$ is a weak Meissner solution of
\eqref{dec-eqfA-1}, then $(f,\A)$ is a weak Meissner solution of
\eqref{eqfA}.
\end{Lem}

The proofs of these lemmas are direct and hence omitted.
A similar discussion shows that \eqref{eqA} can be decomposed into a BVP \eqref{eqA-Omega} and an exterior problem
\eqref{dec-eqfA-2}.

\section{BVP \eqref{dec-eqfA-1}: Uniqueness, Existence and Convergence}

\subsection{Uniqueness and regularity}\

In this subsection we assume $\mB_T\in T\!H^{1/2}(\p\O,\Bbb R^3)$ satisfies \eqref{cond-3.27},
and has a curl-free extension $\tilde{\mB}\in H^1(\O,\Bbb R^3)$, see Lemma \ref{Lem-extension}.
For simplicity we write $\tilde{\mB}$ as $\mB$. Hence
\eq\label{cond-B-4.1}
\mB\in H^1(\O,\curl0),\q  \nu\cdot\curl\mB=0\q\text{on }\p\O.
\eeq
Associated with BVP \eqref{dec-eqfA-1} we can define a functional by
\eq\label{funct-E-Omega}
\aligned
\mE_\O[f,\A]&=\int_{\O}\Bigl\{{\lam^2\over\k^2}|\nabla
f|^2+G(f,\A)(x) +|\lam\,\curl\A-\mB|^2\Bigr\}dx,\\
&\text{where}\qq G(f,\A)(x)=|f(x)\A(x)|^2+{1\over
2}(1-|f(x)|^2)^2.
\endaligned
\eeq
We shall use the following notation
$$\aligned
&\big\langle G'(f,\A), (g,\B)\big\rangle={d\over
dt}\Bigr|_{t=0}G(f+tg,\A+t\B)=G_f'(f,\A)g+G_{\A}'(f,\A)\cdot\bold B,\\
&\text{where}\qq G_f'(f,\A)=2(|f|^2+|\A|^2-1)f,\q G_\A'(f,\A)=2f^2\A,\\
&\big\langle G''(f,\A), (g,\B)\big\rangle={d^2\over
dt^2}\Bigr|_{t=0}G(f+tg,\A+t\B)=2|f \B+2g\A|^2+6g^2\Big(f^2-|\A|^2-{1\over 3}\Big).
\endaligned
$$
Following the ideas in \cite{BCM} we set
\begin{equation}\label{sp-KK}
\aligned
\mK(\O)=&\{(f,\A)\in\mW(\O):~ f(x)>0,\;
f^2(x)-|\A(x)|^2-{1\over 3}>0\text{ \rm on }\overline{\O}\},\\
\overline{\mK}(\O)=&\{(f,\A)\in\mW(\O):~ f(x)>0,\;
f^2(x)-|\A(x)|^2-{1\over 3}\geq 0\text{ \rm on }\overline{\O}\},\\
\mK_\delta(\O)=&\{(f,\A)\in\mW(\O):~f(x)>0,\;
f^2(x)-|\A(x)|^2-{1\over 3}>\delta\text{ \rm on }\overline{\O}\},\\
\mK^1_\delta(\O)=&\{(f,\A)\in\mW(\O):~0<f(x)\leq 1,\;
f^2(x)-|\A(x)|^2-{1\over 3}>\delta\text{ \rm on }\overline{\O}\}.
\endaligned
\end{equation}
We shall always assume $0<\delta<1/3$.

\begin{Rem}\label{Rem4.1} Let $0<\delta<1/3$.
\begin{itemize}
\item[(a)] $\mK(\O)$, $\overline{\mK}(\O)$, $\mK_\delta(\O)$ and $\mK^1_\delta(\O)$ are strict convex sets. $\mE_\O$ is convex on $\overline{\mK}(\O)$ and strictly convex on $\mK(\O)\cap [C^0(\overline{\O})\times C^0(\overline{\O},\Bbb R^3)]$.  A critical point of $\mE_\O$ in $\mW(\O)$ is a weak solution of \eqref{dec-eqfA-1}, and  a critical point of $\mE_\O$ lying in $\mK(\O)$ is a Meissner solution of \eqref{dec-eqfA-1}.
\item[(b)] There exists $C(\delta)>0$ such that, for any $(f_0,\A_0)$,
$(f_1,\A_1)\in\mK^1_\delta(\O)$ with $(g,\B)=(f_1-f_0,\A_1-\A_0)\in\mathcal
W_{t0}(\O)$, it holds that
$$
\big\langle G'(f_1,\A_1)-G'(f_0,\A_0), (g,\B)\big\rangle
\geq C(\delta)(|g|^2+|\B|^2)\q\text{\rm on }\overline{\O}.
$$
\item[(c)] For any $(f_0,\A_0), (f_1,\A_1)\in\overline{\mK}(\O)$, let
$$f_t=(1-t)f_0+tf_1, \qq \A_t=(1-t)\A_0+t\A,\q 0\leq t\leq 1.
$$
Then
$$ f_t^2-|\A_t|^2-{1\over 3}\geq 0,\q\forall\; x\in\O,\; 0\leq t\leq 1.
$$
If furthermore $(f_0,\A_0)\neq (f_1,\A_1)$, then
\eq\label{ineq4.4} f_t^2-|\A_t|^2-{1\over 3}>0,\q\forall\; x\in\O,\; 0<t<1.
\eeq
\item[(d)]  If $(f_0,\A_0), (f_1,\A_1)\in \mK_\delta(\O)$, and define $f_t, \A_t$ as above, then
$$ f_t^2-|\A_t|^2-{1\over 3}\geq \delta,\q\forall\; x\in\O,\; 0\leq t\leq 1.
$$
\end{itemize}
\end{Rem}

To prove the second part of (c), note that if $(f_0,\A_0), (f_1,\A_1)\in\overline{\mK}(\O)$ and  $(f_0,\A_0)\neq (f_1,\A_1)$, then $(f_t,\A_t)\in \mK(\O)$ for any $0<t<1$.

\begin{Lem}\label{Lem4.2} Assume $\O$ is a bounded domain in $\Bbb R^3$ with a $C^2$ boundary.
\begin{itemize}
\item[(i)] BVP \eqref{dec-eqfA-1} has at most one weak solution
in $\overline{\mK}(\O)$.
\item[(ii)] If $\mB_T\in T\!H^{1/2}(\p\O,\Bbb R^3)$ satisfies \eqref{cond-3.27}, then a solution of \eqref{dec-eqfA-1} in $\mK(\O)$ is a
minimizer of $\mE_\O$ in $\mK(\O)$, and hence it is a locally $L^\infty$-stable Meissner solution.
\end{itemize}
\end{Lem}

\begin{proof}
To prove (i), let $(f_0,\A_0)$ and
$(f_1,\A_1)\in\overline{\mK}(\O)$ be two solutions of \eqref{dec-eqfA-1}.
From \eqref{wkdec-eqfA-1}, for any $(g,\B)\in \mW_{t0}(\O)$, we have
\begin{equation}\label{eq4.5}
\aligned
\int_\O\Bigl\{&{\lam^2\over\k^2}\nabla (f_1-f_0)\cdot\nabla g+{1\over
2}\big\langle
[G'(f_1,\A_1)-G'(f_0,\A_0)],(g,\B)\big\rangle\\
&+\lam^2\curl(\A_1-\A_0)\cdot\curl\B\Bigr\}dx=0.
\endaligned
\end{equation}
Let
$$g=f_1-f_0,\q f_t=f_0+tg,\q \B=\A_1-\A_0,\q \A_t=\A_0+t\B,\q 0\leq t\leq 1.
$$
From Lemma \ref{LemB.1}, $\A_j\in H^1(\O,\Bbb R^3)$, so
$(g,\B)\in\mW_{t0}(\O)$.
Plugging $(g,\B)$ into \eqref{eq4.5} we find
\eq\label{eq4.6}
\aligned
&\int_\O\Big\{{\lam^2\over\k^2}|\nabla g|^2+\lam^2|\curl\B|^2\Big\}dx
+\int_\O\int_0^1\{|f_t\B+2g\A_t|^2+(3f_t^2-3|\A_t|^2-1)|g|^2\}dtdx=0.
\endaligned
\eeq
From this and Remark \ref{Rem4.1}~(c)  we see that each term in the left of \eqref{eq4.6} is non-negative, hence $\nabla g=0$ so $g=c$ is a constant, and $\curl\B\equiv \0$.
If $c\neq 0$, then from the second part of Remark \ref{Rem4.1}~(c) we see that the strict inequality \eqref{ineq4.4} holds for all $0<t<1$. Since the weak
Meissner solutions of \eqref{dec-eqfA-1} are continuous in $\overline{\O}$, so
$f_t$ and $\A_t$ are continuous on $\overline{\O}\times [0,1]$. Therefore the integral of $(3f_t^2-3|\A_t|^2-1)|g|^2$ is positive, thus the left side of \eqref{eq4.6} is positive.
This contradiction shows that we must have $c=0$, i.e. $f_1\equiv
f_0$. From this and since $\curl\A_1=\curl\A_0$, using the second equation of \eqref{dec-eqfA-1} we get
$\A_1=\A_0.$

(ii) follows directly from (b).
\end{proof}

\begin{Prop}\label{Prop4.3}
Assume that $\O$ is a bounded domain in $\mathbb R^3$ with a $C^4$ boundary, and
$\mB_T\in T\!H^{3/2}(\p\O,\Bbb R^3)$ satisfies \eqref{cond-3.27}.  Let $(f,\A)$ be a weak Meissner solution of \eqref{dec-eqfA-1} and let
$\H=\lam\,\curl\A$.
\begin{itemize}
\item[(i)] The conclusions  of Theorem \ref{Thm-reg-fA}
hold. In particular, for $0<\b<1/2$,
$$
\aligned
&f\in C^{3+\b}(\overline{\O}),\q
\A\in H^2(\O,\Bbb R^3)\cap C^{1+\b}(\overline{\O},\Bbb R^3),\\
&\H\in H^2(\O,\Bbb R^3),\q \curl\H\in H^2(\O,\Bbb R^3)\cap C^{1+\b}(\overline{\O},\Bbb R^3),
\endaligned
$$
and estimates \eqref{est-3.13}, \eqref{est-3.14}, \eqref{est-3.15},  \eqref{est-3.17}, and \eqref{H-H2} with $\H_T$ replaced by $\mB_T$  hold.
\item[(ii)]
If furthermore $\mB_T\in T\!C^{2+\a}(\p\O,\Bbb R^3)$ with $0<\a<1$, and if $(f,\A)\in \mK(\O)$, then there
exists a positive constant $C=C(\O,\k,\lam,\a,\|\mB_T\|_{C^{2+\a}(\p\O)})$ such that
$$
\aligned
\|f\|_{C^{2+\a}(\overline{\O})}+\|\H\|_{C^{2+\a}(\overline{\O})}+
\|\A\|_{C^{2+a}(\overline{\O})}\leq C.
\endaligned
$$
\end{itemize}
\end{Prop}

Proposition \ref{Prop4.3} is a direct consequence of Lemma \ref{LemB.2} in Appendix
\ref{AppendixB}.

\subsection{Existence}\

It is difficult to get existence for \eqref{dec-eqfA-1} by minimizing the functional $\mE_\O$ as it does not provide control on $\div\A$. Instead, we consider an equivalent system
\begin{equation}\label{eq4.7}
\left\{\aligned
-&{\lam^2\over\k^2}\Delta f=(1-f^2-\lam^2f^{-4}|\curl\H|^2)f\q&\text{\rm in }\O,\\
&\lam^2\curl(f^{-2}\curl\H)+\H=\bold 0\qq\;\;\,&\text{\rm in }\O,\\
&{\p f\over\p\nu}=0,\q \H_T=\mB_T\qqq&\text{\rm on }\p\O.
\endaligned\right.
\end{equation}
Define
\begin{equation}\label{sp-Bkalpa-1}
\mB^{k+\a}(\p\O)=\{\B\in T\!C^{k+\a}(\p\O,\mathbb
R^3):~\nu\cdot\curl\B=0\;\;\text{\rm
on }\p\O\}.
\end{equation}
If $\O$ is simply-connected, every
vector field in $\mB^{k+\a}(\p\O)$ can be
extended to a harmonic gradient on $\overline{\O}$ (see Lemma \ref{Lem-extension}), and we may identify
$\mB^{k+\a}(\p\O)$ with the following space
\eq\label{sp-Bkalpa-2}
\aligned
\mB^{k+\a}(\overline{\O})&=\{\nabla\phi:~ \phi\in
C^{k+1+\a}(\overline{\O}),\; \Delta\phi=0\text{ \rm in }\O,\;\; {\p\phi\over\p\nu}=0\;\;\text{\rm on }\p\O\}.
\endaligned
\eeq

\begin{Lem}[Equivalent system]
Assume $\O$ is a bounded domain in $\mathbb
R^3$ with a $C^{3+\a}$ boundary, $0<\a<1$, and $\mB_T\in T\!C^2(\p\O,\mathbb
R^3)$.
\begin{itemize}
\item[(i)] Let $(f,\A)\in C^2(\overline{\O})\times C^3(\overline{\O}, \mathbb R^3)$
be a solution of \eqref{dec-eqfA-1} with $f>\overline{\O}$, and set
$\H=\lam\,\curl\A$. Then $(f,\H)\in C^2(\overline{\O})\times C^2(\overline{\O},
\mathbb R^3)$ and it is a solution of \eqref{eq4.7}.

\item[(ii)] Assume in addition $\O$ is a simply-connected domain and
has no holes. Let $(f,\H)\in C^{2+\a}(\overline{\O})\times C^{2+\a}(\overline{\O},\mathbb
R^3)$ be a solution of \eqref{eq4.7} with $f>0$ on $\overline{\O}$.
If we assume in addition either $\H\in C^{3+\a}(\overline{\O},\mathbb R^3)$ or $\mB\in\mB^{2+\a}(\p\O)$, then there exists $\A\in
C^{2+\a}(\overline{\O},\mathbb R^3)$ such that $\H=\lam\,\curl\A$ and
$(f,\A)$ is a Meissner solution of \eqref{dec-eqfA-1}.
\end{itemize}
\end{Lem}

The proof of this lemma is omitted here, see Lemma \ref{Lem7.1}
for a more general problem.
Now we look for solutions of \eqref{eq4.7}.  Set
\begin{equation}\label{sp-XKUU}
\aligned
X^{2+\a}=&C^{2+\a}(\overline{\O})\times C^{2+\a}(\overline{\O},\mathbb R^3),\\
\mathbb K(\O)=&\mK(\O)\cap X^{2+\a},\q
\overline{\mathbb K}(\O)=\overline{\mK}(\O)\cap X^{2+\a},\q
\mathbb K_\delta(\O)=\mK_\delta(\O)\cap X^{2+\a},\\
\mathbb U(\O)=&\{(f,\mH)\in X^{2+\a}:~ (f,-\lam f^{-2}\curl\mH)\in\mathbb K(\O)\},\\
\overline{\mathbb U}(\O)=&\{(f,\mH)\in X^{2+\a}:~ (f,-\lam f^{-2}\curl\mH)\in\overline{\mathbb K}(\O)\},\\
\mathbb U_\delta(\O)=&\{(f,\mH)\in X^{2+\a}:~ (f,-\lam f^{-2}\curl\mathcal
H)\in\mathbb K_\delta(\O)\}.
\endaligned
\end{equation}
By using the implicit function theorem we can prove the following lemma.\footnote{For the 2 dimensional case see \cite[Proposition 2.5]{BCM}.}

\begin{Lem}\label{Lem4.5}
Let $\O$ satisfy $(O)$ with $r\geq 3$ and $0<\a<1$.
\begin{itemize}
\item[(i)] Assume  $(f^0,\H^0)\in \mathbb
U(\O)$ is a solution of \eqref{eq4.7} corresponding to
$$(\lam,\k,\mB_T)=(\lam^0,\k^0,\mB^0_T)\in \mathbb R_+\times\mathbb
R_+\times \mB^{2+\a}(\p\O).
$$
Then there exists $\eta=\eta(\O,\lam^0,\k^0,\mB^0_T,\a)>0$ such that if
$$
\aligned
&(\lam,\k,\mB_T)\in \mathbb R_+\times\mathbb
R_+\times\mB^{2+\a}(\p\O),\\
&|\lam-\lam^0|+|\k-\k^0|+\|\mB_T-\mathcal
B^0_T\|_{C^{2+\a}(\p\O)}<\eta,
\endaligned
$$
then \eqref{eq4.7} has a unique solution $(f,\H)\in\mathbb U(\O)$ corresponding to $(\lam,\k,\mB_T)$.

\item[(ii)] There exists a number $\var=\var(\O,\lam,\k,\a)>0$ such that for all
$\mB_T\in\mB^{2+\a}(\p\O)$ with $\|\mB_T\|_{C^{2+\a}(\p\O)}<\var$, \eqref{eq4.7} has a unique solution in $\mathbb
U(\O)$, hence \eqref{dec-eqfA-1} has a unique solution in $\mathbb
K(\O)$.
\end{itemize}
\end{Lem}

Next we look for a bound of the boundary data for BVP \eqref{dec-eqfA-1} to have a solution. For this purpose, as in \cite[Section 7]{BaP} and \cite{P2} we fix  a $\mB_T\in\mB^{2+\a}(\p\O)$ with $\a\in (0,1)$, and  consider problems (\ref{dec-eqfA-1}$\mu$) and (\ref{eq4.7}$\mu$), which are the modified versions of \eqref{dec-eqfA-1} and \eqref{eq4.7}, respectively, with $\mB_T$ replaced by $\mu\mB_T$.
Define
\begin{equation}\label{eq-mu-star}
\aligned
\mu^*(\lam,\k,\mB_T)&=\sup\{b>0:~ \forall\;
0\leq\mu<b, \text{ $(\ref{dec-eqfA-1}\mu)$ has a
solution in $\mathbb K(\O)$}\}\\
&=\sup\{b>0:~ \forall\; 0\leq\mu<b, \text{ (\ref{eq4.7}$\mu$) has a solution in $\mathbb U(\O)$}\}.
\endaligned
\end{equation}

\begin{Prop}[Existence]\label{Prop4.6}
Let $\O$ satisfy $(O)$ with $r\geq 3$ and $0<\a<1$,
$\mB_T\in\mB^{2+\a}(\p\O)$ and $\mB_T\not\equiv\0$.
\begin{itemize}
\item[(i)] $0<\mu^*(\lam,\k,\mB_T)<\infty$. For any $0<\mu\leq
\mu^*(\lam,\k,\mB_T)$, $(\ref{dec-eqfA-1}\mu)$ has a unique
solution $(f_\mu,\A_\mu)\in \overline{\mathbb K}(\O)$, and (\ref{eq4.7}$\mu$)
has a unique solution $(f_\mu,\H_\mu)\in\overline{\mathbb U}(\O)$.
If $0<\mu<\mu^*(\lam,\k,\mB_T)$, then
$(f_\mu,\A_\mu)\in \mathbb K(\O)$ and  $(f_\mu,\H_\mu)\in\mathbb
U(\O)$.
\item[(ii)] The map $\mu\mapsto (f_\mu,\A_\mu)$ is continuous from the interval
$[0,\mu^*(\lam,\k,\mB_T)]$ to $C^{2+\a}(\overline{\O})\times
C^{2+\a}(\overline{\O},\mathbb R^3)$.
\item[(iii)] For $\mu^*=\mu^*(\lam,\k,\mB_T)$ we have
\begin{equation}\label{limit-mu}
\aligned
&\lim_{\mu<\mu^*, \mu\to\mu^*}\min_{x\in\overline{\O}}\;\big[f_\mu^2(x)-|\A_\mu(x)|^2\big]={1\over
3},\\
&\lim_{\mu<\mu^*,\mu\to\mu^*}\min_{x\in\overline{\O}}\;\big[f_\mu^2(x)-\lam^2
f_\mu^{-4}(x)|\curl\H_\mu(x)|^2\big]={1\over 3}.
\endaligned
\end{equation}
\end{itemize}
\end{Prop}
\begin{proof} {\it Step 1}. Let us fix $\lam>0$, $\k>0$ and $\mB_T\not\equiv\0$. From Lemma \ref{Lem4.5} (ii) we see that
$\mu^*(\lam,\k,\mB_T)>0$. Denote
$$
\aligned
&\mathcal J[\H]=\int_\O(\lam^2|\curl\H|^2+|\H|^2)dx,\\
&c(\mB_T)=\inf\{\mathcal J[\H]~:~ \H\in H^1(\O,\div0),\; \H_T=\mB_T\;\text{\rm
on }\p\O\}.
\endaligned
$$
Since $\mB_T\not\equiv\0$ we have $c(\mB_T)>0$.
As in the proof of Lemma 6.1 in \cite{P2} we can show that
\begin{equation}\label{ineq-mu-star}
\mu^*(\lam,\k,\mB_T)\leq {\lam\|\mB_T\|_{L^1(\p\O)}\over\min\{\lam^2,1\}c(\mB_T)}.
\end{equation}
In the following we write $\mu^*$ for $\mu^*(\lam,\k,\mB_T)$.
From Lemma \ref{Lem4.2} (i), for any $\mu\in (0,\mu^*)$, ($\ref{dec-eqfA-1}\mu$) has a solution $(f_\mu,\A_\mu)$ which is the unique solution of ($\ref{dec-eqfA-1}\mu$) in $\Bbb K(\O)$. We show that this is also true for $\mu=\mu^*$. Noting that  $1/\sqrt{3}\leq
f_\mu\leq 1$, from Proposition \ref{Prop4.3} (ii)  we see that the set
$$\{(f_\mu,\A_\mu): 0<\mu<
\mu^*\}
$$
is uniformly bounded in $C^{2+\a}(\overline{\O})\times
C^{2+\a}(\overline{\O},\mathbb R^3)$, hence pre-compact in $C^{2+\b}(\overline{\O})\times C^{2+\b}(\overline{\O},\Bbb R^3)$. Let $\mu_j<\mu^*$ and $\mu_j\to\mu^*$. After passing to a subsequence we may assume that
$$(f_{\mu_j},\A_{\mu_j})\to (f,\A)\q\text{in }C^{2+\b}(\overline{\O})\times
C^{2+\b}(\overline{\O},\mathbb R^3)\q\text{as }j\to\infty,
$$
where $(f,\A)$ is a solution of
(\ref{dec-eqfA-1}$\mu^*$), and $(f,\A)\in\overline{\mathbb K}(\O)$. From  Lemma \ref{Lem4.2} (i) we conclude
that $(f,\A)$ is the unique
solution $(f_{\mu^*},\A_{\mu^*})$ in $\overline{\mK}(\O)$.

{\it Step 2}. By the uniqueness and compactness of solutions in $\Bbb K(\O)$ described in step 1, we can show that the map  $\mu\mapsto (f_\mu,\A_\mu)$ is
continuous from $[0,\mu^*]$ to $C^{2+\b}(\overline{\O})\times
C^{2+\b}(\overline{\O},\mathbb R^3)$ for any $0<\b<\a$.

{\it Step 3}.
Now we show this map is continuous from $[0,\mu^*]$ to $C^{2+\a}(\overline{\O})\times
C^{2+\a}(\overline{\O},\mathbb R^3)$. Let $0\leq \mu_0,\mu_1\leq\mu^*$, let $(f_{\mu_0},\A_{\mu_0}), (f_{\mu_1},\A_{\mu_1})\in\overline{\mathbb K}(\O)$ be the solutions of
(\ref{dec-eqfA-1}$\mu$) for $\mu=\mu_0$ and $\mu_1$ respectively, and let
$$\H_{\mu_j}=\lam\,\curl\A_{\mu_j},\q
\mathcal H_{\mu_j}=\H_{\mu_j}-\mu_j\mB,
$$
where $\mB$ is the curl-free extension of $\mB_T$ on $\overline{\O}$. Applying the
Schauder estimate to the equation of $f_{\mu_1}-f_{\mu_0}$ we have
$$\|f_{\mu_1}-f_{\mu_0}\|_{C^{2+\a}(\overline{\O})}\leq
C(\O)\{\|f_{\mu_1}-f_{\mu_0}\|_{C^\a(\overline{\O})}
+\lam^{-2}\k^2\|z\|_{C^\a(\overline{\O})}\},
$$
where
$$z=(1-|f_{\mu_1}|^2-|\A_{\mu_1}|^2)f_{\mu_1}
-(1-|f_{\mu_0}|^2-|\A_{\mu_0}|^2)f_{\mu_0}.
$$
From step 2, the right hand side of the above inequality converges to
zero as $\mu_1\to\mu_0$, so we have
\begin{equation}\label{conv-f}
\lim_{\mu_1\to\mu_0}\|f_{\mu_1}-f_{\mu_0}\|_{C^{2+\a}(\overline{\O})}=0.
\end{equation}

From the second equation in \eqref{eq4.7} we see that
$$\left\{\aligned
&\lam^2\curl[f_{\mu_0}^{-2}\curl(\mH_{\mu_1}-\mathcal
H_{\mu_0})]
+\mH_{\mu_1}-\mH_{\mu_0}=\bold d\q&\text{\rm in }\O,\\
&(\mH_{\mu_1}-\mathcal
H_{\mu_0})_T=\0\qqq\qqq\qq\;&\text{\rm on }\p\O,
\endaligned\right.
$$
where
$$\bold d=(\mu_0-\mu_1)\mB
+\lam^2\curl[(f_{\mu_0}^{-2}-f_{\mu_1}^{-2})\curl\mathcal
H_{\mu_1}].
$$
From step 2 we have
$$\|\bold d\|_{C^\a(\overline{\O})}\to 0\q\text{as }\mu_1\to\mu_0.
$$
Since   $\O$ is simply-connected and has no holes, and $\div(\mH_{\mu_1}-\mH_{\mu_0})=0$, we
can apply Lemma \ref{Lem-linear-curl} (ii) with $a=f_{\mu_0}^{-2}$ to $\mH_{\mu_1}-\mH_{\mu_0}$ and
find that
$$\aligned
&\|\mH_{\mu_1}-\mH_{\mu_0}\|_{C^{2+\a}(\overline{\O})}\leq
C\|\bold
d\|_{C^\a(\overline{\O})}\to 0\q\text{as }\mu_1\to\mu_0,\\
\endaligned
$$
where $C$ depends only on $\O$, $\a$,
$\|f_{\mu_0}^2\|_{C^{1+\a}(\overline{\O})}$ and
$\|f_{\mu_0}^{-2}\|_{C^1(\overline{\O})}$. Therefore as $\mu_1\to\mu_0$,
\begin{equation} \label{conv-H}
\aligned
\|\H_{\mu_1}-\H_{\mu_0}\|_{C^{2+\a}(\overline{\O})} \leq \|\mathcal
H_{\mu_1}-\mathcal
H_{\mu_0}\|_{C^{2+\a}(\overline{\O})}+|\mu_1-\mu_0|\|\mathcal
B\|_{C^{2+\a}(\overline{\O})}\to 0. \endaligned
\end{equation}
Note  that
$$\div(\A_{\mu_1}-\A_{\mu_0})=2[f_{\mu_0}^{-1}\nabla
f_{\mu_0}\A_{\mu_0}-f_{\mu_1}^{-1}\nabla f_{\mu_1}\A_{\mu_1}].
$$
Then from \eqref{conv-f} and step 2, we see that
$$
\lim_{\mu_1\to\mu_0}\|\div(\A_{\mu_1}-\A_{\mu_0})\|_{C^{1+\a}(\overline{\O})}=0.
$$
From this, \eqref{conv-H} and the fact $\nu\cdot\A_\mu=0$, and since $\O$ is simply-connected, we  apply \eqref{dcg3} to get
$$
\lim_{\mu_1\to\mu_0}\|\A_{\mu_1}-\A_{\mu_0}\|_{C^{2+\a}(\overline{\O})}=0.
$$

{\it Step 4.} The two equalities in \eqref{limit-mu} are equivalent, and
we only need to prove the first one. From step 2 we see that the
function
$$
g(\mu)=\min_{x\in\overline{\O}}\;[f_\mu^2(x)-|\A_\mu(x)|^2]
$$
is continuous in $\mu\in [0,\mu^*]$. By the definition of $\mu^*$ we have
$g(\mu)>1/3$ for any $0<\mu<\mu^*$. So
$$\inf_{0\leq\mu\leq
b}\,g(\mu)>{1\over 3}\q\text{for any }0<b<\mu^*.
$$
Now we show
$$
\liminf_{\mu<\mu^*, \mu\to\mu^*}g(\mu)={1\over 3}.
$$
If not, there exists $\var>0$ such that $g(\mu)\geq 1/3+\var$ for
all $\mu<\mu^*$. From this and continuity of $g(\mu)$ we see that
$g(\mu^*)\geq 1/3+\var$. Then from Lemma \ref{Lem4.5} (i) we conclude that
there exists $\eta>0$ such that (\ref{dec-eqfA-1}$\mu$) has a solution for all
$\mu\in [\mu^*,\mu^*+\eta)$ with $g(\mu)>1/3+\var/2$. This
contradicts the definition of $\mu^*$.

Since $g(\mu)$ is continuous at $\mu^*$ from left, we conclude that
$$g(\mu^*)=\lim_{\mu<\mu^*,\mu\to\mu^*}g(\mu)=\liminf_{\mu<\mu^*,\mu\to\mu^*}g(\mu)={1\over
3}.
$$
So the first equality in \eqref{limit-mu} is proved. \end{proof}

\subsection{Convergence}\

\begin{Prop}[Estimates for large $\k$]\label{Prop4.7}
Let $\O$ satisfy $(O)$ with $r\geq 3$ and $0<\a<1$,
$\mB_T\in\mB^{2+\a}(\p\O)$, and $0<\delta<1/3$. Let $(f_\k,\A_\k)\in \mathbb
K_\delta(\O)$ be a Meissner solution of BVP \eqref{dec-eqfA-1}, and let $\A_\infty\in
C^{2+\a}(\overline{\O},\mathbb R^3)$ be a solution of BVP \eqref{eqA-Omega} with $\mH_T$ replaced by $\mB_T$, and satisfy \eqref{cond-1.8}. Denote $$\H_\k=\lam\,\curl\A_\k,\q
\H_\infty=\lam\,\curl\A_\infty,\q
f_\infty(x)=(1-|\A_\infty(x)|^2)^{1/2}.
$$
\begin{itemize}
\item[(i)] For all $\lam>0$ and $\k\geq \max\{1,\lam\}$, we have
\begin{equation}\label{est-4.16}
\aligned
\|f_\k-f_\infty\|_{L^2(\O)}+\|\A_\k-\A_\infty\|_{L^2(\O)}+\|\H_\k-\H_\infty\|_{L^2(\O)}\leq& C\k^{-3/2},\\
\|f_\k-f_\infty\|_{H^1(\O)}+\|\A_\k-\A_\infty\|_{H^1(\O)}+\|\H_\k-\H_\infty\|_{H^1(\O)}
\leq& C\k^{-1/2},\\
\|f_\k-f_\infty\|_{H^2(\O)}+\|\A_\k-\A_\infty\|_{H^2(\O)}+\|\H_\k-\H_\infty\|_{H^2(\O)}\leq& C\k^{1/2},
\endaligned
\end{equation}
where $C=C(\O,\delta,\lam,\mB_T)$.
\item[(ii)]
\begin{equation}\label{est-4.17}
\lim_{\k\to+\infty}\;\{\|f_\k-f_\infty\|_{C^0(\overline{\O})}+
\|\A_\k-\A_\infty\|_{C^0(\overline{\O})}\} =0.
\end{equation}
\end{itemize}
\end{Prop}

\begin{proof}
The proof of (i) will be given in Appendix \ref{AppendixD}.
In the following we prove (ii). To show the Meissner solutions $(f_\k(x),\A_\k(x))$ uniformly converges, we use the argument by contradiction. Direct computations  show that the rescaled functions $(f^\k(y),\A^\k(y))$ of the Meissner solutions approach a limit field $(f(y),\A(y))$, which is either a solution of \eqref{eqinR3} for $y\in \Bbb R^3$, or a solution of \eqref{eqinhalfsp} for $y\in \Bbb R^3_+$. The key step is to prove that $(f(y),\A(y))$ must be a constant solution.
In the two dimensional case, the nice $C^{k+\a}$ estimates
on $f_\k(x)-f_\infty(x)$ and on $\A_\k(x)-\A_\infty(x)$ are established in \cite{BCM} by using of the Gagliardo-Nirenberg inequality, which yields the uniform convergence of $(f_\k(x),\A_\k(x))$, and actually it also implies that $\A^\k(y)-\A_\infty(x_0)$ converges to zero in $C^0_{\loc}$ topology. Thus $(f(y),\A(y))$ is a solution of either \eqref{eqinR3} or \eqref{eqinhalfsp} with $|\A(y)|\equiv |\A_\infty(x_0)|$, a constant. Then using \cite[Lemma 5.4]{LuP} we can easily show that $f(y)$ must be a constant. However in the three dimensional case we do not have these $C^{k+\a}$ estimate on $f_\k(x)-f_\infty(x)$ and on $\A_\k(x)-\A_\infty(x)$. This makes our proof a bit involving. The key point in our proof is to show that the limiting field $(f(y),\A(y))$ is not only a solution of either \eqref{eqinR3} or \eqref{eqinhalfsp} but it also has the additional property \eqref{property-R3} or \eqref{property-R3+} respectively, which makes it possible to show  that $(f(y),\A(y))$ is a constant solution.

Now we begin to prove (ii). Suppose \eqref{est-4.17} were not true. Then there would exist $\eta>0$ and a sequence
$\k_j\to+\infty$ such that
$$\aligned
&\text{either }\q \|f_{\k_j}-f_\infty\|_{C^0(\overline{\O})}\geq\eta\q\;\;\,\text{for all $j$},\\
&\text{or }\;\; \qq\|\A_{\k_j}-\A_\infty\|_{C^0(\overline{\O})} \geq\eta\q\text{for all $j$.}
\endaligned
$$
For simplicity of notation we denote $\k_j$ by $\k$, and assume $x_{\k}\in \overline{\O}$ such that
\begin{equation}\label{ineq4.18}
\aligned
&\text{either }\q |f_{\k}(x_{\k})-f_\infty(x_{\k})|=\|f_{\k}-f_\infty\|_{C^0(\overline{\O})}\geq\eta
\qq\;\;\text{for all } \k,\\
&\text{or }\qq\;\; |\A_{\k}(x_{\k})-\A_\infty(x_{\k})|
=\|\A_{\k}-\A_\infty\|_{C^0(\overline{\O})} \geq\eta\q\text{for all } \k.
\endaligned
\end{equation}
Passing to another subsequence we may assume that $x_\k\to x_0$.
Denote
$$\O_\k=\k (\O-\{x_\k\}),\q \rho_\k=\k\,\dist(x_\k,\p\O),
$$
and set
\begin{equation}\label{4.19}
\aligned
 f^\k(y)=f_\k\Big(x_\k+{y\over\k}\Big),\q
\A^\k(y)=\A_\k\Big(x_\k+{y\over\k}\Big),\q
\A_\infty^\k(y)=\A_\infty\Big(x_\k+{y\over\k}\Big).
\endaligned
\end{equation}
Then $(f^\k,\A^\k)$ satisfies
\begin{equation}\label{eq4.20}
\left\{\aligned
-&\lam^2\Delta_y f^\k=(1-{|f^\k|}^2-|\A^\k|^2) f^\k
\q\;\text{ \rm in }\O_\k,\\
&\lam^2\curl_y^2\A^\k+{1\over\k^2}{|f^\k|}^2\A^\k=\bold 0
\qq\q\;\text{\rm in }\O_\k,\\
&{\p f^\k\over\p\nu}=0,\q
\lam(\curl_y{\A^\k})^-_T={1\over\k}\tilde{\mB}_T\q\text{ \rm
on }\p\O_\k.
\endaligned\right.
\end{equation}
Passing to a subsequence again and rotating the coordinate system if necessary, we may only consider the following two cases.

{\it Case 1.} $\lim_{\k\to+\infty}\rho_\k=+\infty$.

{\it Step 1.1}. In this case, for any $R>0$, there exists $\k(R)$ such that for all $\k>\k(R)$ we have $R<\rho_\k$, so
$B_R(0)\subset\O_\k$. We show that, after passing to a subsequence again if necessary, we have, as $\k\to\infty$,
\eq\label{conv4.21}
\aligned
&f^\k\to f\q\text{in } C^{2+\a}_{\loc}(\Bbb R^3),\q\text{and weakly in }H^1_{\loc}(\Bbb R^3),\\
& \A^\k\to\A\q\text{weakly in }H^1_{\loc}(\Bbb R^3,\Bbb R^3),\q\text{and strongly in } L^2_{\loc}(\Bbb R^3,\Bbb R^3),\\
&f-f_\infty(x_0)\in H^1(\Bbb R^3),\q \A-\A_\infty(x_0)\in H^1(\Bbb R^3,\Bbb R^3);\\
&\curl\A=\0,\q \div(f^2\A)=0\q\text{and}\q |\A(y)|^2\leq f^2(y)-{1\over 3}-\delta\q\text{in }\Bbb R^3.
\endaligned
\eeq

From Proposition \ref{Prop4.3} we have the estimate \eqref{est-3.17}, using which we can show that $\{f^\k\}$ is bounded in $C^{2+\a}_{\loc}(\mathbb
R^3)$. Since $(f_\k,\A_\k)\in\mathbb K_\delta(\O)$,
from \eqref{est-4.16} we have, for any $R>0$ and any $\k$,
\eq\label{est-fkAk}
\aligned
&\int_{B(0,R)}\{|f^\k(y)-f_\infty^\k(y)|^2+|Df^\k(y)-Df_\infty^\k(y)|^2\}dy
\leq C,\\
&\int_{B(0,R)}\{|\A^\k(y)-\A_\infty^\k(y)|^2
+|D\A^\k(y)-D\A_\infty^\k(y)|^2\}dy
\leq C,
\endaligned
\eeq
where $C$ is independent of $\k$ and $R$.
Hence $\{f^\k-f^\k_\infty\}$ is bounded in $H^1_{\loc}(\Bbb R^3)$ and $\{\A^\k-\A^\k_\infty\}$ is bounded in $H^1_{\loc}(\Bbb R^3,\Bbb R^3)$.
Noting that
$$\aligned
&\int_{B(0,R)}|Df_\infty^\k(y)|^2dy
=\k\int_{B(x_\k,R/\k)}|Df_\infty(x)|^2dx
\leq C\|Df_\infty\|_{C^0(\overline{\O})}^2R^3\k^{-2},\\
&\int_{B(0,R)}|D\A_\infty^\k(y)|^2dy
=\k\int_{B(x_\k,R/\k)}|D\A_\infty(x)|^2dx
\leq C\|D\A_\infty\|_{C^0(\overline{\O})}^2R^3\k^{-2},
\endaligned
$$
so we have
$f^\k_\infty(y)\to f_\infty(x_0)$ in $H^1_{\loc}(\Bbb R^3)$ and
$\A^\k_\infty(y)\to \A_\infty(x_0)$ in $H^1_{\loc}(\Bbb R^3,\Bbb R^3)$. Thus
$\{f^\k-f_\infty(x_0)\}$ is bounded in $H^1_{\loc}(\Bbb R^3)$ and $\{\A^\k-\A_\infty(x_0)\}$ is bounded in $H^1_{\loc}(\Bbb R^3,\Bbb R^3)$.
Therefore, after passing to another subsequence, there exist  a function $g$ and a vector field $\B$ such that, as $\k\to\infty$,
$$\aligned
&f^\k-f_\infty(x_0)\to g\q\text{in } C^{2+\a}_{\loc}(\Bbb R^3),\q\text{and weakly in }H^1_{\loc}(\Bbb R^3),\\
& \A^\k-\A_\infty(x_0)\to\B\q\text{weakly in }H^1_{\loc}(\Bbb R^3,\Bbb R^3),\q\text{and strongly in } L^2_{\loc}(\Bbb R^3,\Bbb R^3).
\endaligned
$$
Write
$$a=f_\infty(x_0),\q \bold b=\A_\infty(x_0),\q
f=g+f_\infty(x_0)=g+a,\q
\A=\B+\A_\infty(x_0)=\B+\bold b.
$$
Then we get the first two lines in \eqref{conv4.21}.
Letting $\k$ go to infinity in \eqref{est-fkAk} we get
$$
\int_{B(0,R)}(|g|^2+|Dg|^2)dy
\leq C,\qq
\int_{B(0,R)}(|\B|^2
+|D\B|^2)dy
\leq C,
$$
where $C$ is independent of $R$, so $g\in H^1(\Bbb R^3)$ and $\B\in H^1(\Bbb R^3,\Bbb R^3)$.

Since  $f^\k(y)\to f(y)$ and $\A^\k(y)\to\A(y)$ for a.e. $y\in\Bbb R^3$, using \eqref{dec-eqfA-1} we see that $(f,\A)$ is a solution of the following equations for $y\in \Bbb R^3$:
\eq\label{eqinR3}
-\lam^2\Delta f=(1-f^2-|\A|^2)f,\q \curl^2\A=\0\q\text{in }\Bbb R^3.
\eeq
From the second equation in \eqref{eqinR3} and since $\curl\A=\curl\B$, we have
$$\int_{\Bbb R^3}\curl\B\cdot\curl\D\, dx=0,\q\forall\; \D\in C^1_c(\Bbb R^3,\Bbb R^3).
$$
Since $\B\in H^1(\Bbb R^3,\Bbb R^3)$, we can approximate $\B$ in $H^1(\Bbb R^3,\Bbb R^3)$ by a sequence $\D_j\in C^1_c(\Bbb R^3,\Bbb R^3)$, then apply the above equality with  $\D=\D_j$ and take limit as $j\to\infty$ to obtain
$$\int_{\Bbb R^3}|\curl\B|^2 dx=0.
$$
So $\curl\A=\curl\B=\0$ for a.e. $y\in\Bbb R^3$.

From the second equation in \eqref{eq4.20} we have
$\div(|f^\k|^2\A^\k)=0$ in $\O_\k$. Taking limit we get
$\div(f^2\A)=0$ in $\Bbb R^3$.
Since $(f_\k,\A_\k)\in\Bbb K_\delta(\O)$, so
$|\A^\k(y)|^2<|f^\k(y)|^2-{1\over3}-\delta$,
hence
$$|\A(y)|^2\leq\liminf_{\k\to\infty}|\A^\k(y)|^2\leq |f(y)|^2-{1\over3}-\delta,\q\forall y\in\Bbb R^3.
$$
Now \eqref{conv4.21} is proved.

{\it Step 1.2}. Since $\div(f^2\A)=0$ in $\Bbb R^3$, we have
\eq\label{eq4.24}
\int_{\Bbb R^3}f^2\A\cdot\nabla\zeta\, dy=0,\qq\forall\; \zeta\in C^1_c(\Bbb R^3).
\eeq
Since $\B\in H^1(\Bbb R^3,\curl0)$, there exists a function $\phi$ with $\nabla \phi\in H^1(\Bbb R^3,\Bbb R^3)$
such that $\B=\nabla\phi$, so $\A=\nabla\phi+\bold b$.
Note that $a^2+|\bold b|^2=1$, $g\in L^2(\Bbb R^3)\cap L^\infty(\Bbb R^3)$, and $\B\in L^2(\Bbb R^3,\Bbb R^3)\cap L^\infty(\Bbb R^3,\Bbb R^3)$, so
$$1-|f(x)|^2-|\A(x)|^2=-2ag-g^2-2\bold b\cdot\B-|\B|^2\in L^2(\Bbb R^3).
$$
From this and \eqref{eqinR3}, and since $\nabla f=\nabla g\in L^2(\Bbb R^3,\Bbb R^3)$, we have
\eq\label{eq4.25}
\int_{\Bbb R^3}\{\lam^2\nabla f\cdot\nabla h-(1-f^2-|\A|^2)fh\} dy=0,\qq\forall\;  h\in H^1(\Bbb R^3).
\eeq
Now we see that the limit $(f,\A)$ is a solution of \eqref{eqinR3} and has the following property (with $x_0\in\Bbb R^3$ and $0<\delta<1$ being given):
\eq\label{property-R3}
\aligned
&f=a+g,\q  a=f_\infty(x_0),\q g\in H^1(\Bbb R^3),\\
&\A=\bold b+\nabla\phi,\q \bold b=\A_\infty(x_0),\q \nabla\phi\in H^1(\Bbb R^3,\Bbb R^3),\\
&\div(f^2\A)=0\q\text{and}\q |\A(y)|^2\leq |f(y)|^2-{1\over 3}-\delta\q \text{in }\Bbb R^3.
\endaligned
\eeq

{\it Step 1.3}.
Now we show  that \eqref{eqinR3} has only one solution that has the property \eqref{property-R3}. Otherwise suppose \eqref{eqinR3} has two solutions $(f_0,\A_0)$ and $(f_1,\A_1)$ and they have the property \eqref{property-R3},
that is,
$$
\aligned
&f_0=a+g_0,\q f_1=a+g_1,\q g_0,\; g_1\in H^1(\Bbb R^3),\\
&\A_0=\bold b+\nabla\phi_0,\q \A_1=\bold b+\nabla\phi_1,
\qq \nabla\phi_0,\; \nabla\phi_1\in H^1(\Bbb R^3,\Bbb R^3),\\
&|\A_0(y)|^2\leq |f_0(y)|^2-{1\over 3}-\delta,\q
|\A_1(y)|^2\leq |f_1(y)|^2-{1\over 3}-\delta,
\endaligned
$$
and they satisfy \eqref{eq4.24}.
Write
$$\aligned
&h=g_1-g_0=f_1-f_0,\q f_t=f_0+th,\\
&\psi=\phi_1-\phi_0,\q
\D=\A_1-\A_0=\nabla\psi,\q
\A_t=\A_0+t\D,\q 0\leq t\leq 1.
\endaligned
$$
Then $h\in H^1(\Bbb R^3)$, $\D\in H^1(\Bbb R^3,\Bbb R^3)$, and using Remark \ref{Rem4.1}(d) we have
\eq\label{eq4.27}
|f_t(y)|^2-|\A_t(y)|^2-{1\over 3}-\delta\geq 0,\q\forall\; y\in\Bbb R^3.
\eeq

We first apply \eqref{eq4.25} to $f_0$ and $f_1$ with this choice of $h$  and obtain
\eq\label{eq4.28}
\aligned
0=&\int_{\Bbb R^3}\{\lam^2|\nabla h|^2-(1-|f_1|^2-|\A_1|^2)f_1h+(1-|f_0|^2-|\A_0|^2)f_0h\}dy\\
=&\int_{\Bbb R^3}\Big\{\lam^2|\nabla h|^2-\int_0^1{d\over dt}(1-|f_t|^2-|\A_t|^2)f_t h dt\Big\}dy\\
=&\int_0^1dt\int_{\Bbb R^3}\{\lam^2|\nabla h|^2+(3|f_t|^2+|\A_t|^2-1)h^2+2f_th\A_t\cdot\D\}dy.
\endaligned
\eeq
Then we apply \eqref{eq4.24} to $(f_0,\A_0)$ and $(f_1,\A_1)$ to get, for any $\zeta\in C^1_c(\Bbb R^3)$,
$$\aligned
0=&\int_{\Bbb R^3}(f_1^2\A_1-f_0^2\A_0)\cdot\nabla\zeta dy
=\int_{\Bbb R^3}\int_0^1{d\over dt}(f_t^2\A_t)\cdot\nabla\zeta dtdy\\
=&\int_0^1dt\int_{\Bbb R^3}(f_t^2\D+2f_th\A_t)\cdot\nabla\zeta dy.
\endaligned
$$
Since $h\in L^2(\Bbb R^3)$ and $\D\in L^2(\Bbb R^3,\Bbb R^3)$, we see that
$f_t^2\D+2f_th\A_t\in L^2(\Bbb R^3,\Bbb R^3)$. After approximating $\eta_\rho(\psi-c_\rho)$ by smooth functions $\zeta_j$ with compact support,  applying the above equality with $\zeta=\zeta_j$, and taking limit as $j\to\infty$, we find
\eq\label{eq4.29}
\int_0^1dt\int_{\Bbb R^3}(f_t^2\D+2f_th\A_t)\cdot\nabla[\eta_\rho(\psi-c_\rho)] dy=0,
\eeq
where  $c_\rho$ is a constant, and $\eta_\rho$ is a cut-off function such that
$\eta_\rho(y)=1$ for $|y|\leq \rho$, $\eta_\rho(y)=0$ for $|y|\geq 2\rho$, and $|\nabla\eta_\rho|\leq C_1/\rho$, where $C_1$ is independent of $\rho$.
For each $\rho$ we choose a suitable $c_\rho$ and using the Poincar\'e inequality to get that
$$\int_{\rho\leq|y|\leq 2\rho}|\psi-c_\rho|^2dy\leq
C_2\rho^2\int_{\rho\leq|y|\leq 2\rho}|\nabla\psi|^2dy,
$$
where $C_2$ is independent of $\rho$. Since $\nabla\psi\in L^2(\Bbb R^3,\Bbb R^3)$
and $f_t^2\D+2f_th\A_t\in L^2(\Bbb R^3,\Bbb R^3)$, we have, as $\rho\to\infty$,
$$
\int_{\rho\leq|y|\leq 2\rho}|f_t^2\D+2f_th\A_t|^2dy\to 0,\qq \int_{\rho\leq|y|\leq 2\rho}|\nabla\psi|^2dy\to 0.
$$
So we have
$$\aligned
&\Bigl|\int_{\rho\leq|y|\leq 2\rho}(f_t^2\D+2f_th\A_t)\cdot(\psi-c_\rho)\nabla\eta_\rho]dy\Bigr|\\
\leq&C_1\rho^{-1}\Bigl(\int_{\rho\leq|y|\leq 2\rho}|f_t^2\D+2f_th\A_t|^2dy\Bigr)^{1/2}
\Bigl(\int_{\rho\leq|y|\leq 2\rho}|\psi-c_\rho|^2dy\Bigr)^{1/2}\\
\leq &C_1C_2^{1/2}\Bigl(\int_{\rho\leq|y|\leq 2\rho}|f_t^2\D+2f_th\A_t|^2dy\Bigr)^{1/2}
\Bigl(\int_{\rho\leq|y|\leq 2\rho}|\nabla\psi|^2dy\Bigr)^{1/2}
\to 0\q\text{as }\rho\to\infty.
\endaligned
$$
From this and \eqref{eq4.29}, and recalling that $\nabla\psi=\bold D$, we obtain
\eq\label{eq4.30}
\int_0^1dt\int_{\Bbb R^3}(f_t^2\D+2f_th\A_t)\cdot\D dy=0.
\eeq

Adding \eqref{eq4.28} and \eqref{eq4.30} together we get
$$\aligned
0=&\int_0^1dt\int_{\Bbb R^3}\{\lam^2|\nabla h|^2+(3|f_t|^2+|\A_t|^2-1)h^2+4f_th\A_t\cdot\D+f_t^2|\D|^2\}dy\\
=&\int_0^1dt\int_{\Bbb R^3}\{\lam^2|\nabla h|^2+(3|f_t|^2-3|\A_t|^2-1)h^2+|2h\A_t+f_t\D|^2\}dy.
\endaligned
$$
So we have $h=0$ and $\D=\0$.
Hence the only solution of \eqref{eqinR3} that has property \eqref{property-R3} is unique.

Note that $(f,\A)=(a,\bold b)=(f_\infty(x_0),\A_\infty(x_0))$ is a solution of \eqref{eqinR3} and has property \eqref{property-R3}, hence the unique solution of \eqref{eqinR3} having property \eqref{property-R3} must be $(f,\A)=(f_\infty(x_0),\A_\infty(x_0))$, namely, in \eqref{property-R3} we must have $g=0$ and $\nabla\phi=\0$.

Therefore the limit field $(f(y),\A(y))$ obtained in Step 1.1 must be constant:  $f\equiv f_\infty(x_0)$ and $\A\equiv \A_\infty(x_0)$.

{\it Step 1.4}. From Steps 1.1-1.3 we have, as $\k\to\infty$,
\eq\label{eq4.31}
\aligned
&f^\k\to f_\infty(x_0)\q\text{in }C^{2+\a}_{\loc}(\Bbb R^3),\\
&\A^\k\to\A_\infty(x_0)\q\text{weakly in }H^1_{\loc}(\Bbb R^3,\Bbb R^3),\q\text{and strongly in } L^2_{\loc}(\Bbb R^3,\Bbb R^3).
\endaligned
\eeq
Hence for any fixed $R>0$ we have
\eq\label{eq4.32}
\max_{y\in \overline{B}(0,R)}|f^\k(y)-f_\infty(x_0)|\to 0,\q
\max_{y\in \overline{B}(0,R)}|\Delta f^\k(y)|\to 0.
\eeq
So
\eq\label{eq4.33}
\|f_\k-f_\infty\|_{C^0(\overline{\O})}=|f_\k(x_\k)-f_\infty(x_0)|=|f^\k(0)-f_\infty(x_0)|\to 0.
\eeq

From the second equality in \eqref{eq4.32}, the fact
 $|f^\k(y)|^2\geq 1/3$, and the first equation in \eqref{eq4.20}, we find
$$\max_{y\in \overline{B}(0,R)}|1-|f^\k(y)|^2-|\A_k(y)|^2|\to 0,
$$
so
$$\aligned
&\max_{y\in \overline{B}(0,R)}||\A_\infty(x_0)|^2-|\A^\k(y)|^2|
=\max_{y\in \overline{B}(0,R)}|1-|f_\infty(x_0)|^2-|\A^\k(y)|^2|\\
=&\max_{y\in \overline{B}(0,R)}|1-|f^\k(y)|^2-|\A^\k(y)|^2|\to 0.
\endaligned
$$
Therefore, if $\A_\infty(x_0)=\0$, then we have
$$\|\A^\k-\A_\infty(x_0)\|_{C^0(\overline{B}(0,R)}\to 0;
$$
and if $\A_\infty(x_0)\neq\0$, then there exists  an orthogonal matrix-valued function $Q^\k(y)$ such that
$$\max_{y\in\overline{B}(0,R)}|\A^\k(y)-Q^\k(y)\A_\infty(x_0)|\to 0.
$$
From this and the second line in \eqref{eq4.31} we see that
$Q^\k(y)=I$, the identity matrix, for a.e. $y\in\Bbb R^3$. Since $\A^\k(y)$ is continuous, we must have
$Q^\k(y)=I$ for all $y\in\Bbb R^3$. Therefore
$$\max_{y\in\overline{B}(0,R)}|\A^\k(y)-\A_\infty(x_0)|\to 0.
$$
Hence
$$
\|\A_{\k}-\A_\infty\|_{C^0(\overline{\O})}=|\A_{\k}(x_\k)-\A_\infty(x_\k)|
=|\A^\k(0)-\A_\infty(x_\k)|\to 0.
$$
Combining this with \eqref{eq4.33} we see that \eqref{ineq4.18} can not be true.

\v0.05in

{\it Case 2.} $\k\,\dist(x_\k,\p\O)\leq C$ for all $\k$.

We may assume, after passing to a subsequence, that $x_\k\to
x_0$, and $x_0\in\p\O$. We define $f^\k$ and $\A^\k$ as in \eqref{4.19}.
As in case 1 we can show that, as $\k\to\infty$,
$$\aligned
&f^\k\to f\q\text{in } C^{2+\a}_{\loc}(\Bbb R^3_+),\q\text{and weakly in }H^1_{\loc}(\Bbb R^3_+),\\
& \A^\k\to \A\q\text{weakly in }H^1_{\loc}(\Bbb R^3_+,\Bbb R^3),\q\text{and strongly in } L^2_{\loc}(\Bbb R^3_+,\Bbb R^3),
\endaligned
$$
where
$$\mathbb R^3_+=\{x\in\mathbb R^3: x_3>0\},\q
f-f_\infty(x_0)\in H^1(\Bbb R^3_+),\q
\A-\A_\infty(x_0)\in H^1(\Bbb R^3_+,\Bbb R^3).
$$
Since $(f_\k,\A_\k)$ is a Meissner solution, so $\nu\cdot\A_\k=0$ on $\p\O$, hence
$\nu\cdot\A^\k=0$ on $\p\O_\k$. Since $\A^\k\to\A$ weakly in $H^1_{\loc}(\Bbb R^3_+,\Bbb R^3)$, we see that $\nu\cdot\A=0$ on $\p\Bbb R^3_+$
in the sense of trace in $H^{1/2}_{\loc}(\p\Bbb R^3_+)$. Hence $(f,\A)$ is a solution of the following equations in $\Bbb R^3_+$:

\eq\label{eqinhalfsp}
\left\{\aligned
-&\lam^2\Delta f=(1-f^2-|\A|^2)f,\q \curl^2\A=\0\q&\text{in }\Bbb R^3_+,\\
&{\p f\over\p\nu}=0,\q \nu\cdot\A=0,\q (\curl\A)_T=\0\q&\text{\rm on}\;\;\p\Bbb R^3_+.
\endaligned\right.
\eeq
Moreover $f^2(y)\geq 1/3+\delta$.
Since $\curl\A\in L^2(\Bbb R^3_+,\Bbb R^3)$,
$\curl^2\A=\0$ in $\Bbb R^3_+$ and $(\curl\A)_T=\0$ on $\p\Bbb R^3_+$,
we obtain as in case 1 that
$$\int_{\Bbb R^3_+}|\curl\A|^2dy=0,
$$
hence
$\curl\A=\0$ in $\Bbb R^3_+$.
So there exists a function $\phi$ with $\nabla\phi\in H^1(\Bbb R^3_+,\Bbb R^3)$ such that $\A=\nabla\phi+\bold b$, where $\bold b=\A_\infty(x_0)$.
Therefore the limit $(f,\A)$ is a solution of \eqref{eqinhalfsp} in $\Bbb R^3_+$, and it has the following property (with $x_0\in\Bbb R^3_+$ and $0<\delta<1$ being given):
\eq\label{property-R3+}
\aligned
&f=a+g,\q  a=f_\infty(x_0),\q g\in H^1(\Bbb R^3_+),\\
&\A=\bold b+\nabla\phi,\q \bold b=\A_\infty(x_0),\q \nabla\phi\in H^1(\Bbb R^3_+,\Bbb R^3),\\
& \div(f^2\A)=0\q\text{in }\Bbb R^3_+,\q \nu\cdot\A=0\q\text{on }\p\Bbb R^3_+,\\
&|\A(y)|^2\leq |f(y)|^2-{1\over 3}-\delta,\q\forall y\in\Bbb R^3_+.
\endaligned
\eeq

If \eqref{eqinhalfsp} has two solutions $(f_0,\A_0)$ and $(f_1,\A_1)$ which have property \eqref{property-R3+}, then we argue as in Case 1 to obtain \eqref{eq4.28} and \eqref{eq4.30} with $\Bbb R^3$ replaced by $\Bbb R^3_+$ and with
$h=f_1-f_0$, $\D=\A_1-\A_0$, $f_t=f_0+th$, $\A_t=\A_0+t\D$.
The inequality \eqref{eq4.27} remains true with $\Bbb R^3$ replaced by $\Bbb R^3_+$.
So we also get $h=0$ and $\D=\0$, that is, $f_0=f_1$ and $\A_0=\A_1$. Hence \eqref{eqinhalfsp} has only one solution which satisfies \eqref{property-R3+}. Obviously
$(f,\A)=(a,\bold b)=( f_\infty(x_0), \A_\infty(x_0))$ is a solution of \eqref{eqinhalfsp} satisfying \eqref{property-R3+}, so it must be the only solution of \eqref{eqinhalfsp} having property \eqref{property-R3+}. Hence the limit $(f,\A)$ of the rescaled functions must be equal to
$(f_\infty(x_0), \A_\infty(x_0))$.
Therefore we have, as $\k\to\infty$,
$$
\aligned
&f^\k\to f_\infty(x_0)\q\text{in }C^{2+\a}_{\loc}(\overline{\Bbb R^3_+}),\\
&\A^\k\to\A_\infty(x_0)\q\text{weakly in }H^1_{\loc}(\Bbb R^3_+,\Bbb R^3),\q\text{and strongly in } L^2_{\loc}(\Bbb R^3_+,\Bbb R^3).
\endaligned
$$
Finally we argue as in Step 1.4 to show that, as $\k\to\infty$,
$$f_\k(x_\k)=f^\k(0)\to f_\infty(x_0),\q \A_\k(x_\k)=\A^\k(0)\to\A_\infty(x_0).
$$
Again we see that \eqref{ineq4.18} can not be true.
\end{proof}

\subsection{Estimate of $\mu^*(\lam,\k,\mB_T)$ for large $\k$}\

Now we consider BVP (\ref{eqA-Omega}$\mu$) which is a modified version of \eqref{eqA-Omega} with $\mB_T$ replaced by $\mu\mB_T$.
Define
\begin{equation}\label{mu-star-lam}
\mu^*(\lam,\mB_T)=\sup\;\Big\{b>0:~\forall 0<\mu<b,\;\text{(\ref{eqA-Omega}$\mu$) has
a solution $\A$
with $\|\A\|_{C^0(\overline{\O})}<{1\over\sqrt{3}}$}\Big\}.
\end{equation}
The following conclusions have been proved in \cite[Section 7]{BaP} under the
condition that $\O$ is a bounded and simply-connected domain with a $C^4$ boundary and without holes:\footnote{Please note that in \cite[Section 7]{BaP} the conclusions are stated with respect to the equivalent system for $\H=\lam\,\curl\A$.}

(i) For any $\mB_T\in \mB^{2+\a}(\p\O)$ not identically zero and $\lam>0$, we have
$0<\mu^*(\lam,\mB_T)<\infty$.

(ii) For any $\mu\in (0,\mu^*(\lam,\mB_T))$,
(\ref{eqA-Omega}$\mu$) has a unique solution $\mA_\mu$ with
$\|\mA_\mu\|_{C^0(\overline{\O})}<1/\sqrt{3}$.

(iii) $\mu^*(\lam,\mB_T)$ has the following
characterization\footnote{The equality \eqref{limit-1/3} follows from
\cite[(7.6)]{BaP} and the relation between $\|\mathcal
A_\mu\|_{C^0(\overline{\O})}$ and $\lam\|\curl\mathcal
H_\mu\|_{C^0(\overline{\O})}$, where $\mH_\mu=\lam\;\curl\mathcal
A_\mu$. Please note that to avoid confusion here we use $\mathcal
A_\mu$ and $\mH_\mu$ to denote the vector fields $\A_\mu$ and
$\H_\mu$ in \cite{BaP}.}
\begin{equation}\label{limit-1/3}
\lim_{\mu\to \mu^*(\lam,\mB_T)^-}\|\mathcal
A_\mu\|_{C^0(\overline{\O})}={1\over\sqrt{3}}.
\end{equation}

(iv) The following asymptotic estimate holds:
\begin{equation}\label{limit-mu-star}
\lim_{\lam\to 0^+}\mu^*(\lam,\mB_T)=\sqrt{5\over
18}\Big(\|\mB_T\|_{C^0(\p\O)}\Big)^{-1}.
\end{equation}

Let us fix $\lam>0$ and $\mB_T\in\mB^{2+\a}(\p\O,\mathbb
R^3)$ which is not identically zero. For $0<\mu<\mu^*(\lam,\mB_T)$, let $\mA_\mu$ denote the unique solution of
(\ref{eqA-Omega}$\mu$) satisfying $\|\mathcal
A_\mu\|_{C^0(\overline{\O})}<1/\sqrt{3}$.
For each small $\var>0$ we define
\begin{equation}\label{mu-var}
\mu(\lam,\var)=\min\,\Big\{\mu>0:~\|\mA_\mu\|_{C^0(\overline{\O})}^2\geq {1\over
3}-\var\Big\}.
\end{equation}
From \cite[Lemma 7.1]{BaP}, the function $\mu\mapsto \|\mathcal
\A_\mu\|_{C^0(\overline{\O})}$ is continuous, so $\mu(\lam,\var)$ is achieved
for small $\var>0$:
$$\|\mathcal
A_{\mu(\var)}\|_{C^0(\overline{\O})}^2={1\over 3}-\var.
$$
Using \eqref{limit-1/3} we can show that
\begin{equation}\label{lim-mu-var}
\lim_{\var\to
0^+}\mu(\lam,\var)=\mu^*(\lam,\mB_T).
\end{equation}

Similarly for the fixed $\lam$ and $\mB_T$ as above, let
$(f_{\k,\mu},\A_{\k,\mu})$ denote the unique solution of
(\ref{eq4.7}$\mu$) in $\mathbb K(\O)$, and define
$$
\mu_\k(\lam,\var)=\min\,\Big\{\mu>0:~\min_{x\in\overline{\O}}\;\big[|f_{\k,\mu}(x)|^2-|\A_{\k,\mu}(x)|^2\big]
\leq {1\over 3}+2\var\Big\}.
$$

\begin{Lem}
Assume the conditions of Proposition \ref{Prop4.6} and $\lam>0$ is fixed. For small $\var>0$ we
have
\begin{equation}\label{lim-mu-k-var}
\liminf_{\k\to\infty}\mu_\k(\lam,\var)\geq \mu(\lam,\var),
\end{equation}
and
\begin{equation}\label{lim-lam-k}
\liminf_{\k\to\infty}\mu^*(\lam,\k,\mB_T)\geq\mu^*(\lam,\mB_T).
\end{equation}
\end{Lem}
\begin{proof} From Proposition \ref{Prop4.6}, for every $\k>1$, $\mu_\k(\lam,\var)$ is
achieved, namely, for $\mu=\mu_\k(\lam,\var)$, (\ref{eq4.7}$\mu$) has a solution
$$(f_{\k,\mu_\k(\lam,\var)},\A_{\k,\mu_\k(\lam,\var)})\in\mathbb K(\O),
$$
which
is denoted for simplicity by $(f_\k,\A_\k)$, such that
\begin{equation}\label{eq4.43}
\min_{x\in\overline{\O}}\;[f_\k^2(x)-|\A_\k(x)|^2]= {1\over 3}+2\var.
\end{equation}
Let us choose $x_\k\in\overline{\O}$ such that
$$
f_\k^2(x_\k)-|\A_\k(x_\k)|^2=
\min_{x\in\overline{\O}}\;[f_\k^2(x)-|\A_\k(x)|^2].
$$
From \eqref{ineq-mu-star}, for fixed $\mB_T$,
$\mu^*(\lam,\k,\mB_T)$ is bounded as $\k\to\infty$. Hence $\mu_\k(\lam,\var)$ is bounded as
$\k\to\infty$. After passing to a subsequence we may assume that
$x_\k\to x_0\in\overline{\O}$ and $\mu_\k(\lam,\var)\to\mu_0$ as $\k\to\infty$, where $\mu_0$ depends on $\lam$. Since
$(f_\k,\A_\k)\in\mathbb K_{2\var}(\O)$, from
Proposition \ref{Prop4.7}~(ii) we know that $(f_\k,\A_\k)$ converges to
$(f_\infty,\A_\infty)$ uniformly on $\overline{\O}$ as $\k\to\infty$, where
$\A_\infty$ is a
solution of (\ref{eqA-Omega}$\mu$) for $\mu=\mu_0$, that is,
$\A_\infty=\mA_{\mu_0}$, and $f_\infty(x)=(1-|\mA_{\mu_0}(x)|^2)^{1/2}$.
Therefore
$$
\lim_{\k\to\infty}\; [f_\k^2(x_\k)-|\A_\k(x_\k)|^2]=
f_\infty^2(x_0)-|\mA_{\mu_0}(x_0)|^2=1-2|\mA_{\mu_0}(x_0)|^2.
$$
From this and \eqref{eq4.43} we have
$1-2|\mA_{\mu_0}(x_0)|^2=1/3+2\var$, so
$$\|\mA_{\mu_0}\|_{C^0(\overline{\O})}^2\geq |\mA_{\mu_0}(x_0)|^2= {1\over
3}-\var.
$$
Hence $\mu(\lam,\var)\leq \mu_0$, so
$\liminf_{\k\to\infty}\mu_\k(\lam,\var)=\mu_0\geq\mu(\lam,\var).$ Therefore
\eqref{lim-mu-k-var} is true. From \eqref{lim-mu-k-var} we have
$$\liminf_{\k\to\infty}\mu^*(\lam,\k,\mB_T)\geq
\liminf_{\k\to\infty}\mu_\k(\lam,\var)\geq\mu(\lam,\var).
$$
Letting $\var$ go to $0$ and using \eqref{lim-mu-var} we get \eqref{lim-lam-k}.
\end{proof}

\begin{Thm}\label{Thm4.9}
Let $\O$ be a bounded and simply-connected
domain in $\mathbb R^3$ without holes and with a $C^{3+\a}$ boundary, $0<\a<1$. Let $\mB_T\in\mB^{2+\a}(\p\O)$ and
\begin{equation}\label{ineq4.44}
\|\mB_T\|_{C^0(\p\O)}<\sqrt{5\over
18}.
\end{equation}
There exist $\lam_{f\A}(\O,\mB_T)>0$, and $\k_{f\A}(\O,\mB_T,\lam)>0$ for all $0<\lam< \lam_{f\A}(\O,\mB_T)$, such that the following conclusions are true:
\begin{itemize}
\item[(i)]  For any
$\k>\k_{f\A}(\O,\mB_T,\lam)$, \eqref{dec-eqfA-1} has a classical Meissner solution
$(f_\k,\A_\k)$ which is unique in $\mathbb K(\O)$. Let $\H_\k=\lam\,\curl\A_\k$, then
$(f_\k,\H_\k)\in \mathbb U(\O)$ and it is a Meissner solution of
\eqref{eq4.7}.
\item[(ii)]  As $\k\to\infty$, $(f_\k,\A_\k)$ uniformly converges to
$(f_\infty,\A_\infty)$ on $\overline{\O}$, where $\A_\infty$ is a solution
of \eqref{eqA-Omega} with $\mH_T=\mB_T$, and $f_\infty(x)=(1-|\A_\infty(x)|^2)^{1/2}$.
\end{itemize}
\end{Thm}

\begin{proof} From \eqref{limit-mu-star} and \eqref{ineq4.44}, there exist $\eta>0$ and
$\lam_{f\A}=\lam_{f\A}(\O,\mB_T)>0$ such that $\mu^*(\lam,\mB_T)>1+\eta$ for all $0<\lam< \lam_{f\A}$. From \eqref{lim-lam-k}, we can find
$\k_{f\A}(\lam)=\k_{f\A}(\O,\mB_T,\lam)>0$ such that, for all $0<\lam<\lam_{f\A}$ and
 $\k>\k_{f\A}(\lam)$ we have $\mu^*(\lam,\k,\mB_T)>1+\eta$, hence \eqref{dec-eqfA-1} has a Meissner solution. So (i) is true.

To prove (ii), note that from \eqref{lim-mu-var}, for any $0<\lam< \lam_{f\A}(\O,\mB_T)$, there exits $\var_0=\var_0(\lam)>0$
such that $\mu(\lam,\var)>1+\eta$ for all $0<\var\leq \var_0$. Then from
\eqref{lim-mu-k-var} we can find $\k(\lam,\var_0)> \k_{f\A}(\O,\mB_T,\lam)$ such that
$\mu_\k(\lam,\var_0)>1$ if $\k>\k(\lam,\var_0)$.
Hence for each $\k>\k(\lam,\var_0)$, the solution $(f_\k,\A_\k)$ of
\eqref{dec-eqfA-1} satisfies
$$\min_{x\in\overline{\O}}\;[f_\k^2(x)-|\A_\k(x)|^2]> {1\over
3}+\var_0.
$$
So $(f_\k,\A_\k)\in\mathbb K_{2\var_0}(\O)$ for all
$\k>\k(\lam,\var_0)$. It follows from Proposition \ref{Prop4.7} (ii) that $(f_\k,\A_\k)$
uniformly converges to $(f_\infty,\A_\infty)$ as $k\to\infty$.
\end{proof}

\begin{Rem}
Proposition \ref{Prop4.7} and Theorem \ref{Thm4.9} give the convergence in $H^1$ and in $C^0$ as $\k\to\infty$ of the magnetic potential part of the solutions of BVP \eqref{dec-eqfA-1} to  a solution of \eqref{eqA-Omega} when the boundary datum $\mH_T$ is given. These results  imply that the magnetic potential part of the Meissner solutions of \eqref{eqfA}-\eqref{1.4} with fixed tangential component of $\curl\A$ converge to a  solution of \eqref{eqA}-\eqref{1.4} as $\k\to\infty$.  Proposition \ref{Prop4.7} also  suggests sub-convergence of the Meissner solutions of \eqref{eqfA}-\eqref{1.4} with the tangential component of $\curl\A$ being uniformly bounded in $C^{2+\a}(\p\O,\Bbb R^3)$.
\end{Rem}

\section{The Exterior Problem}\label{Section5}

In this section we study exterior problem \eqref{dec-eqfA-2}.
Denote
$$\aligned
&H^1_{0,\loc}(\O^c)=\{u\in H^1_{\loc}(\O^c):~ u=0\text{ \rm on } \p\O\},\\
&C^{k+\a}_{\loc}(\O^c)=\{u:~ u\in C^{k+\a}(\overline{B})\text{ for any ball
$B\Subset \O^c$}\},\\
&C^{k+\a}_{\loc}(\overline{\O^c})=\{u:~ u\in C^{k+\a}(\overline{B\cap \O^c})\text{ for any ball
$B\subset\Bbb R^3$}\}.
\endaligned
$$
Similarly we define $H^k_{\loc}(\O^c,\mathbb R^3)$,
$C^{k+\a}_{\loc}(\O^c,\mathbb R^3)$, $C^{k+\a}_{\loc}(\overline{\O^c},\mathbb R^3)$.

%\subsection{Classification of solutions of \eqref{dec-eqfA-2}}\

Assume $\O$ is a bounded domain in
$\mathbb R^3$ with a $C^2$ boundary,  and $\mB_T\in
T\!H^{1/2}(\p\O,\mathbb R^3)$. Let $\A$ be a weak solution of
\eqref{dec-eqfA-2}-\eqref{1.4} and set $\H=\lam\,\curl\A$. Then $\H\in H^1_{\loc}(\O^c,\mathbb R^3)\cap
C^{\infty}_{\loc}(\O^c,\mathbb R^3)$, and it satisfies the following
\begin{equation}\label{Eq5.1}
\left\{\aligned
&\curl\H=\0\q\text{\rm and}\q\div\H=0\q\text{\rm in }\O^c,\\
&\H_T^+=\mB_T\q\text{\rm on }\p\O,\q
\H-\mH^e\to\0\q\text{\rm as }|x|\to\infty.\endaligned\right.
\end{equation}
However, even if the boundary datum $\mA_T\in
H^{3/2}(\p\O,\mathbb R^3)$, in general the solution $\A$ of \eqref{dec-eqfA-2} does not belong to
$H^1_{\loc}(\O^c,\mathbb R^3)$. To see this, assume $\A_0\in
H^1_{\loc}(\O^c,\mathbb R^3)$ is a weak solution of
\eqref{dec-eqfA-2}. Let $\psi\in H^1_{0,\loc}(\O^c)\setminus
H^2_{\loc}(\O^c)$ which vanishes near $\p\O$, and set $\A=\A_0+\nabla\psi.$ Then $\A$
is also a weak solution of \eqref{dec-eqfA-2}, but $\A\not\in
H^1_{\loc}(\O^c,\mathbb R^3)$.

\subsection{Existence and classification of solutions of \eqref{Eq5.1}}\

\begin{Lem}\label{Lem5.1} Assume $\O$ satisfies $(O)$, $0<\a<1$, $\mH^e\in C^{1+\a}_{\loc}(\overline{\O^c},\curl0,\div0)$, $\mB_T\in T\!C^{1+\a}(\p\O,\mathbb R^3)$ satisfies
\begin{equation}\label{cond-5.2}
\nu\cdot\curl(\mB_T-(\mH^e)^+_T)=0\q\text{\rm on}\;\;\p\O,
\end{equation}
where $\nu$ is the unit normal vector of $\p\O$ pointing into $\O^c$.
Then \eqref{Eq5.1} has a one-parameter family of solutions $\H_\mu\in
C^{1+\a}_{\loc}(\overline{\O^c},\mathbb R^3)$ with $\mu\in\Bbb R$, and they have the form
\begin{equation}\label{eqH-mu}
\H_\mu=\mathcal
H^e+\nabla\phi_\mu,
\end{equation}
where $\phi_\mu$ satisfies
\begin{equation}\label{eqphi-mu}
\left\{\aligned
&\Delta\phi=0\q\text{\rm in }\O^c,\qq (\nabla\phi)_T=\mB_T-(\mH^e)^+_T\q\text{\rm on
}\p\O,\\
&\int_{\p\O}{\p\phi\over\p\nu} dS=\mu,\qq \phi(x)=O(|x|^{-1})\q\text{\rm as }|x|\to\infty.
\endaligned\right.
\end{equation}
$\H_\mu$'s are the only solutions of \eqref{Eq5.1}.
\end{Lem}

\begin{proof} If $\H$ is a solution of \eqref{Eq5.1} and if we let
$\w=\H-\mH^e$, then \eqref{Eq5.1} is transformed to
\begin{equation}\label{eq5.5}
\left\{\aligned
&\curl\w=\0\q\text{and}\q\div\w=0\q\text{\rm in }\O^c,\\
&\w_T^+=\mB_T-(\mH^e)^+_T\q\text{\rm on }\p\O,\qq
\w\to\0\q\text{\rm as }|x|\to\infty.\endaligned\right.
\end{equation}
Since $\O^c$ is simply-connected, we can write $\w=\nabla\phi$. Then $\phi$ is a solution of
\eq\label{eq5.6}
\left\{
\aligned &\Delta\phi=0\q\text{in }\O^c,\qq
(\nabla\phi)_T=\mB_T-(\mathcal
H^e)^+_T\q\text{\rm on}\;\;\p\O,\\
& \lim_{|x|\to\infty}\nabla\phi(x)=\0.
\endaligned\right.
\eeq
From \cite[Lemma 2.7]{NW} we know that, under condition $(O)$, for any $\mu\in\Bbb R$, \eqref{eq5.6} has at most one solution satisfying the following condition
\eq\label{eq5.7}
\int_{\p\O}{\p\phi\over\p\nu} dS=\mu.
\eeq
On the other hand, from \cite[Corollary 2.1, Lemma 2.6]{NW} we know that, under the conditions on $\O$, $\mH^e$ and $\mB_T$ mentioned in the lemma, problem \eqref{eq5.6}-\eqref{eq5.7} has a solutions $\phi_\mu\in C^{2+\a}_{\loc}(\overline{\O^c})$.
So all the solutions of \eqref{eq5.5} are in the form $\w_\mu=\nabla\phi_\mu$, $\mu\in\Bbb R$, and hence all the solutions of \eqref{Eq5.1} are given by \eqref{eqH-mu}.
Furthermore from \cite[Lemma 2.6]{NW} we know that $\phi_\mu(x)=O(|x|^{-1})$ as $|x|\to\infty$.  From this and
using \cite[Lemma 2.2]{NW} we conclude that
$|\w_\mu(x)|=|\nabla\phi_\mu(x)|=O(|x|^{-2})$ as $|x|\to\infty$.
\end{proof}

Condition \eqref{cond-5.2} is necessary for \eqref{Eq5.1} to have a solution. In fact, if
\eqref{Eq5.1} has a solution $\H$, then
$$0=(\nu\cdot\curl(\H-\mH^e))^+=(\nu\cdot\curl(\H-\mH^e)_T)^+=\nu\cdot\curl(\mB_T-(\mH^e)^+_T).
$$

\begin{Rem}\label{Rem5.2}
Let $\phi_\mu$ be the solution of \eqref{eqphi-mu}. For any simple, closed and oriented surface $S\subset\O^c$ which encloses $\O$, it holds that
$$\int_{S}\nu_S\cdot\nabla\phi_\mu dS=\mu,
$$
where $\nu_S$ is the unit outer normal to $S$. Hence $\nabla\phi_\mu$ has zero flux only when $\mu=0$. Since
$$\int_{\p\O}{\p\phi_0\over\p\nu}dS=0,
$$
from the asymptotic behavior of harmonic functions (see for instance \cite[p.391-392, Proposition 17]{DaL1}) we know that
$$\phi_0(x)=O(|x|^{-2}),\q\text{\rm and}\q
|\nabla\phi_0(x)|=O(|x|^{-3})\q\text{\rm as }|x|\to\infty.
$$
\end{Rem}

\vskip0.1in

\subsection{Existence of solutions of \eqref{dec-eqfA-2}}\

If there exists $\mF^e\in C^{1+\a}_{\loc}(\overline{\O^c},\div0)$ such that $\curl\mF^e=\mH^e$ in $\O^c$, and if we let $\u=\lam\,\A-\mF^e$, then \eqref{dec-eqfA-2} is transformed to
\eq\label{eq5.8}
\curl\u=\nabla\phi_\mu\q\text{in }\O^c,\q \u_T^+=\lam\mA_T-(\mF^e)^+_T\q\text{\rm on}\;\;\p\O,
\eeq
for some $\mu\in\Bbb R$, where $\phi_\mu$ is given in \eqref{eqH-mu}.
If \eqref{eq5.8} has a solution $\u$, then $\nabla\phi_\mu\in \curl[H^1_{\loc}(\O^c,\Bbb R^3)]$, hence $\nabla\phi_\mu$ has zero flux in $\O^c$, and  from Remark \ref{Rem5.2} we have $\mu=0$. On the other hand, the following theorem shows that when $\mu=0$, \eqref{eq5.8} is solvable if $\mA_T$ satisfies
\eqref{cond-5.9}.

\begin{Thm}\label{Thm5.3}
Assume $\O, \mH^e, \mB_T$ satisfy the conditions in Lemma \ref{Lem5.1} with $0<\a<1$, assume $\mA_T\in T\!C^{1+\a}(\p\O,\mathbb R^3)$ satisfies
\begin{equation}\label{cond-5.9}
\lam\,\nu\cdot\curl\mathcal
A_T= (\nu\cdot\mH^e)^++{\p\phi_0\over\p\nu}\q\text{\rm on }\p\O,
\end{equation}
where $\phi_0$ is the solution of \eqref{eqphi-mu} with $\mu=0$ for the given $\mB_T$ and $\mH^e$, and assume there exists $\mF^e\in C^{1+\a}_{\loc}(\overline{\O^c},\div0)$ such that $\curl\mF^e=\mH^e$ in $\O^c$. Then we have the following conclusions:
\begin{itemize}
\item[(i)] Problem \eqref{dec-eqfA-2}-\eqref{1.4} has a weak solution $\A_0\in
C^1_{\loc}(\overline{\O^c},\mathbb R^3)$, which can be represented as
$\A_0=\lam^{-1}(\mF^e+\u_0),$
where $\u_0$ is the unique solution of
\begin{equation}\label{eq5.10}
\left\{\aligned
&\curl\u_0=\nabla\phi_0\q\text{\rm and}\q \div\u=0 \q\text{\rm in }\O^c,\\
&\u_{0,T}^+=\lam\mA_T-(\mF^e)^+_T\qqq\q\;\;\,\text{\rm on }\p\O,\\
&\int_{\p\O}\u_0\cdot\nu dS=0,\q \u_0(x)\to\0\q\text{\rm as }|x|\to\infty.
\endaligned\right.
\end{equation}
Moreover $\u_0$ has the decay rate
\begin{equation}
|\u_0(x)|=O\Big({\log|x|\over |x|^2}\Big)\q\text{\rm as }|x|\to\infty. \label{eq5.11}
\end{equation}
\item[(ii)] The general solution of \eqref{dec-eqfA-2}-\eqref{1.4} can be written
as
\begin{equation}\label{eq5.12}
\A=\lam^{-1}(\mF^e+\u_0)+\nabla\psi,\q\forall \psi\in
H^1_{\loc}(\O^c)\;\;\text{\rm satisfying } \nabla\psi=\0\;\;\text{\rm on }\p\O.
\end{equation}
\item[(iii)]   Problem \eqref{dec-eqfA-2}-\eqref{1.4} with $\mA_T$ replaced by $\lam\mF^e_T$ and $\mB_T$ replaced by $\mH^e_T$ is solvable, and all the solutions can be written as
\begin{equation}\label{eq5.13}
\A=\lam^{-1}\mF^e+\nabla\psi,\q\forall \psi\in
H^1_{\loc}(\O^c)\;\;\text{\rm satisfying } \nabla\psi=\0\;\;\text{\rm on }\p\O.
\end{equation}
\end{itemize}
\end{Thm}

\begin{proof} Let $\H_0=\mH^e+\nabla\phi_0$ be the solution of \eqref{Eq5.1} given in \eqref{eqH-mu} with $\mu=0$. Consider the following equation
\eq\label{eq5.14}
\lam\,\curl\A=\H_0\q\text{in }\O^c,\q \A_T^+=\mA_T\q\text{\rm on}\;\;\p\O.
\eeq
If \eqref{eq5.14} has a solution $\A\in C^{1+\a}_{\loc}(\overline{\O^c},\Bbb R^3)$, then
 $\A$ is a solution of \eqref{dec-eqfA-2}-\eqref{1.4}.
To solve \eqref{eq5.14}, we let $\u=\lam\,\A-\mF^e$ and transfer \eqref{eq5.14} to
 \eq\label{eq5.15}
\curl\u=\nabla\phi_0\q\text{in }\O^c,\q \u_T^+=\lam\,\mA_T-(\mF^e)^+_T\q\text{\rm on}\;\;\p\O.
\eeq
Note that any solution $\u_0$ of \eqref{eq5.10} is a solution of \eqref{eq5.15}, hence we only need to examine solvability of \eqref{eq5.10}.
Since $\nabla\phi_0$  has zero flux in $\O^c$, has decay rate given in Remark \ref{Rem5.2}, and satisfies  \eqref{cond-5.9},
using \cite[Theorem 3.3, Lemma 3.5]{NW} and Remarks (1) and (2) on p.1368 in \cite{NW}, we see that
\eqref{eq5.10} has a unique solution $\u_0\in C^{1+\a}(\overline{\O^c},\Bbb R^3)$, and it has the decay rate given in \eqref{eq5.11}.
Then $\A_0=\lam^{-1}(\mF^e+\u_0)$ is a solution of \eqref{dec-eqfA-2}-\eqref{1.4}.

If $\A$ is another solution, then $\curl(\A-\A_0)=\0$, so $\A-\A_0=\nabla\psi$ for some function $\psi$, because $\O$ is simply-connected, so is $\O^c$.  Hence the general solution of \eqref{dec-eqfA-2}-\eqref{1.4} is given by \eqref{eq5.12}.

When $\mB_T=\mH^e_T$, the only solution of \eqref{eqphi-mu} is $\phi_0=0$. So $\mA_T=\lam\mF^e_T$ satisfies \eqref{cond-5.9}. By the uniqueness of solutions to \eqref{eq5.10} we have $\u_0=\0$. So we get \eqref{eq5.13} from \eqref{eq5.12}.
\end{proof}

Note that condition \eqref{cond-5.9} is necessary for
\eqref{dec-eqfA-2}-\eqref{1.4} to have a solution. In fact if
\eqref{dec-eqfA-2}-\eqref{1.4} has a solution, then \eqref{eq5.10} has
a solution $\u$, hence on $\p\O$ we have
$$
{\p\phi_0\over\p\nu}=(\nu\cdot\curl\u)^+ =(\nu\cdot\curl\u_T)^+
=(\nu\cdot(\lam\,\curl\mA_T-\mH^e))^+,
$$
which gives \eqref{cond-5.9}.

We may view \eqref{cond-5.9} as a requirement on $\mA_T$, $\mH^e$ and $\mB_T$ (through $\phi_0$). Condition \eqref{cond-5.9} implies that $\mH^e$ satisfies \eqref{cond-vanish}, see Proposition \ref{PropE.2}. On the other hand, if $\mB_T$ and $\mH^e$ are given and $\mH^e$ satisfies \eqref{cond-vanish}, then there exists $\mA_T$ that satisfies \eqref{cond-5.9}, see Proposition \ref{PropE.2} in Appendix \ref{AppendixE}.

\section{The Limiting System}\label{Section6}

In this section we examine existence and classification of classical solutions of the limiting problem \eqref{eqA}.
Equivalence of \eqref{eqA}-\eqref{1.4}-\eqref{cond-1.8} with \eqref{eqH}-\eqref{cond-1.12}-\eqref{cond-1.10} in the sense of classical solutions has been discussed in \cite[Lemma 3.3]{P3}. So we start with discussions on \eqref{eqH}. Recall that for a ``classical" solution $\H$ of \eqref{eqH} we only require its tangential component to be continuous across $\p\O$, namely $\H_T^+=\H_T^-$£¬ see \cite[Definition 3.1]{P3}. If in addition the normal component of $\H$ is also continuous, so $\H$ is continuous across $\p\O$, then there exists a solution $\A$ of \eqref{eqA} such that $\H=\lam\,\curl\A$, see \cite[Lemma 3.3]{P3}.

\subsection{Existence and classification of solutions of \eqref{eqH}-\eqref{cond-1.12}}\

\begin{Lem}\label{Lem6.1}
 Assume $\O$ is a bounded domain in $\Bbb R^3$ with a $C^{3+\a}$ boundary, $0<\a<1$, and $\mH^e$ satisfies $(H_0)$. Let $\H$ be a solution of \eqref{eqH}-\eqref{cond-1.12}-\eqref{cond-1.10}, and assume
\eq\label{cond-H-6.1}
\H\in
\C^{1+\a,0}_t(\overline{\O},\overline{\O^c},\mathbb R^3)\cap C^{2+\a}(\overline{\O},\Bbb R^3).
\eeq
Set
$$ \H_\O=\H|_{\overline{\O}},\qq
\mH=\H_{\overline{\O^c}}.
$$
Then we have the following conclusions:
\begin{itemize}
\item[(i)]  $\H_\O\in C^{2+\a}(\overline{\O},\mathbb R^3)$ and it satisfies the equation
\eq\label{eq6.2}
-\lam^2\curl \bigl[F(\lam^2|\curl \H_\O|^2)\curl \H_\O\bigr]=\H_\O\q\text{\rm in }\O;
\eeq
$\mH\in C^{1+\a}_{\loc}(\overline{\O^c},\curl0,\div0)$,
$\lim_{|x|\to\infty}(\mH-\mH^e)=\0$; and
\eq\label{bdry6.3}
(\H_\O)^-_T=\mH_T^+\q\text{\rm and}\q \nu\cdot\curl [(\H_\O)_T^-]=\nu\cdot\curl(\mH_T^+)=0\q\text{\rm on }\p\O.
\eeq
\item[(ii)] If furthermore $\O^c$ is simply-connected, then
\eq\label{eq6.4}
\H=\mH^e+\nabla\phi\q\text{\rm in }\O^c,
\end{equation}
where $\phi\in C^{2+\a}_{\loc}(\overline{\O^c})$ and it is determined by
\begin{equation}\label{eq6.5}
\left\{\aligned
&\Delta\phi=0\qqq\qqq\;\;\text{\rm in }\O^c,\\
&(\nabla\phi)_T^+=(\H_\O)_T^- -(\mathcal
H^e)_T^+\q\text{\rm on }\p\O,\\
&\lim_{|x|\to\infty}\nabla\phi=\0.
\endaligned\right.
\end{equation}
\end{itemize}
\end{Lem}
\begin{proof} From \eqref{cond-H-6.1} we get the first equality in \eqref{bdry6.3}, which yields
$$\nu\cdot\curl [(\H_\O)_T^- - \mH_T^+]=0\q\text{on }\p\O.
$$
If $\O$ hence $\O^c$ is simply-connected, since $\curl\mH=\0$ in $\overline{\O^c}$, we get the second equality in \eqref{bdry6.3}. Part (ii) is cited from \cite[Lemma 3.5]{P3}.
\end{proof}

{\bf Remark (i)}. Now we describe an observation in \cite[Lemma 3.6]{P3}. Assume the conditions of Lemma \ref{Lem6.1} and assume $\O$ is simply-connected. For any $\H_\O\in C^{1+\a}(\overline{\O},\mathbb R^3)$, Eq. \eqref{eq6.5} is solvable if and only if
$$(\nu\cdot\curl(\H_\O-\mH^e)_T)^-=0\q\text{on }\p\O,
$$
see \cite[Lemmas 2.5, 2.6]{NW}.
From $(H_0)$ we have $\nu\cdot\curl\mH^e_T=0$ on $\p\O$, hence the solvability condition reads
\eq\label{cond-6.6}
(\nu\cdot\curl(\H_\O)_T)^-=0\q\text{\rm on }\p\O.
\eeq
If $\H_\O$ satisfies \eqref{cond-6.6}, then the solutions of \eqref{eq6.5} form a one-parameter family $\{\phi^\tau: \tau\in \Bbb R\}$, where $\phi^\tau$ satisfies
$$
\int_{\p\O}{\p\phi^\tau\over\p\nu}dS=\tau.
$$
$\phi^\tau$ is uniquely determined by $\tau$ and $(\H_\O-\mathcal
H^e)_T^-$. Note that $\nabla\phi^\tau$ has zero flux if and only if $\tau=0$. Thus we can verify that (see Remark \ref{Rem5.2}) $\phi^0$ satisfies
\eq\label{eq6.7}
\left\{\aligned
&\Delta\phi^0=0\q\text{\rm in }\O^c,\q (\nabla\phi^0)_T^+=(\H_\O-\mathcal
H^e)_T^-\q\text{\rm on }\p\O,\\
&\int_{\p\O}{\p\phi^0\over\p\nu}dS=0,\q
 \phi^0(x)=O(|x|^{-2}),\q |\nabla\phi^0(x)|=O(|x|^{-3})\q\text{as }
|x|\to\infty.
\endaligned\right.
\eeq
\qed

{\bf Remark (ii)}. Equations \eqref{eqfA}, \eqref{eqA} and \eqref{eqH} require the continuity of the tangential component of $\A$ and of $\H=\lam\,\curl\A$, but not of their normal component.

Continuity of normal component of a solution $\A$ of \eqref{eqA} can always be satisfied after modifying the value of $\A$ in $\O^c$ by adding a gradient
if $\A\in \C^{2+\a,0}_t(\overline{\O},\overline{\O^c},\Bbb R^3)$, which yields a new solution of \eqref{eqA} that is continuous across $\p\O$, see \cite[Lemma 3.2]{P3}.

However, continuity of normal component of a solution $\H$ of \eqref{eqH} can not be made up by adding a gradient in $\O^c$, because this continuity requires $\mH$  to satisfy an integral condition
\eq\label{cond-6.8}
\int_{\p\O} \nu\cdot\mH^+ dS=0.
\eeq
To see this, assume the conditions of Lemma \ref{Lem6.1} hold, and suppose the normal component of $\H$ is continuous on $\p\O$, so
$\H\in
\C^{1+\a,0}(\overline{\O},\overline{\O^c},\mathbb R^3)\cap C^{2+\a}(\overline{\O},\Bbb R^3)$ and
$\nu\cdot(\H_\O^- -\mH^+)=0$ on $\p\O$. This and  the divergence theorem gives
$$\int_{\p\O}\nu\cdot\mH^+ dS=\int_{\p\O}\nu\cdot\H_\O^- \,dS=\int_\O\div\H \,dx=0.
$$
\qed

{\bf Remark (iii)}. Assume  $\O, \mH^e$ and $\mF^e$ satisfy $(O), (H_0), (F)$ respectively, and let $\H$ be a solution of \eqref{eqH}. If there exists a solution $\A$ of \eqref{eqA} such that $\H=\lam\,\curl\A$, then $\mH=\H|_{\overline{\O^c}}$ must satisfy the stronger condition $(H)$, and the function $\phi$ in \eqref{eq6.4} must be equal to $\phi^0$.
Indeed, since $\mH=\lam\,\curl\A$ in $\O^c$, $\mH$ has zero flux in $\O^c$, hence \eqref{cond-6.8} holds. Since $\mH^e=\curl\mF^e$ also has zero flux in $\O^c$, so $\nabla\phi^\tau=\mH-\mH^e$ has zero flux, hence $\tau=0$.
\qed

Next we consider existence of solutions of \eqref{eqH}. We assume $\O$ and $\mH^e$ satisfy
\eq\label{cond-6.9}
\text{$\O$ satisfies $(O)$ with $r\geq 3$ and $0<\a<1$, \q $\mH^e$ satisfies $(H_0)$},
\eeq
and $\mH$ satisfies
\begin{equation}\label{cond-6.10}
 \aligned
&\mH\in C^{2+\a}_{\loc}(\overline{\O^c},\curl0,\div0),\qq
\|\mH_T^+\|_{C^0(\p\O)}< \sqrt{5\over 18},\\
&\mH-\mH^e\in C^{\a}(\overline{\O^c},\mathbb R^3),\qq
\lim_{|x|\to\infty}(\mH-\mH^e)=\0.
\endaligned
\end{equation}

\begin{Lem}[\cite{P3} Lemma 3.4]\label{Lem6.2} Assume $\O$ and $\mH^e$ satisfy \eqref{cond-6.9}, and $\mH$ satisfies \eqref{cond-6.10}.
Then there exists $\lam_{\H}(\O,\mH_T)>0$ such that for all $0<\lam<\lam_{\H}(\O,\mH_T)$,
\eqref{eqH} has a solution $\H$ satisfying \eqref{cond-1.10} and \eqref{cond-H-6.1},
and $\H=\mH$ in $\O^c$, hence $\H$ satisfies \eqref{cond-1.12}.
\end{Lem}
\begin{proof} Let $\mH$ satisfy \eqref{cond-6.10}. From \cite[Theorem 7.4]{BaP}, there exists $\lam_\H=\lam_\H(\O,\mH_T)>0$ such that,  for all $0<\lam<\lam_\H$, the following BVP
\begin{equation}\label{eq6.11}
\left\{\aligned
-&\lam^2\curl \bigl[F(\lam^2|\curl \H|^2)\curl \H\bigr]=\H\q&\text{ \rm in } \O,\\
&\H_T^-=\mH_T^+\qqq\qqq\q&\text{ \rm on }\p\O.
\endaligned\right.
\end{equation}
has a unique solution $\H^\lam\in C^{2+\a}(\overline{\O},\Bbb R^3)$ satisfying \eqref{cond-1.10}.
Define a vector field $\H$ on $\Bbb R^3$ by letting $\H=\H^\lam$ in $\overline{\O}$ and $\H=\mH$ in $\O^c$.
Then $[\H_T]=\0$ on $\p\O$, so $\H\in C^{2+\a,0}_t(\overline{\O},\overline{\O^c},\Bbb R^3)$, hence $\H$ satisfies \eqref{cond-H-6.1}. Thus $\H$ is a solution of \eqref{eqH}-\eqref{cond-1.10}-\eqref{cond-1.12}.
\end{proof}

\subsection{Classification of solutions of \eqref{eqA}-\eqref{1.4}}\

%We first classify all solutions to \eqref{eqA}-\eqref{1.4}.

\begin{Lem}\label{Lem6.3}  Assume $\O$ and $\mH^e$ satisfy \eqref{cond-6.9}, and $\mF^e$ satisfies $(F)$. Let $\A\in
\C^{2+\a,0}_t(\overline{\O},\overline{\O^c},\mathbb R^3)$ be a solution of \eqref{eqA}-\eqref{1.4}-\eqref{cond-1.8}, and assume $\H=\lam\,\curl\A\in \C^{1+\a,0}_t(\overline{\O},\overline{\O^c},\Bbb R^3)$.
Denote
$$\A_\O=\A|_{\overline{\O}},\qq
\mA=\A|_{\overline{\O^c}},\qq
\H_\O=(\lam\,\curl\A)\big|_{\O}.
$$
Then we have the following conclusions:
\begin{itemize}
\item[(i)] $\A_\O\in C^{2+\a}(\overline{\O},\mathbb R^3)$ and it is a solution of
$$
-\lam^2\curl^2\A_\O=(1-|\A_\O|^2)\A_\O\q\text{\rm in }\O.
$$
\item[(ii)] $\mA\in C^{2+\a}_{\loc}(\overline{\O^c},\Bbb R^3)$ and it can be
represented in $\O^c$ by $\mA=\lam^{-1}(\mF^e+\u)$,
where
$\u\in C^{1+\a}_{\loc}(\overline{\O^c},\Bbb R^3)$ and it is determined
by
\begin{equation}\label{eq6.12}
\left\{\aligned
&\curl\u=\nabla\phi^0\q\text{\rm in }\O^c,\\
&\u_T^+=\lam\,(\A_\O)^-_T -(\mathcal
F^e)_T^+ \q\text{\rm on }\p\O,
\endaligned\right.
\end{equation}
where $\phi^0$ is the unique solution of \eqref{eq6.7} associated with this $\H_\O$.
\item[(iii)] $\A_\O$, $\H_\O$ and $\phi^0$ satisfy
\eq\label{cond-6.13}
\nu\cdot\curl((\H_\O)_T^-)=0,\q
\lam(\nu\cdot\curl\A_\O)^-=\nu\cdot(\mH^e)^+ +{\p\phi^0\over\p\nu}\q\text{\rm on }\p\O.
\eeq
\end{itemize}
\end{Lem}

\begin{proof} From  Lemma \ref{Lem6.1},  $\H_\O$ satisfies the first equality in \eqref{cond-6.13}. Since $\O^c$ is simply connected,
$\H$ can be represented by \eqref{eq6.4} with $\phi$ being a solution of \eqref{eq6.5}. From the classification of solutions of \eqref{eq6.5} we know that $\phi=\phi^\tau$ for some $\tau\in\Bbb R$, see Remark (i) after Lemma \ref{Lem6.1}. From this and \eqref{eq6.4} we have
$$\lam\,\curl\mA=\mH^e+\nabla\phi^\tau\q\text{in }\O^c.
$$
From this and using $(F)$ we can write $\mA=\lam^{-1}(\mF^e+\u)$ in $\O^c$ for some vector field $\u$. It follows that $\nabla\phi^\tau=\curl\u$ in $\O^c$, thus $\nabla\phi^\tau$ has zero flux, which implies $\tau=0$, so $\phi^\tau=\phi^0$ is the solution of \eqref{eq6.7}. From this and \eqref{eqA} we see that $\u$ is a solution of \eqref{eq6.12}.
Solvability of \eqref{eq6.12} implies the second equality in \eqref{cond-6.13}, see \cite[Lemma 2.5]{NW}.
\end{proof}

\subsection{Existence of solutions to  \eqref{eqA}}\

\begin{Def}\label{Def6.4}
Assume $\O$ satisfies $(O)$ and $\mH$ satisfies \eqref{cond-6.10}. Let $\lam_\H(\O,\mH_T)$ be the number given in Lemma \ref{Lem6.2}. For $0<\lam< \lam_\H(\O,\mH_T)$, denote by
$\mS(\lam,\mH_T^+)$
the solution $\H_\O$ of \eqref{eq6.2}
under the boundary condition
$(\H_\O)_T^-=\mH_T^+$ on $\p\O$.
\end{Def}
The notation $\mS(\lam,\mH_T^+)$ reflects the fact that $\H_\O$ is determined by $\lam$ and $\mH_T^+$ only, if $\O$ is fixed.
Now we show that \eqref{eqA} has a solution $\A$ for all small $\lam$ if there exists a vector field $\mH$ satisfying the following conditions for some $\var_0>0$:
\begin{itemize}
\item[(i)]
$\mH$ satisfies \eqref{cond-6.10} and
\eq\label{cond-6.14}
\|\mH_T^+\|_{C^0(\p\O)}\leq \sqrt{5\over 18}-\var_0;
\eeq
\item[(ii)] the following comparability condition holds
\eq\label{cond-6.15}
\nu\cdot[\mS(\lam,\mH_T^+)]^-=\nu\cdot\mH^+\q\text{\rm on }\p\O;
\eeq
\item[(iii)] there exists a vector field $\mF$
such that
\begin{equation}\label{cond-6.16}
 \mF\in
C^{2+\a}_{\loc}(\overline{\O^c},\div0),\q \curl\mF=\mathcal
H\q\text{\rm  in }\O^c.
\end{equation}
\end{itemize}

Note that \eqref{cond-6.15} implies \eqref{cond-6.8}.  Another form of \eqref{cond-6.15} will be given in \eqref{eq6.33}.

\begin{Prop}\label{Prop6.5}
Assume $\O$ and $\mH^e$ satisfy \eqref{cond-6.9}. For any $\var_0>0$ small, there exists $\lam_\A(\O,\var_0)>0$, such that if  there exists a vector field $\mH$ on $\overline{\O^c}$ which satisfies \eqref{cond-6.14}, \eqref{cond-6.15} and \eqref{cond-6.16}, then for all $0<\lam<\lam_\A(\O,\var_0)$,  problem \eqref{eqA}-\eqref{cond-1.8}-\eqref{1.4} has a Meissner solution
$\A_\lam\in \C^{2+\a,0}(\overline{\O},\overline{\O^c},\mathbb R^3)$ such that
$$\lam\,\curl\A_\lam=\mH\q\text{\rm and} \q \lam\A_\lam=\mF+\nabla\psi^\lam\q\text{\rm in }\O^c,
$$
where $\psi^\lam\in C^{3+\a}_{\loc}(\overline{\O^c})$ satisfies
\eq\label{cond-6.17}
(\nabla\psi^\lam)_T=\lam\,(\A_\lam)_T^--\mF^+_T\q\text{\rm and}\q {\p\psi^\lam\over\p\nu}=-\nu\cdot\mF^+\q\text{\rm on }\p\O.
\eeq
Moreover $\curl\A_\lam\in \C^{2+\a,0}(\overline{\O},\overline{\O^c},\Bbb R^3)$.
\end{Prop}

\begin{proof} {\it Step 1}. Assume $\mH$ satisfies \eqref{cond-6.10} and $\mF$ satisfies \eqref{cond-6.16}. Consider the problem
\begin{equation}\label{eq6.18}
\left\{\aligned
-&\lam^2\curl^2\A=(1-|\A|^2)\A\;\;&\text{\rm in }\O,\\
&(\lam\,\curl\A)^-_T=\mH_T^+\qqq&\text{\rm on }\p\O.
\endaligned\right.
\end{equation}
From \cite[Theorem 7.4]{BaP} and the equivalence between \eqref{eq6.18} and \eqref{eq6.11} we know that, for the number $\lam_\H(\O,\mH_T)$ given in  Lemma \ref{Lem6.2}, if $0<\lam<\lam_\H(\O,\mH_T)$, then BVP \eqref{eq6.18}-\eqref{cond-1.8} has a unique solution $\A^\lam\in C^{2+\a}(\overline{\O},\Bbb R^3)$,
$\H^\lam=\lam\,\curl\A^\lam$ is a solution of \eqref{eq6.11}, and
$$\A^\lam=-\lam F(\lam^2|\curl\H^\lam|^2)\curl\H^\lam,\q x\in\O,
$$
where $F$ is the function appeared in \eqref{eqH}.
Hence
$$
(\nu\cdot\A^\lam)^-=-[\lam\, F(\lam^2|\curl\H^\lam|^2)(\nu\cdot\curl\mH_T)]^-=0\q\text{\rm on}\;\;\p\O,
$$
so $\A^\lam$ is a Meissner solution of \eqref{eq6.18}. From the regularity results of \eqref{eq6.11} in \cite[Theorem 5.1]{BaP} we see that $\H^\lam\in C^{2+\a}(\p\O,\Bbb R^3)$.
Moreover, from the discussions in \cite{BaP} we see that, for a given $\var_0>0$, $\lam_\H(\O,\mH_T)$ is uniform for all $\mH_T$ satisfying \eqref{cond-6.14}. So we can find a positive number $\lam_\A(\O,\var)$ such that
\eq\label{lam-A-var}
\lam_\H(\O,\mH_T)\geq \lam_\A(\O,\var_0)\q\text{for all $\mH$ satisfying \eqref{cond-6.14}}.
\eeq

{\it Step 2}.
Denote $\mA_T^\lam=(\A^\lam)^-_T$. We look for a solution of
\begin{equation}\label{eq6.20}
\left\{\aligned
&\curl^2\A^o={\bold 0}\qqq\qqq&\text{\rm in }\O^c,\\
&(\A^o)_T^+=\mA_T^\lam,\q\lam(\curl\A^o)^+_T=\mH_T^+\q& \text{\rm
on }\p\O,\\
&\lam\,\curl\A^o-\mH^e\to\0\q\text{as }|x|\to\infty,
\endaligned\right.
\end{equation}
see \eqref{dec-eqfA-2}-\eqref{1.4}. Note that $\mA_T^\lam,\; \mH_T^+\in T\!C^{2+\a}(\p\O,\Bbb R^3)$.

We first derive a necessary condition for solvability of \eqref{eq6.20}.
If \eqref{eq6.20} has a solution $\A^o$ and letting $\u=\lam\,\A^o-\mF$, from \eqref{cond-6.16} and \eqref{eq6.20} we see that $\u$ is a solution of
\begin{equation}\label{eq6.21}
\left\{\aligned
&\curl^2\u={\bold 0}\qqq\qqq\qq\q&\text{ \rm in }\O^c,\\
&\u_T^+=\lam\,\mA_T^\lam-\mF^+_T,\q(\curl\u)^+_T=\0\q\;& \text{\rm
on }\p\O,\\
&\curl\u\to\0\q\text{as }|x|\to\infty.
\endaligned\right.
\end{equation}
If \eqref{eq6.21} has a solution $\u\in C^{2+\a}_{\loc}(\overline{\O^c},\Bbb R^3)$ and letting $\w=\curl\u$, then
\eq\label{eq6.22}
\left\{\aligned
&\curl\w={\bold 0}\q\text{and}\q\div\w=0\q&\text{ \rm in }\O^c,\\
&\w^+_T=\0\q&\text{\rm on }\p\O,\\
& \w\to\0\q&\text{as }|x|\to\infty.
\endaligned\right.
\eeq
Since $\curl\w=\0$ and $\O^c$ is simply-connected,  we can write $\w=\nabla \eta$ for some function $\eta$, and from \eqref{eq6.22} we find that
\eq\label{eq6.23}
\left\{\aligned
&\Delta\eta=0\q\text{ \rm in }\O^c,\q
(\nabla\eta)_T=\0\q\text{\rm
on }\p\O,\\
&\nabla\eta\to\0\q\text{as }|x|\to\infty.
\endaligned\right.
\eeq
The fact that $\nabla\eta=\w=\curl\u$ for some $\u$ implies that $\nabla\eta$ has zero flux, so
$$\int_{\p\O}{\p\eta\over\p\nu}dS=0.
$$
From this and \eqref{eq6.23}, and using \cite[Lemma 2.7]{NW}, we find
$\nabla\eta=\0$, so $\curl\u=\0$. This and \eqref{eq6.21} imply that there exists a function $\psi$ such that
$\u=\nabla\psi$ in $\O^c$, where
\eq\label{eq6.24}
\psi\in C^{3+\a}_{\loc}(\overline{\O^c}),\qq
(\nabla\psi)_T=\lam\,\mA_T^\lam-\mF^+_T\q\text{\rm on}\;\;\p\O.
\eeq
Since $\O^c$ is simply-connected, from \cite[Lemma 2.5]{NW} we see that \eqref{eq6.24} is solvable if and only if
$$\nu\cdot\curl(\lam\,\mA_T^\lam-\mF^+_T)=0\q\text{on }\p\O.
$$
By \eqref{cond-6.16} and \eqref{eq6.18} we have
$$
\nu\cdot\curl(\lam\,\mA_T^\lam-\mF^+_T)=\nu\cdot(\lam\,\curl(\A^\lam)^-_T-\curl\mF_T^+)
=\nu\cdot((\H^\lam)^- -\mH^+).
$$
Hence the above condition of solvability for  \eqref{eq6.24} can be written as
$$\nu\cdot[(\H^\lam)^--\mH^+]=0\q\text{on }\p\O,
$$
which is exactly \eqref{cond-6.15}. Thus we have proved that, if \eqref{eq6.20}, hence \eqref{eq6.21}, has a $C^{2+\a}_{\loc}$ solution, then \eqref{cond-6.15} holds.

{\it Step 3}. Now assume \eqref{cond-6.15} holds, and we show that  \eqref{eqA}-\eqref{cond-1.8}-\eqref{1.4} has a solution. By \eqref{cond-6.15} and from \cite[Lemma 2.5]{NW} we know that  \eqref{eq6.24} is solvable. Let us denote by $\psi^\lam$ a general solution of \eqref{eq6.24} in $C^{3+\a}_{\loc}(\overline{\O^c})$. Then the general solution of \eqref{eq6.20} in $C^{2+\a}_{\loc}(\overline{\O^c},\Bbb R^3)$ can be written as
$\A^o=\lam^{-1}(\mF+\nabla\psi^\lam)$.
We define $\A_\lam$ and $\H_\lam$ on $\Bbb R^3$ by letting
$$
\A_\lam=\begin{cases} \A^\lam&\text{in }\overline{\O},\\
\lam^{-1}(\mF+\nabla\psi^\lam)&\text{in }\O^c,\end{cases}\qq\q
\H_\lam=\begin{cases}\H^\lam&\text{in }\O,\\
\mH&\text{in }\O^c.
\end{cases}
$$
Then $\A_\lam\in \C^{2+\a,0}_t(\overline{\O},\overline{\O^c},\Bbb R^3)$ and it is a solution of \eqref{eqA}-\eqref{1.4}-\eqref{cond-1.8}.
Using the fact $(\A_\lam)_T^-=(\A_\lam)_T^+$ we can verify that $\curl\A_\lam$ exists in $\Bbb R^3$ and $\lam\,\curl\A_\lam=\H_\lam$. Then from \eqref{cond-6.15} we have $\H_\lam\in \C^{2+\a,0}(\overline{\O},\overline{\O^c},\Bbb R^3)$.

{\it Step 4}. Note that $\A_\lam$ and $\H_\lam$ constructed in step 3 depend on the choice of $\psi^\lam$.
Now we look for a function $\psi^\lam$ such that \eqref{eq6.24} holds and the associated $\A_\lam$ is a Meissner solution of \eqref{eqA}. For this purpose, let $\psi^\lam_0$ be a solution of \eqref{eq6.24} such that
\eq\label{eq6.25}
\left\{
\aligned
&\Delta\psi^\lam_0=0\q\text{in }\O^c,\qq(\nabla\psi^\lam_0)_T=\lam\,\mA_T^\lam-\mF^+_T\q\text{\rm on}\;\;\p\O,\\
&\int_{\p\O}{\p\psi^\lam_0\over\p\nu}dS=0,\q \psi^\lam_0\to 0\q\text{as }|x|\to\infty.
\endaligned\right.
\eeq
Since \eqref{cond-6.15} holds and $\lam\,\mA_T^\lam-\mF^+_T\in T\!C^{2+\a}(\p\O,\Bbb R^3)$,  from \cite[Corollary 2.1, Lemma 2.6]{NW} we see that \eqref{eq6.25} has a unique solution $\psi^\lam_0$. The integral condition in \eqref{eq6.25} implies that $\psi^\lam_0=O(|x|^{-2})$ and hence $|\nabla\psi^\lam_0(x)|=O(|x|^{-3})$ as $|x|\to\infty$.

Next, we choose $\xi^\lam\in C^{3+\a}_{\loc}(\overline{\O^c})$ such that
\eq\label{cond-6.26}
\xi^\lam=0\q\text{and}\q {\p\xi^\lam\over\p\nu}=-\nu\cdot\mF^+-{\p\psi_0^\lam\over\p\nu}\q\text{\rm on}\;\;\p\O.
\eeq
Existence of $\xi^\lam\in H^2(\O^c)$ satisfying \eqref{cond-6.26} is a consequence of the trace theorem for $H^2(\O^c)$, see for instance \cite[Theorem 7.53]{Ad}. We can actually find a function $\xi^\lam\in H^2(\O^c)$ so that its $H^2$-norm $\|\xi^\lam\|_{H^2(\O^c)}$ is the least among all $H^2$ functions satisfying \eqref{cond-6.26}, so $\xi^\lam$ satisfies a fourth order elliptic equation of constant coefficients. Then,  using the condition
$$\nu\cdot\mF^+ +{\p\psi^\lam_0\over\p\nu}\in C^{2+\a}(\p\O),
$$
and applying the Schauder estimate of elliptic equations,
we find that $\xi^\lam\in C^{3+\a}_{\loc}(\overline{\O^c})$.

Finally we fix $\psi^\lam=\psi^\lam_0+\xi^\lam$, and define $\A_\lam$ using this $\psi^\lam$ as in step 3. Recalling $(\nu\cdot\A^\lam)^-=0$ on $\p\O$, we see
that $\A_\lam\in \C^{2+\a,0}(\overline{\O},\overline{\O^c},\Bbb R^3)$ and it is a Meissner solution of \eqref{eqA}-\eqref{1.4}-\eqref{cond-1.8}.
\end{proof}

We mention that part of argument in the proof of  Proposition \ref{Prop6.5} has been used in the proof of Lemma 3.3 in \cite{P3}.

Now we examine all vector fields satisfying \eqref{cond-6.10}. We first mention that, if $\mH_1$ and $\mH_2$ satisfy \eqref{cond-6.8} and \eqref{cond-6.10}, and $\mH_{1,T}=\mH_{2,T}$ on $\p\O$, then $\mH_1\equiv \mH_2$ on $\O^c$.
To prove, Let $\w=\mH_2-\mH_1.$
Then $\w\in C^{2+\a}_{\loc}(\overline{\O^c},\Bbb R^3)\cap C^\a(\overline{\O^c},\Bbb R^3)$ and it satisfies
$$\left\{\aligned
&\curl\w=\0\q\text{and}\q\div\w=0\q\text{in }\O^c,\\
&\w_T=\0\q\text{on }\p\O,\q\lim_{|x|\to\infty}\w(x)=\0,\q
\int_{\p\O}\nu\cdot\w dS=0.
\endaligned\right.
$$
Since $\O$ is simply-connected and without holes, from \cite[Theorem 3.3 (b)]{NW} we know that
$\w=\0$.

\begin{Lem}\label{Lem6.6}  Assume $\O$ and $\mH^e$ satisfy \eqref{cond-6.9} with $0<\a<1$.
\begin{itemize}
\item[(i)] Every $\mH$ satisfying \eqref{cond-6.10} can be represented by
\eq\label{eq6.27}
\mH=\mH^e+\nabla\phi_{\bold v,\mu},\q x\in \O^c,
\eeq
where $\bold v\in \mB^{2+\a}(\p\O)$ satisfies
\begin{equation}\label{cond-6.28}
 \|\mH^e_T+\bold
v\|_{C^0(\p\O)}< \sqrt{5\over 18},
\end{equation}
$\mu\in \Bbb R$, and
$\phi_{\bold v,\mu}\in
C^{3+\a}_{\loc}(\overline{\O^c})\cap C^{1+\a}(\overline{\O^c})$ is a solution of
\begin{equation}\label{eq6.29}
\left\{\aligned
&\Delta\phi_{\bold v,\mu}=0\q\text{\rm in }\O^c,\qq
 (\nabla\phi_{\bold v,\mu})_T=\bold v\q\text{\rm on }\p\O,\\
& \int_{\p\O}{\p\phi_{\bold v,\mu}\over\p\nu}dS=\mu,\q \phi_{\bold v,\mu}(x)=O(|x|^{-1})\q\text{\rm as }|x|\to\infty.
\endaligned\right.
\end{equation}
\item[(ii)] Assume in addition $\mF^e$ satisfies $(F)$. Then any pair $\mH$ and $\mF$, which have the properties \eqref{cond-6.10} and \eqref{cond-6.16}, can be written as
\eq\label{eq6.30}
 \mH=\mH^e+\nabla\phi_{\bold v,0},\q \mF=\mF^e+\w_{\bold v}+\nabla g,
\eeq
where $\bold v\in \mB^{2+\a}(\p\O)$ satisfies \eqref{cond-6.28}, $\phi_{\bold v,0}$ is the solution of \eqref{eq6.29} for this $\bold v$ and for $\mu=0$, $\w_\bold v$ is the solution of
\eq\label{eq6.31}
\left\{\aligned
&\curl\w_{\bold v}=\nabla\phi_{\bold v,0}\q\text{\rm and}\q \div\w_{\bold v}=0\q&\text{\rm in }\O^c,\\
&\nu\cdot\w_{\bold v}=-\nu\cdot\mF^e\q&\text{\rm on }\p\O,\\
&\w_{\bold v}(x)\to \0\q&\text{\rm as }|x|\to\infty,
\endaligned\right.
\eeq
and $g\in C^{3+\a}_{\loc}(\overline{\O^c})$ is a harmonic function in $\O^c$.
\end{itemize}
\end{Lem}

\begin{proof} {\it Step 1}. Assume $\mH$ satisfies \eqref{cond-6.10}. Since $\O^c$ is simply-connected, we can write
$\mH=\mH^e+\nabla\phi$ for some $\phi\in C^{3+\a}_{\loc}(\overline{\O^c})\cap C^{1+\a}(\overline{\O^c})$.
Let $\bold
v=(\nabla\phi)_T$. Then $\bold v\in \mB^{2+\a}(\p\O)$ and
satisfies \eqref{cond-6.28}, and $\phi$ solves the following equation for
this $\bold v$:
$$
\left\{\aligned
&\Delta\phi=0\q\text{\rm in }\O^c,\q (\nabla\phi)_T=\bold v\q\text{\rm on}\;\;\p\O,\\ &\phi(x)=O(|x|^{-1})\q\text{as }|x|\to\infty.
\endaligned\right.
$$
From \cite[Corollary 2.1, Lemma 2.6]{NW}, for any real
number $\mu$, the above equation has a unique solution $\phi_{\bold v,\mu}\in
C^{3+\a}_{\loc}(\overline{\O^c})\cap C^{1+\a}(\overline{\O^c})$ such that
$$
\int_{\p\O}{\p\phi_{\bold v,\mu}\over\p\nu}dS=\mu,\q\text{and}\q
|\nabla\phi_{\bold v,\mu}(x)|=O(|x|^{-2})\q\text{as }
|x|\to\infty.
$$
Hence $\phi=\phi_{\bold v,\mu}$ for some $\mu$, and $\mH$ is represented by \eqref{eq6.27}.
On the other hand, any $\mH$ represented by \eqref{eq6.27} satisfies \eqref{cond-6.10}.

{\it Step 2}. Assume $(F)$ holds, and assume $\mH$ and $\mF$ satisfy \eqref{cond-6.10} and \eqref{cond-6.16}. Let
$\w=\mF-\mF^e$, where $\mF^e$ is given in $(F)$. From \eqref{eq6.27},
$\curl\w=\nabla\phi_{\bold v,\mu}$ in $\O^c$, so $\nabla \phi_{\bold v,\mu}$ has zero flux, hence $\mu=0$. Thus $\mH=\mH^e+\nabla\phi_{\bold v,0}$ and $\w$ satisfies
\eq\label{eq6.32}
\curl\w=\nabla\phi_{\bold v,0}\q\text{and}\q \div\w=0\q\text{in }\O^c.
\eeq
Since $\nabla\phi_{\bold v,0}$ has zero flux and $|\nabla\phi_{\bold v,0}(x)|=O(|x|^{-3})$ as $|x|\to\infty$,
we can apply \cite[Theorem 3.2]{NW} to conclude that there exists a unique $\w_{\bold v}\in C^{2+\a}_{\loc}(\overline{\O^c},\Bbb R^3)$ which satisfies \eqref{eq6.31}.\footnote{In fact $\w_{\bold v}$ has the asymptotic behavior $|\w_{\bold v}(x)|=O\Big({\log |x|\over|x|^2}\Big)$ as $|x|\to\infty$.}
$\phi_{\bold v,0}$ is uniquely determined by $\bold v$, so is $\w_{\bold v}$.
The general solution of \eqref{eq6.32} is $\w=\w_\bv+\nabla g$ where $g$ is any harmonic function in $\O^c$. So we get \eqref{eq6.30}.
\end{proof}

Combining Proposition \ref{Prop6.5} (in particular \eqref{cond-6.15}) and  Lemma \ref{Lem6.6} we get the following criterium for solvability of \eqref{eqA}.

\begin{Thm}\label{Thm6.7}  Assume $\O$ and $\mH^e$ satisfy \eqref{cond-6.9} with $0<\a<1$, and $\mF^e$ satisfies $(F)$.
\begin{itemize}
\item[(i)] Problem \eqref{eqA}-\eqref{1.4}-\eqref{cond-1.8} has a classical Meissner solution for small $\lam$ if and only if there exists a vector field $\bold v\in \mB^{2+\a}(\p\O)$ satisfying \eqref{cond-6.28} such that
\eq\label{eq6.33}
\nu\cdot\Big[\mS\Big(\lam, (\mH^e_T)^++\bold v\Big)\Big]^-=\nu\cdot(\mH^e)^+ +{\p\phi_{\bold v,0}\over\p\nu}\q\text{\rm on }
\p\O,
\eeq
where $\phi_{\bold v,0}$ is the solution of \eqref{eq6.29} for this $\bold v$ and for $\mu=0$.
\item[(ii)] For any $\var_0>0$ small, there exists $\lam_\A(\O,\var_0)>0$ such that, for all $0<\lam<\lam_\A(\O,\var_0)$, if there exists  $\bold v\in \mB^{2+\a}(\p\O)$ satisfying \eqref{eq6.33} and
\begin{equation}\label{cond-6.34}
 \|\mathcal H^e_T+\bold
v\|_{C^0(\p\O)}\leq \sqrt{5\over 18}-\var_0,
\end{equation}
then problem \eqref{eqA}-\eqref{1.4}-\eqref{cond-1.8} has a classical Meissner solution $\A_\lam$ such that
$\A_\lam\in\C^{2+\a,0}_t(\overline{\O},\overline{\O^c},\Bbb R^3)$,
$\curl\A_\lam\in \C^{2+\a,0}(\overline{\O},\overline{\O^c},\Bbb R^3)$,
and
$$\aligned
\lam\,\curl\A_\lam=&\begin{cases}
\mS\Big(\lam, (\mH^e_T)^++\bold v\Big)\q&\text{\rm in }\O,\\
\mH^e+\nabla\phi_{\bold v,0}\q&\text{\rm in }\O^c,
\end{cases}\\
\lam\,\A_\lam=&\mF^e+\w_{\bold v}+\nabla \psi^\lam\q\text{\rm in }\O^c,
\endaligned
$$
where $\w_{\bold v}$ is the solution of \eqref{eq6.31}, and $\psi^\lam\in C^{3+\a}_{\loc}(\overline{\O^c})$ is any
function satisfying \eqref{cond-6.17} for $\mF=\mF^e+\w_{\bold v}$. We can choose $\psi^\lam$ such that $\A_\lam\in C^{2+\a,0}(\overline{\O},\overline{\O^c},\Bbb R^3)$.
\end{itemize}
\end{Thm}
\begin{proof} For part (i), we only need to show that \eqref{eq6.33} is a necessary and sufficient condition for \eqref{eqA} to have a solution for small $\lam$. Let $\mH^e$ and $\mF^e$ satisfy \eqref{cond-6.10} and $(F)$. Then from Lemma \ref{Lem6.6} can write $\mH=\mH^e+\nabla\phi_{\bold v,0}$, where $\bold v\in \mB^{2+\a}(\p\O)$ satisfies \eqref{cond-6.34},  and $\phi_{\bold v,0}$ is a solution of \eqref{eq6.29} with $\mu=0$. On $\p\O$ we have
$$\aligned
&\mH_T^+=(\mH^e_T)^+ +(\nabla\phi_{\bold v,0})_T^+=(\mH^e_T)^++\bold v,\\
&\nu\cdot\mH^+=\nu\cdot[\mH^e +\nabla\phi_{\bold v,0}]^+=\nu\cdot(\mH^e)^+ +{\p\phi_{\bold v,0}\over\p\nu}.
\endaligned
$$
Hence \eqref{cond-6.15} takes the form of \eqref{eq6.33}. So conclusion (i) follows from Proposition \ref{Prop6.5} and Step 2 of its proof.
Conclusion (ii) follows from Proposition \ref{Prop6.5} directly.
\end{proof}

\v0.05in

\section{The Meissner System}\label{Section7}

\subsection{Existence of solutions to \eqref{eqfH}-\eqref{cond-1.12}}\

We first show that the equivalence of \eqref{eqfA}-\eqref{1.4} with \eqref{eqfH}-\eqref{cond-1.12} holds if we require the solutions of \eqref{eqfH} having continuous normal components across $\p\O$.

\begin{Lem}[Equivalence]\label{Lem7.1}
Assume that $\O$ is a bounded domain in $\mathbb R^3$ with a $C^4$ boundary, and $\mH^e$ satisfies $(H)$.
\begin{itemize}
\item[(i)] Let $(f,\A)\in C^2(\overline{\O})\times \mathbb
C^{3,0}_t(\overline{\O},\overline{\O^c}, \mathbb R^3)$ be a solution of
 \eqref{eqfA}-\eqref{1.4} with $f>0$ on $\overline{\O}$, and set
$\H=\lam\,\curl\A$. Then $(f,\H)\in C^2(\overline{\O})\times
C^{2,0}_t(\overline{\O},\overline{\O^c}, \mathbb R^3)$ and it is a solution of
\eqref{eqfH}-\eqref{cond-1.12}. If furthermore $\A\in \mathbb
C^{3,1}_t(\overline{\O},\overline{\O^c},\mathbb R^3)$, then $\H\in\mathbb
C^{2,0}(\overline{\O},\overline{\O^c},\mathbb R^3)$.

\item[(ii)] Assume in addition $\O$ is simply-connected and
without holes, and assume there exists $\mF^e\in
C^{2+\a}_{\loc}(\overline{\O^c},\div0)$ satisfying $\curl\mF^e=\mH^e$ in $\O^c$, where $0<\a<1$. Let
$(f,\H)\in C^{2+\a}(\overline{\O})\times \mathbb C^{2+\a,0}(\overline{\O},\overline{\O^c},\mathbb
R^3)$ be a solution of \eqref{eqfH}-\eqref{cond-1.12} such that $f>0$ on
$\overline{\O}$, and there exist $\delta>0$ and $\g>2$ such that
\begin{equation}\label{cond-7.1}
\aligned
&\H-\mH^e\in C^2_{\loc}(\overline{\O^c},\mathbb R^3)\cap
C^\delta(\overline{\O^c},\mathbb R^3),\\
&\H-\mH^e=O(|x|^{-\g})\q\text{as }|x|\to\infty.
\endaligned
\end{equation}
\begin{itemize}
\item[(iia)] If $\mF^e\in
C^{3+\a}_{\loc}(\overline{\O^c},\mathbb R^3)$ and $\H\in \mathbb
C^{3+\a,0}(\overline{\O},\overline{\O^c},\mathbb R^3)$, then there exists $\A\in \mathbb C^{2+\a,0}(\overline{\O},\overline{\O^c},\mathbb R^3)$ such that
$\H=\lam\,\curl\A$ and $(f,\A)$ is a solution of \eqref{eqfA}-\eqref{1.4}.

\item[(iib)] If $\H\in \mathbb C^{2+\a,0}(\overline{\O},\overline{\O^c},\mathbb R^3)$ and
$\nu\cdot\curl\H=0$ on $\p\O$, then there exists $\A\in\mathbb
C^{2+\a,0}(\overline{\O},\overline{\O^c},\mathbb R^3)$ such that
$\H=\lam\,\curl\A$ and $(f,\A)$ is a Meissner solution of
\eqref{eqfA}-\eqref{1.4}.
\end{itemize}
\end{itemize}
\end{Lem}

The proof of Lemma \ref{Lem7.1} involves some arguments used in
\cite[Lemma 3.3]{P3}. For completeness we give a brief proof in Appendix \ref{AppendixF}.
We emphasize that continuity of normal components of solutions to \eqref{eqfH} is important for constructing solutions of \eqref{eqfA} using the solutions of \eqref{eqfH}. When such continuity is required, we need the following assumption:
\eq\label{cond-7.2}
\text{$\O$ satisfies $(O)$ with $r\geq 3$ and $0<\a<1$, \q $\mH^e$ satisfies $(H)$},
\eeq
where the condition $(H_0)$ is replaced by a stronger condition $(H)$.

Similar to Lemma \ref{Lem6.2} we have the following existence result for \eqref{eqfH}-\eqref{cond-1.12}.

\begin{Lem}\label{Lem7.2}
Assume $\O$ and $\mH^e$ satisfy \eqref{cond-7.2}, and $\mH$ satisfies \eqref{cond-6.10}. Let $\lam_{f\A}(\O,\mH_T)$ and $\k_{f\A}(\O,\mH_T,\lam)$ be the numbers given in Theorem \ref{Thm4.9}.
For all $0<\lam< \lam_{f\A}(\O,\mH_T)$ and  $\k> \k_{f\A}(\O,\mH_T,\lam)$,
problem \eqref{eqfH}-\eqref{cond-1.12} has a solution $(f,\H)\in\mathbb
U(\O)$ with $\H=\mH$ in $\O^c$, and
$$(f,\H)\in
C^{3+\b}(\overline{\O})\times \C^{2+\a,0}_t(\overline{\O},\overline{\O^c},\mathbb
R^3)\q \text{for any }0<\b<\min\{1/2,\a\}.
$$
\end{Lem}

\begin{proof} Let $\mH$ satisfy \eqref{cond-6.10}. From Theorem \ref{Thm4.9},  for all $0<\lam<\lam_{f\A}(\O,\mH_T)$ and $\k>\k_{f\A}(\O,\mH_T,\lam)$, \eqref{eq4.7} with
boundary data $\mB_T$ replaced by $\mH_T$ has a unique
solution $(f,\H^i)\in C^{2+\a}(\overline{\O})\times C^{2+\a}(\overline{\O},\mathbb
R^3)$, and in fact $(f,\H^i)\in \mathbb U(\O)$. From Theorem \ref{Thm-reg-fA},
$f\in C^{3+\b}(\overline{\O})$ for any $0<\b<\min\{1/2,\a\}$. Define a vector
field $\H$ in $\Bbb R^3$ by letting $\H=\H^i$ in $\O$ and $\H=\mH$ in
$\O^c$. Since $\H^i_T=\mH_T$ on $\p\O$, we have $[\H_T]=0$ on
$\p\O$, and hence $\H\in
\C^{2+\a,0}_t(\overline{\O},\overline{\O^c},\mathbb R^3)$. Thus $(f,\H)$
solves \eqref{eqfH}-\eqref{cond-1.12}.
\end{proof}

\subsection{Existence of solutions to \eqref{eqfA}-\eqref{1.4}}\

\begin{Def}\label{Def7.3} Assume $\O$ and $\mH^e$ satisfy \eqref{cond-7.2}, and $\mH$ satisfies \eqref{cond-6.10}.
Let $\lam_{f\A}(\O,\mH_T)$ and $\k_{f\A}(\O,\mH_T,\lam)$ be the numbers given in Theorem \ref{Thm4.9}.
For all $0<\lam< \lam_{f\A}(\O,\mH_T)$ and  $\k> \k_{f\A}(\O,\mH_T,\lam)$, let $(f,\A)$ be the unique Meissner solution of \eqref{dec-eqfA-1} lying in $\Bbb K(\O)$ with the boundary data $\mB_T$ replaced by $\mH_T^+$. Then we denote
\eq\label{map-P}
\mathcal P(\lam,\k,\mH_T^+)=\lam\,\curl\A.
\eeq
\end{Def}

\begin{Lem}\label{Lem7.4}
Assume $\O$ and $\mH^e$ satisfy \eqref{cond-7.2}. For any $\var_0>0$ small, there exist positive numbers $\lam_{f\A}(\O,\var_0)$ and $\k_{f\A}(\O,\var_0,\lam)$ such that, for any $0<\lam<\lam_{f\A}(\O,\var_0)$ and $\k>\k_{f\A}(\O,\var_0,\lam)$, if there exists $\mH$ on
$\overline{\O^c}$ satisfying the following conditions:
\begin{itemize}
\item[(a)]  $\mH$ satisfies \eqref{cond-6.10}  and \eqref{cond-6.14};
\item[(b)] $\mH$ satisfies the following comparability condition
\eq\label{cond-7.4}
\nu\cdot[\mathcal P(\lam,\k,\mH^+_T)]^-=\nu\cdot\mH^+\q\text{\rm on}\;\;\p\O;
\eeq
\item[(c)] $\mH$ satisfies  \eqref{cond-6.16} for some vector field $\mF$;
\end{itemize}
then  problem \eqref{eqfA}-\eqref{1.4} has a Meissner solution
$(f,\A)$ such that $\lam\,\curl\A=\mH$ in $\O^c$,
and $(f,\A)\in
C^{3+\b}(\overline{\O})\times\C^{2+\a,0}(\overline{\O},\overline{\O^c},\mathbb R^3)$
for any $0<\b<1/2$.
\end{Lem}

\begin{proof}
Let $\mH$ satisfy \eqref{cond-6.10}. From Theorem \ref{Thm4.9}, for all $0<\lam<\lam_{f\A}(\O,\mH_T)$ and $\k>\k_{f\A}(\O,\mH_T,\lam)$, BVP \eqref{dec-eqfA-1} with $\mB_T=\mH_T$
has a Meissner solution $(f_\lam,\A^i_\lam)\in C^{2+\a}(\overline{\O})\times C^{2+\a}(\overline{\O},\mathbb R^3)$, and  it is the only solution in $\mathbb K(\O)$.

We set $\mA_T=(\A^i_\lam)^-_T$ and look for a solution $\A^o$ of \eqref{dec-eqfA-2}-\eqref{1.4} for these $\mA_T$ and $\mB_T$. If \eqref{dec-eqfA-2}-\eqref{1.4} has a solution $\A^o$, since $\mH$ satisfies \eqref{cond-6.16} for some $\mF$, we can write as in the proof of Proposition \ref{Prop6.5} that $\lam\,\A^o=\mF+\nabla\psi$ in $\O^c$, where $\psi$ satisfies
\eq\label{cond-7.5}
(\nabla\psi)_T=\lam(\A^i_\lam)^-_T-\mF_T^+\q\text{\rm on}\;\;\p\O.
\eeq
Since $(\A^i_\lam)^-\in C^{2+\a}(\p\O,\Bbb R^3)$ and $\mF\in
C^{2+\a}(\p\O,\Bbb R^3)$, so $\lam(\A^i_\lam)^- -\mF\in
C^{2+\a}(\p\O,\Bbb R^3).$
Since $\O^c$ is simply-connected, as in the proof of Proposition \ref{Prop6.5} we
can show that, existence of $\psi$ satisfying \eqref{cond-7.5} is true if and only if the following equality holds:
\eq\label{cond-7.6}
\nu\cdot\curl(\lam\,(\A^i_\lam)_T-\mF^+_T)=0\q\text{\rm on}\;\;\p\O,
\eeq
which is exactly \eqref{cond-7.4} because
$$\nu\cdot\curl(\lam\,(\A^i_\lam)_T-\mF^+_T)=\lam\nu\cdot\curl(\A^i_\lam)_T-\nu\cdot\mH=\nu\cdot[\mathcal P(\mH^+_T)]^- -\nu\cdot\mH^+.
$$
Thus \eqref{cond-7.4} is necessary for solvability of \eqref{dec-eqfA-2}-\eqref{1.4}.

Now we fix $\var_0>0$. From the discussions in \cite{BaP} we know that the constants $\lam_{f\A}(\O,\mH_T)$ and $\k_{f\A}(\O,\mH_T,\lam)$ can be chosen uniformly valid for all $\mH$ satisfying \eqref{cond-6.14}. So we can find positive constants $\lam_{f\A}(\O,\var_0)$ and $\k_{f\A}(\O,\var_0,\lam)$ such that, for all $\mH$ satisfying \eqref{cond-6.14} it holds that
\eq\label{lam-ka-var}
\aligned
&\lam_{f\A}(\O,\mH_T)\geq \lam_{f\A}(\O,\var_0),\\
&\k_{f\A}(\O,\mH_T,\lam)\geq \k_{f\A}(\O,\var_0,\lam)\q\forall 0<\lam<\lam_{f\A}(\O,\var_0,\lam).
\endaligned
\eeq

Now assume $0<\lam<\lam_{f\A}(\O,\var_0)$ and $\k>\k_{f\A}(\O,\var_0,\lam)$, and assume $\mH$ satisfies \eqref{cond-6.14} and \eqref{cond-7.4}. Then \eqref{cond-7.6} holds, and we can find a harmonic function $\psi^\lam\in C^{3+\a}_{\loc}(\overline{\O^c})$ satisfying \eqref{cond-7.5}.
We choose $\xi^\lam\in C^{3+\a}_{\loc}(\overline{\O^c})$ such that
$$
 \xi^\lam=0\q\text{\rm and}\q {\p\xi^\lam\over\p\nu}=-\nu\cdot\mF^+-{\p\psi^\lam\over\p\nu}\q\text{\rm on }\p\O,
$$
and define
$$\A_\lam=\begin{cases} \A^i_\lam\q&\text{in }\O,\\
\lam^{-1}(\mF+\nabla\psi^\lam+\nabla\xi^\lam)\q&\text{in }\O^c.
\end{cases}
$$
Using the facts $(\nabla\xi^\lam)_T=0$ and $\nu\cdot\A^i_\lam=0$ on $\p\O$, we see that $(\A_\lam)^-=(\A_\lam)^+$ on
$\p\O$, and thus $\A_\lam\in \C^{2+\a,0}(\overline{\O},\overline{\O^c},\mathbb R^3)$.
So $(f_\lam,\A_\lam)$ is a Meissner solution of problem \eqref{eqfA}-\eqref{1.4}, and it has
the properties mentioned in the lemma.
\end{proof}

Combining Lemma \ref{Lem7.4} with Lemma \ref{Lem-loc-stable} and Lemma \ref{Lem6.6} we get the following existence result for problem \eqref{eqfA}-\eqref{1.4}, which is similar to Theorem \ref{Thm6.7}.

\begin{Thm}\label{Thm7.5} Assume $\O$ and $\mH^e$ satisfy \eqref{cond-7.2}.
\begin{itemize}
\item[(i)] For small $\lam$ and large $\k$, problem \eqref{eqfA}-\eqref{1.4} has a classical Meissner solution if and only if there exists a vector field $\bold v\in \mB^{2+\a}(\p\O)$ satisfying \eqref{cond-6.28} such that
\eq\label{cond-7.8}
\nu\cdot[\mathcal P(\lam,\k,(\mH^e_T)^++\bold v)]^-=\nu\cdot(\mH^e)^+ +{\p\phi_{\bold v,0}\over\p\nu}\q\text{\rm on}\;\;\p\O,
\eeq
where $\phi_{\bold v,0}$ is the solution of \eqref{eq6.29} for $\mu=0$.
\item[(ii)] Given $\var_0>0$, let $\lam_{f\A}(\O,\var_0)$ and $\k_{f\A}(\O,\var_0,\lam)$ be the numbers given in Lemma \ref{Lem7.4}. For any $0<\lam<\lam_{f\A}(\O,\var_0)$ and $\k>\k_{f\A}(\O,\var_0,\lam)$, if \eqref{cond-7.8} holds for some $\bold v\in \mB^{2+\a}(\p\O)$ satisfying \eqref{cond-6.34}, then
\eqref{eqfA}-\eqref{1.4} has a locally $L^\infty$-stable, classical Meissner solution $(f,\A)$ such that $(f,\A)\in C^{3+\b}(\overline{\O})\times\mathbb
C^{2+\a,0}(\overline{\O},\overline{\O^c},\mathbb R^3)$ with $0<\b<1/2$, $\A$ satisfies \eqref{cond-1.8} and
$$
\lam\,\curl\A=\begin{cases}
\mathcal P(\lam,\k, (\mH^e_T)^++\bold v)\q&\text{\rm in }\O,\\
\mH^e+\nabla\phi_{\bold v,0}\q&\text{\rm in }\O^c.
\end{cases}
$$
\end{itemize}
\end{Thm}

\begin{Def}\label{Def7.6} Assume $\O$ and $\mH^e$ satisfy \eqref{cond-7.2} and $\mH$ satisfies \eqref{cond-6.10}. Let $(f,\A)$ be the unique Meissner solution of \eqref{dec-eqfA-1} in $\Bbb K(\O)$ with boundary data $\mB_T$ replaced by $\mH_T^+$. Define a \emph{Dirichlet-to-Neumann} map $\Pi$ by
\eq\label{map-pi}
\Pi(\lam,\k,\mH_T^+)=\nu\cdot[\mathcal P(\lam,\k,\mH_T^+)]^-\equiv \nu\cdot(\lam\,\curl\A)^-,
\eeq
where $\nu$ is the unit outer normal vector to $\p\O$.
Note that up to the scalar multiplier $\lambda$, the operator $\Pi$ maps the tangential component of $\curl\A$ of the magnetic potential part of  a solution $(f,\A)$ of \eqref{dec-eqfA-1} to the normal component of the curl.
With this map, the comparability condition \eqref{cond-6.15} can be written as
\eq\label{eq-pi}
\Pi(\lam,\k,\mH_T^+)=\nu\cdot\mH^+\q\text{\rm on }\p\O.
\eeq
\end{Def}

%\v0.1in

\appendix

\section{Proof of Lemmas \ref{Lem3.5}, \ref{Lem-Uniq}}\label{AppendixA}

\begin{proof}[\bf Proof of Lemma \ref{Lem3.5}]\

{\it Step 1}. We prove (i). Let $(f,\A)$ be a weak solution of \eqref{eqfA}. Fix $R>0$ so that $\O\Subset B(0,R/2)$. We decompose
$\A=\A_R+\nabla\phi_R$ on $B(0,R)$, where $\phi_R\in H^1(B(0,R),\Bbb R^3)$, $\A_R\in H^1_{n0}(B(0,R),\div0)$.
From \eqref{wkeqfA-2} we have
$$
\int_{\O^c\cap B(0,R)}\curl\A_R\cdot\curl\B\, dx
=0,\q \forall \B\in C^1_c(\O^c\cap B(0,R),\Bbb R^3).
$$
Since $\div\A_R=0$ in $B(0,R)$, using this equality and applying the difference-quotient method (see for instance \cite[Section 4]{BaP}), we can show that $\A_R\in H^2_{\loc}(\O^c\cap B(0,R),\Bbb R^3)$, and
$\curl^2\A_R=\0$ a.e. in $\O^c\cap B(0,R)$.
This is true for any large $R$,
so \eqref{curl2A0} is true. This together with
the assumption $\curl\A\in L^2_{\loc}(\Bbb R^3,\Bbb R^3)$ (see Definition \ref{Def3.1})
imply that $\curl\A\in \mH_{\loc}(\O^c,\curl0)$. By the trace theorem of $\curl$-spaces (see
\cite[p.204, Theorm 2]{DaL3}), the outer tangential trace $(\curl\A)_T^+$
exists in $T\!H^{-1/2}(\p\O,\Bbb R^3)$. Then using \eqref{wkeqfA-2} and integration by parts we get
\eqref{wkeqA-exterior}.

{\it Step 2}. We prove (ii).
Let $\overline{B}(x_0,R)\subset\O$, denote $B_R=B(x_0,R)$, and decompose $\A=\A_R+\nabla\phi_R$ in $B_R$ with
$\A_R\in H^1_{n0}(B_R,\div0)$. If $\B\in C^2_c(B_R,\Bbb R^3)$,
we have
$$\aligned
&\int_{B_R}\curl\A_R\cdot\curl\B\, dx=\int_{B_R}\A_R\cdot\curl^2\B\, dx
=\int_{B_R}\A_R\cdot[-\Delta\B+\nabla\div\B] dx\\
=&\int_{B_R}[D\A_R\cdot D\B-(\div\A_R)\cdot(\div \B)]dx
=\int_{B_R}D\A_R\cdot D\B\, dx.
\endaligned
$$
So from \eqref{3.6} we have
$$\int_{B_R}\{D\A_R\cdot D\B+\lam^{-2}f^2\A\cdot\B\}dx=0,\q \forall \B\in C^2_c(B_R,\Bbb R^3).
$$
Hence $\A_R$ is a weak solution of $\Delta\A_R=\lam^{-2}f^2\A$ in $B_R$.
By the standard $L^2$ estimate of Laplace equation we have $\A_R\in H^2(B_{R/2},\Bbb R^3)$, so
$\curl\A=\curl\A_R\in H^1(B_{R/2},\Bbb R^3)$, hence
$\curl^2\A=\curl^2\A_R$
exists for a.e. $x\in B_{R/2}$ and it belongs to $L^2(B_{R/2},\Bbb R^3)$, and
the equality
$\curl^2\A=\curl^2\A_R=-\Delta \A_R=-\lam^{-2}f^2\A$
holds for a.e. $x\in B_{R/2}$. So we conclude that $\curl^2\A$ exists for a.e. $x\in\O$, and
\eq\label{A.1}
\curl^2\A=-\lam^{-2}f^2\A\q\text{for a.e. }x\in\O.
\eeq
The right hand side of \eqref{A.1} belongs to $L^2(\O,\Bbb R^3)$,
so does the left hand side, hence $\curl\A\in \mH(\O,\curl)$. Then from the trace theorem of curl-spaces the inner tangential trace
$(\curl\A)_T^-\in T\!H^{-1/2}(\p\O,\Bbb R^3)$.
From this and Step 1 we have $[(\curl\A)_T]=(\curl\A)_T^+ -(\curl\A)_T^-\in H^{-1/2}(\p\O,\Bbb R^3)$.
Then using \eqref{3.6} and integration by parts we have
$$
\int_\O\B\cdot(\lam^2\curl^2\A+f^2\A)dx-\lam\int_{\p\O}\{\B\times[(\curl\A)_T]\}\cdot\nu dS=0,\q\forall \B\in C^1_c(\Bbb R^3,\Bbb R^3).
$$
From this and \eqref{A.1} we get \eqref{t-cont-curlA}.

{\it Step 3}. We prove (iii). Since $\A\in \mA(\O,\Bbb R^3)$,
so $\mB=\curl\A\in L^2_{\loc}(\Bbb R^3,\Bbb R^3)$. From \eqref{t-cont-curlA} we have $[\mB_T]=\0$ on $\p\O$,
and since $\curl(\mB|_{\O})\in L^2(\O,\Bbb R^3)$ and
$\curl(\mB|_{\O^c})=\0$, we see that
$\curl\mB$ is well-defined
in $\Bbb R^3$ and $\curl\mB\in \mH(\Bbb R^3,\curl, \div0)$, hence
$\mB=\curl\A\in H^1_{\loc}(\Bbb R^3,\Bbb R^3)$.
So $(\curl\A)^+=(\curl\A)^-\in H^{1/2}(\p\O,\Bbb R^3)$.
\end{proof}

\begin{proof}[\bf Proof of Lemma \ref{Lem-Uniq}]\

Let $G(f,\A)(x)$ be the function given in \eqref{funct-E-Omega}, and we use the notation $G'(f,\A)$, $G''(f,\A)$, $G'_f(f,\A)$ and $G'_\A(f,\A)$ given after \eqref{funct-E-Omega}.
Let $(f_0,\A_0)$ and
$(f_1,\A_1)\in H^1(\O)\times \mA(\O,\Bbb R^3,\lam^{-1}\mH^e)$ be two Meissner solutions of \eqref{eqfA}-\eqref{1.4}.
Applying the integral form of \eqref{eqfA} to $(f_0,\A_0)$ and $(f_1,\A_1)$ respectively and subtracting one from another we get, for any $(g,\w)\in H^1(\O)\times \mB(\O,\Bbb R^3)$,
$$
\aligned
\int_\O\Bigl\{&{\lam^2\over\k^2}\nabla (f_1-f_0)\cdot\nabla g+{1\over
2}\big\langle
[G'(f_1,\A_1)-G'(f_0,\A_0)],(g,\w)\big\rangle\Big\}dx\\
&+\lam^2\int_{\Bbb R^3}\curl(\A_1-\A_0)\cdot\curl\w dx=0.
\endaligned
$$
Let $g=f_1-f_0$, $f_t=f_0+tg$,
$\w=\A_1-\A_0$, $\A_t=\A_0+t\w$. Since $\mH^e$ satisfies $(H)$, we can show $\w\in\mB(\O,\Bbb R^3)$.
Hence we have
\eq\label{A.2}
\aligned
&{\lam^2\over\k^2}\int_\O|\nabla g|^2dx+\lam^2\int_{\Bbb R^3}|\curl\w|^2dx\\
&+\int_\O\int_0^1\{|f_t\w+2g\A_t|^2+(3f_t^2-3|\A_t|^2-1)|g|^2\}dtdx=0.
\endaligned
\eeq
Since $(f_0,\A_0)$ and $(f_1,\A_1)$ satisfy \eqref{f-A-stable}, we can show that
$$
3f_t^2-3|\A_t|^2-1>0\q\text{for } x\in \O,\; 0\leq t\leq 1,
$$
see second part of Remark \ref{Rem4.1} (c). From this and \eqref{A.2} we see that $g=0$ in $\O$, i.e. $f_1\equiv
f_0$ on $\O$, which together with \eqref{A.2} implies that
$f_t\w+2g\A_t=f_0\w=\0$ in $\O$.
Since $f_0>0$, we have $\w=\0$ in $\O$, so $\A_1=\A_0$ in $\O$. Hence $(f_1,\A_1)=(f_0,\A_0)$ in $\O$.
Finally since $\curl\w=\0$ in $\Bbb R^3$ we see that
$\curl\A_1=\curl\A_0$ in $\O^c$.
\end{proof}

\section{Estimates of Solutions to BVP \eqref{dec-eqfA-1}}\label{AppendixB}

If $(f,\A)$ is a weak solution of \eqref{dec-eqfA-1}, then
$(f,\A)\in\mW(\O)$, $f$ is a weak solution of \eqref{3.4},
and $\A$ is a weak solution of
\eq\label{eqB.1}
\left\{\aligned
&\lam^2\curl^2\A+f^2\A=\0\q&\text{in }\O,\\
&\lam(\curl\A)^- _T=\mB_T\q&\text{\rm on }\p\O.
\endaligned\right.
\eeq

\begin{Lem}\label{LemB.1}
Assume that $\O$ is a bounded domain in $\mathbb R^3$
with a $C^2$ boundary and
$\mB_T\in T\!H^{-1/2}(\p\O,\mathbb R^3)$.
Let $(f,\A)\in\mW(\O)$ be a weak solution of \eqref{dec-eqfA-1} and set $\H=\lam\,\curl\A$.
\begin{itemize}
\item[(i)] For all $1<p<\infty$ and all $\a\in (0,1)$, we have
$$
\aligned
&f\in W^{2,p}(\O)\cap C^{1+\a}(\overline{\O}),\qq |f|\leq 1,\\
&f^2\A\in \mH(\O,\curl,\div0),\qq
\H\in\mH(\O,\curl,\div0),
\endaligned
$$
and there exists a positive constant $C=C(\O,\|\A\|_{L^\infty(\O)},\a)$ such that
\eq\label{f-C1}
\aligned
\Big({\lam\over\k}\Big)^\a[f]_\a+{\lam\over\k}\|Df\|_{C^0(\overline{\O})}+\Big({\lam\over\k}\Big)^{1+\a}[Df]_\a \leq C.
\endaligned
\eeq
If furthermore $\mB_T\in H^{1/2}(\p\O,\Bbb R^3)$,
then $\H\in H^1(\O,\Bbb R^3)$.

\item[(ii)] If $\A$ satisfies \eqref{1.5}, then $f^2\A\in H^1(\O,\Bbb R^3)$ and $\mB_T$ must satisfy \eqref{cond-3.27}.
If furthermore $\p\O$ is of $C^3$ and $\mB_T\in T\!H^{3/2}(\O,\Bbb R^3)$, then $\H\in H^2(\O,\Bbb R^3)$.
\end{itemize}
\end{Lem}

\begin{proof}
(i) Since $(f,\A)\in\mW(\O)$, so $\A\in L^\infty(\O),\Bbb R^3)$.
 From \eqref{wkdec-eqfA-1} we get
\begin{equation}\label{eqB.3}
\|\nabla
f\|_{L^2(\O)}^2=\lam^{-2}\k^2\int_\O(1-f^2-|\A|^2)|f|^2dx,
\end{equation}
from which we get
\begin{equation}\label{est-B.4}
\|f\|_{H^1(\O)}\leq
(1+\lam^{-1}\k)\|f\|_{L^2(\O)}.
\end{equation}
Separate the equation for $f$ in \eqref{3.4}, and use elliptic regularity  theory we have $f\in
W^{2,p}(\O)$ for all $1<p<\infty$ and $f\in C^{1+\a}(\overline{\O})$ for all $0<\a<1$.
The maximum principle applying to \eqref{3.4} gives $|f|\leq 1$, and elliptic estimates applying to the rescaled functions gives \eqref{f-C1}.

Now assume $\mB_T\in T\!H^{1/2}(\p\O,\Bbb R^3)$.
Since $\curl\H=-\lam^{-1}f^2\A\in L^2(\O,\Bbb R^3)$, $\div\H=0$ and $\H_T=\mB_T\in H^{1/2}(\p\O,\Bbb R^3)$, by \eqref{dcg2} we have $\H\in H^1(\O,\Bbb R^3)$.

(ii) If $\A$ satisfies \eqref{1.5}, then $\nu\cdot(f^2\A)=0$ on $\p\O$. This together with $\curl(f^2\A)\in L^2(\O,\Bbb R^3)$ and $\div(f^2\A)=0$ implies that $f^2\A\in H^1(\O,\Bbb R^3)$, see \eqref{dcg1}. From the second
equation in \eqref{dec-eqfA-1} we have
$\nu\cdot\curl\mB_T=\nu\cdot\curl\H_T=\lam^{-2}\nu\cdot (f^2\A)=0.$
So $\mB_T$ satisfies \eqref{cond-3.27}.
\end{proof}

In the following lemma we assume $\k\geq \max\{1,\lam\}$, which makes the estimates simpler. The regularity results remain true without this assumption.

\begin{Lem}\label{LemB.2}
Let $\O$ be a bounded domain in $\Bbb R^3$  with a $C^2$ boundary, $\mB_T\in T\!H^{1/2}(\p\O,\Bbb R^3)$.
Assume $(f,\A)\in \mW(\O)$ is a weak Meissner solution of \eqref{dec-eqfA-1} and let
$\H=\lam\,\curl\A$. Denote
\eq\label{B.5}
c=\min_{\overline{\O}}f,\q M=\|\A\|_{L^\infty(\O)},\q d_1=c^{-1}\lam^{-1}\k,
\eeq
and let  $0<\a<1$,  $0<\beta<1/2$, $\k\geq \max\{1,\lam\}$. Then we have the following conclusions:
\begin{itemize}
\item[(a)] $f\in H^2(\O)\cap C^{1+\a}(\overline{\O})$, $\A\in H^1(\O,\Bbb R^3)$,
$\H\in H^1(\O,\Bbb R^3)$, $\curl\H\in H^1(\O,\Bbb R^3)\cap L^\infty(\O,\Bbb R^3)$, and
\eq\label{B.6}
\aligned
&\|\H\|_{L^2(\O)}\leq C(\O)\lam^{1/2}M^{1/2}\|\mB_T\|_{L^1(\p\O)}^{1/2},\\
&\|\H\|_{H^1(\O)}\leq C(\O)\{(\lam+\lam^{-1})M
+\|\mB_T\|_{H^{1/2}(\p\O)}\},\\
\endaligned
\eeq
\begin{equation}\label{B.7}
\aligned
&\|\curl\H\|_{L^\infty(\O)}\leq C(\O)\lam^{-1}M,\\
&\|\curl\H\|_{H^1(\O)} \leq
C(\O)\lam^{-1}\{d_1 M+\|\mB_T\|_{L^1(\p\O)}\},
\endaligned
\end{equation}
\begin{equation}\label{A-H1}
\aligned
&\|\A\|_{H^1(\O)}\leq C(\O)\{d_1 M +\|\mB_T\|_{L^1(\p\O)}\},\\
&\|\A\|_{C^\b(\overline{\O})}\leq C_1\{d_2M+\lam^{-1}\|\mB_T\|_{L^1(\p\O)}\},
\endaligned
\end{equation}
where   $d_2=1+\lam^{-2}+d_1$ and $C_1=C(\O,M,\b)$.

\item[(b)] If $\p\O$ is of class $C^3$, then
$f\in H^3(\O)\cap C^{2+\beta}(\overline{\O})$, $\A\in H^2(\O,\Bbb R^3)$, $\curl\H\in H^2(\O,\Bbb R^3)$, and
\begin{equation}\label{A-H2}
\|\A\|_{H^2(\O)}\leq C_2\{d_1M +\|\mB_T\|_{L^1(\p\O)}\},
\end{equation}
\begin{equation}\label{curl-H2}
\|\curl\H\|_{H^2(\O)}\leq
C_2\lam^{-1}\{d_1M+\|\mB_T\|_{L^1(\p\O)}\},
\end{equation}
where $C_2=C(\O,M)c^{-2}d_1$.
If $\mB_T\in T\!H^{3/2}(\p\O,\Bbb R^3)$, then $\H\in H^2(\O,\Bbb R^3)$, and
\begin{equation}\label{B.11}
\|\H\|_{H^2(\O)}\leq
C(\O)\lam^{-1}\{d_1M+\|\mB_T\|_{L^1(\p\O)}+\lam\|\mB_T\|_{H^{3/2}(\p\O)}\}.
\eeq

\item[(c)] If $\p\O$ is of class $C^{3+\beta}$, then
$f\in C^{3+\b}(\overline{\O})$,
$\A\in H^2(\O,\Bbb R^3)\cap C^{1+\b}(\overline{\O},\Bbb R^3)$,
$\curl\H\in C^{1+\b}(\overline{\O},\Bbb R^3)$, and we have the estimate
$$
\sum_{n=0}^3\Big({\lam\over\k}\Big)^n\|D^n
f\|_{C^0(\overline{\O})}+\Big({\lam\over\k}\Big)^{3+\b}[D^3f]_\b \leq
C(\O,M,\b).
$$
\begin{itemize}
\item[(c.1)] If furthermore $\p\O$ is of class $C^{m+2}$ with $m\geq
2$, and $\mB_T\in H^{n+1/2}(\p\O,\mathbb R^3)$ for some $1\leq n\leq
m$, then $\H\in H^{n+1}(\O,\mathbb R^3)$, and
\begin{equation}\label{H-H3}
\aligned
\|\H\|_{H^3(\O)}\leq& C_3\{d_1M +\|\mB_T\|_{L^1(\p\O)}\}
+C(\O)\|\mB_T\|_{H^{5/2}(\p\O)},\\
\|\A\|_{C^{1+\b}(\overline{\O})}\leq&
C_4\{d_3M +\|\mB_T\|_{L^1(\p\O)}+c\lam\k^{-1}\|\mB_T\|_{H^{3/2}(\p\O)}\},
\endaligned
\end{equation}
where $C_3=C(\O,M)c^{-2}\lam^{-1}d_1$, $C_4=C(\O,M,\b)\lam^{-1}d_1$, $d_3=d_1(c+\lam\k)$.

\item[(c.2)] If $\p\O$ is of class $C^{m+2+\a}$ with $m\geq 2$,
$\mB_T\in C^{n+\a}(\p\O,\mathbb R^3)$ for some $1\leq n\leq m-1$, and $(f,\A)\in \mK(\O)$, then
$f\in C^{n+3+\a}(\overline{\O})$, $\A\in C^{n+1+\a}(\overline{\O},\mathbb R^3)$,
$\H\in C^{n+\a}(\overline{\O},\mathbb R^3)$, $\curl\H\in
C^{n+1+\a}(\overline{\O},\mathbb R^3)$,
and there exists a positive constant $C=C(\O,n,\k,\a,\lam,\|\mB_T\|_{C^{n+\a}(\p\O)})$ such that
\eq\label{B.13}
\aligned
\|f\|_{C^{n+\a}(\overline{\O})}+\|\H\|_{C^{n+\a}(\overline{\O})}+
\|\A\|_{C^{n+a}(\overline{\O})}\leq C.
\endaligned
\eeq
\end{itemize}
\end{itemize}
\end{Lem}

\begin{proof} {\it Step 1.} Assume $\p\O$ is of $C^2$, $(f,\A)\in\mW(\O)$
is a weak Meissner solution of \eqref{dec-eqfA-1}.

{\it Step 1.1}. We show $\A\in H^1(\O,\Bbb R^3)$ and $\div\A\in H^1(\O)$. From Lemma \ref{LemB.1},
$f\in H^2(\O)\cap C^{1+\a}(\overline{\O})$, $f^2\A\in H^1(\O,\Bbb R^3)$, and there exists $c>0$ such that
$c\leq f(x)\leq 1$. So $\A=f^{-2}(f^2\A)\in H^1(\O,\mathbb R^3)$.
Since $\div(f^2\A)=0$, we have
\begin {equation}\label{eq-divA}
\div\A=-{2\over f}\nabla f\cdot\A\q\text{a.e. in }\O.
\end{equation}
So $\div\A\in H^1(\O)\cap L^\infty(\O)$, and we can apply the standard difference-quotient method to \eqref{eqB.1} and show $\A\in
H^2_{\loc}(\O,\mathbb R^3)$.

{\it Step 1.2}. We show $\H, \curl\H\in H^1(\O,\Bbb R^3)$. Since $\A\in H^2_{\loc}(\O,\Bbb R^3)$, from \eqref{eqB.1} we have
\begin{equation}\label{eq-curlH}
\curl\H=-\lam^{-1}f^2\A\q\text{a.e. in }\O.
\end{equation}
The right side belongs to $H^1(\O,\Bbb R^3)\cap L^\infty(\O,\Bbb R^3)$. Since $\div\H=0$ in $\O$ and $\H_T=\mB_T\in H^{1/2}(\p\O,\Bbb R^3)$, using \eqref{dcg2} we get $\H\in H^1(\O,\mathbb R^3)$.
From \eqref{eqB.1} we have
\begin{equation}\label{eq-B.16}
\aligned
\int_\O(\lam^2|\curl\A|^2+f^2|\A|^2)dx
=\lam\int_{\p\O}(\A_T\times\H_T)\cdot\nu dS. \endaligned
\end{equation}
So the first inequality in \eqref{B.6} follows.
From \eqref{eq-curlH} and \eqref{eq-B.16} we have
$$\aligned
&\int_\O(|\H|^2+|\curl\H|^2)dx=
\int_\O(|\H|^2+\lam^{-2}f^4|\A|^2)dx
\leq
\max\{\lam,\lam^{-1}\}\|\A\|_{C^0(\overline{\O})}\|\mB_T\|_{L^1(\p\O)}.
\endaligned
$$
From this and \eqref{dcg2} we have
$$
\|\H\|_{H^1(\O)}^2\leq C(\O)\big(\max\{\lam,\lam^{-1}\}\|\A\|_{L^\infty(\O)}+\|\mB_T\|_{H^{1/2}(\p\O)}\big)\|\mB_T\|_{H^{1/2}(\p\O)}.
$$
Since
$\|\mB_T\|_{H^{1/2}(\p\O)}\leq C(\O)\|\H\|_{H^1(\O)},$
we get the second inequality in \eqref{B.6}.

From \eqref{eq-curlH} we get the first inequality of
\eqref{B.7}. Using \eqref{eq-curlH} we also get
\begin{equation*}
\aligned
&\|\curl\H\|_{H^1(\O)}\leq \|D(\curl\H)\|_{L^2(\O)}+\|\curl\H\|_{L^2(\O)}\\
\leq&
\lam^{-1}(\|D(f^2\A)\|_{L^2(\O)}+\|f^2\A\|_{L^2(\O)})\\
\leq&\lam^{-1}\Bigl\{\|f\|_{L^\infty(\O)}^2\|D\A\|_{L^2(\O)}
+2\|f\|_{L^\infty(\O)}\|\nabla
f\|_{L^2(\O)}\|\A\|_{L^\infty(\O)}+\|f^2\A\|_{L^2(\O)}\Bigr\}.
\endaligned
\end{equation*}
From this, \eqref{eqB.3}, \eqref{est-B.4}, \eqref{B.5},   we get the second inequality of \eqref{B.7}.

Since $\nu\cdot\A=0$ on $\p\O$, we use \eqref{eqB.3}, \eqref{eq-divA}, the first inequality in \eqref{B.6}, and \eqref{dcg1} to get
$$
\aligned
\|\A\|_{H^1(\O)} \leq& C(\O)\{\|2f^{-1}\A\|_{L^\infty(\O)}\|\nabla f\|_{L^2(\O)}+\lam^{-1}\|\H\|_{L^2(\O)}+M\}\\
\leq& C(\O)\{d_1M+\lam^{-1}(M+\lam\|\mB\|_{L^1(\p\O)})+M\},
\endaligned
$$
which yields the first inequality in \eqref{A-H1}.

{\it Step 1.3}. We show $\A$ is H\"older continuous. Take $\phi\in \dot H^1(\O)$ such that
\eq\label{eq-phi}
\Delta\phi=\lam\,\div\A\q\text{in }\O,\q {\p\phi\over\p\nu}=0\q\text{\rm on}\;\;\p\O.
\eeq
From this and \eqref{eq-divA} we have, for any $1<p<\infty$,
$$
\|\phi\|_{W^{2,p}(\O)}\leq C(\O,p)\lam\|f^{-1}\nabla f\cdot\A\|_{L^p(\O)}\leq
C(\O,M,p)c^{-1}\k M.
$$
Then from the Sobolev embedding theorem we have, for $0<\a<1$,
\begin{equation}\label{phi-Holder}
\|\phi\|_{C^{1+\a}(\overline{\O})}\leq
C(\O,M,\a)c^{-1}\k M.
\end{equation}
Let $\B=\lam\,\A-\nabla\phi$. Using \eqref{phi-Holder} we have
\eq\label{B-bd}
\|\B\|_{L^\infty(\O)}\leq C(\O,M)\lam(1+d_1)M.
\eeq
Since $\B$ satisfies
\begin{equation}\label{eq-BH}
\curl\B=\H\q\text{and}\q\div\B=0\q\text{\rm in }\O,\q
\nu\cdot\B=0\q\text{\rm on }\p\O,
\end{equation}
by \eqref{dcg3} and Sobolev imbedding we have,  for any $0<\b<1/2$,
\begin{equation}\label{B-Holder}
\|\B\|_{C^\b(\overline{\O})}\leq C(\O,\b)\|\B\|_{W^{1,6}(\O)}\leq C(\O,\b)(\|\H\|_{L^6(\O)}+\|\B\|_{L^6(\O)}).
\end{equation}
Thus $\A=\lam^{-1}(\B+\nabla\phi)\in C^{\b}(\overline{\O},\Bbb R^3)$. From \eqref{phi-Holder}, \eqref{B-bd}, \eqref{B-Holder} we have
$$
\aligned
\|\A\|_{C^\b(\overline{\O})}\leq &C(\O,\b)\lam^{-1}(\|\H\|_{H^1(\O)}+\|\B\|_{L^\infty(\O)}+\|\nabla\phi\|_{C^\b(\overline{\O})})\\
\leq &C(\O,M,\b)\{(1+\lam^{-2}+d_1)M +\lam^{-1}\|\mB_T\|_{L^1(\p\O)}\}.
\endaligned
$$
So we get the second inequality of \eqref{A-H1}.

{\it Step 2}. Assume $\p\O$ is of class $C^3$.
From step 1 the right-hand side of the equation in \eqref{3.4}
belongs to $H^1(\O)$. Applying the $H^k$ estimate to \eqref{3.4} we see that $f\in H^3(\O)$.
Since $\div\A\in H^1(\O)$, $\curl\A=\lam^{-1}\H\in H^1(\O,\Bbb R^3)$ and $\nu\cdot\A=0$ on $\p\O$, using \eqref{dcg1} and the first inequality in \eqref{A-H1} we see that $\A\in H^2(\O,\Bbb R^3)$, and \eqref{A-H2} holds.
It follows that $\curl\H=-\lam^{-1} f^2\A\in H^2(\O,\Bbb R^3)$, and
$$\aligned
\|D^2\curl\H\|_{L^2(\O)}\leq& 4\lam^{-1}\{\|D^2\A\|_{L^2(\O)}+\|\nabla f\|_{L^\infty(\O)}\|D\A\|_{L^2(\O)}\\
&\qq\q +\|D^2f\|_{L^2(\O)}\|\A\|_{L^\infty(\O)}+\|\nabla f\|_{L^2(\O)}^2\|\A\|_{L^\infty(\O)}\}\\
\leq &C(\O,M)\lam^{-1}\{\|\A\|_{H^2(\O)}+\lam^{-1}\k\|\A\|_{H^1(\O)}
+\lam^{-2}\k^2M\}.
\endaligned
$$
This together with \eqref{B.7} and  \eqref{A-H2} yields \eqref{curl-H2}.
Now the right side of the equation of \eqref{3.4} belongs to $C^\beta(\overline{\O})$, hence $f\in C^{2+\beta}(\overline{\O})$.

Now assume $\H_T\in T\!H^{3/2}(\p\O,\mathbb R^3)$. Using the first inequality in \eqref{B.6}, the second inequality in \eqref{B.7}, and \eqref{dcg2} with $k=p=2$,  we get
$$
\|\H\|_{H^2(\O)}\leq
C(\O)\{d_1M+\lam^{-1}\|\mB_T\|_{L^1(\p\O)}+\|\mB_T\|_{H^{3/2}(\p\O)}\}.
$$
So \eqref{B.11} is true.

{\it Step 3.} Assume $\p\O$ is of $C^{3+\beta}$.
Using \eqref{eq-divA} we write \eqref{eq-phi} as follows
\begin{equation}\label{eq-phi2}
-\Delta\phi=2f^{-1}\nabla f\cdot(\nabla\phi+\B)\q\text{\rm in
}\O,\q {\p\phi\over\p\nu}=0\q\text{\rm on }\p\O.
\end{equation}
By Schauder estimate, \eqref{phi-Holder} and \eqref{B-Holder} we find $\phi\in C^{2+\b}(\overline{\O})$ and
\begin{equation}\label{est-B.23}
\|\phi\|_{C^{2+\b}(\overline{\O})}\leq C\bigl\{\|\nabla\phi\|_{C^\b(\overline{\O})}
+\|\B\|_{C^\b(\overline{\O})}\bigr\},
\end{equation}
where $C$ depends only on $\O$, $\b$ and $\|f^{-1}\nabla
f\|_{C^{\b}(\overline{\O})}$.

Now $\A$ and $\curl\A$ belong to $C^\b(\overline{\O},\mathbb R^3)$, from
\eqref{eq-divA} $\div\A\in C^\b(\overline{\O})$, and since $\nu\cdot\A=0$ on
$\p\O$, by \eqref{dcg3} we have $\A\in C^{1+\b}(\overline{\O},\mathbb R^3)$.  Applying the
Schauder estimate to \eqref{3.4} we find that $f\in C^{3+\b}(\overline{\O})$, and
\begin{equation}\label{est-B.24}
[D^3f]_{\b}\leq C_\b\Bigl({\k\over\lam}\Bigr)^{3+\b},
\end{equation}
where $C_\b$ depends on $\O$, $\b$ and $\|\A\|_{C^{1,\beta}(\overline{\O})}$. From this and the second equation of \eqref{dec-eqfA-1} we have $\curl\H\in C^{1+\b}(\overline{\O},\mathbb R^3)$.

{\it Step 3.1}. We prove (c.1). Assume $\H_T\in
H^{5/2}(\p\O,\mathbb R^3)$.  Using the first inequality in \eqref{B.6}, the second inequality in \eqref{B.7},   \eqref{curl-H2}, and \eqref{dcg2} with $k=3$ and $p=2$ we get
\begin{equation}\label{B.25}
\|\H\|_{H^3(\O)}\leq C(\O,M)c^{-3}\lam^{-2}\k\{d_1M+\|\mB_T\|_{L^1(\p\O)}\}
+C(\O)\|\mB_T\|_{H^{5/2}(\p\O)}.
\end{equation}
Using Sobolev embedding theorem,  \eqref{dcg1} with $k=3$,  \eqref{B-bd} and \eqref{eq-BH},  we have
$$
\aligned
\|\B\|_{C^{1+\b}(\O)}\leq &C(\O,\b)\|\B\|_{H^3(\O)}\leq C(\O,\b)(\|\H\|_{H^2(\O)}+\|\B\|_{L^2(\O)})\\
\leq &C(\O,\b)\{\|\curl\H\|_{H^1(\O)}+\|\H_T\|_{H^{3/2}(\p\O)}+\|\B\|_{L^2(\O)}\}\\
\leq &C(\O,\b)\lam^{-1}\{d_1 M+\|\mB_T\|_{L^1(\p\O)}+\lam\|\mB_T\|_{H^{3/2}(\p\O)}\},\\
\endaligned
$$
so
$$\aligned
\|\A\|_{C^{1+\b}(\overline{\O})} \leq & \lam^{-1}(\|\B\|_{C^{1+\b}(\overline{\O})}+\|\nabla\phi\|_{C^{1+\b}(\overline{\O})})\\
\leq &
C(\O,M,\b)\lam^{-1}\{d_1^2(c+\lam\k)M +d_1\|\mB_T\|_{L^1(\p\O)}
 +\|\mB_T\|_{H^{3/2}(\p\O)}\}.
\endaligned
$$
From this and \eqref{B.25} we get \eqref{H-H3}, and $\curl\H=-\lam^{-1}f^2\A\in C^{1+\b}(\overline{\O},\mathbb R^3)$.

{\it Step 3.2}. We prove (c.2). Assume  $(f,\A)\in\mK(\O)$ and $\H_T=\mB_T\in
C^{n+\a}(\p\O,\mathbb R^3)$, where $1\leq n\leq m-1$.
Then
$$|\A(x)|^2\leq f^2(x)-1/3\leq 1-1/3=2/3.
$$
Hence
$0\leq M\leq {2/3}$.
Denote $\{\a,\b\}=\min\{\a,\b\}$. Since $\curl\H\in C^{1,\b}(\overline{\O},\Bbb R^3)$, $\div\H=0$ in $\O$ and $\H_T=\mB_T\in
T\!C^{n+\a}(\p\O,\mathbb R^3)$ with $1\leq n\leq m-1$, using \eqref{dcg4} we have
$\H\in C^{1+\{\a,\b\}}(\overline{\O},\mathbb R^3)$. From this and
\eqref{eq-BH} we see that $\B\in C^{2+\{\a,\b\}}(\overline{\O},\mathbb
R^3)$.

Since $\curl\A=\lam^{-1}\H\in C^{1+\{\a,\b\}}(\overline{\O},\mathbb
R^3)$, from \eqref{eq-divA} $\div\A\in C^{1+\b}(\overline{\O})$, and $\nu\cdot\A=0$, using \eqref{dcg3} we have $\A\in
C^{2+\{\a,\b\}}(\overline{\O},\mathbb R^3)$, so $\nabla\phi=\lam\A-\B\in
C^{2+\{\a,\b\}}(\overline{\O},\mathbb R^3)$. Then the right hand side of the equation in
\eqref{eq-phi2} is in $C^{2+\{\a,\b\}}(\overline{\O})$, hence $\phi\in
C^{4+\{\a,\b\}}(\overline{\O})$.
Going back to \eqref{3.4} we
see that $f\in C^{4+\{\a,\b\}}(\overline{\O})$, and from \eqref{eq-curlH}
$\curl\H=-\lam^{-1}f^2\A\in C^{2+\{\a,\b\}}(\overline{\O},\mathbb R^3)$. This together with $\div\H=0$ and $\H_T=\mB_T\in
C^{n+\a}(\p\O,\mathbb R^3)$ implies that
$$
\H\in
C^{\min\{3+\{\a,\b\}, n+\a\}}(\overline{\O},\mathbb R^3).
$$
If $n=1$, this implies $\curl\A=\lam^{-1}\H\in C^{1+\a}(\overline{\O},\Bbb R^3)$. Since $\nu\cdot\A=0$ on $\p\O$, and from \eqref{eq-divA} $\div\A\in C^{2+\{\a,\b\}}(\overline{\O},\Bbb R^3)\subset C^{1+\a}(\overline{\O},\Bbb R^3)$, \eqref{dcg3} yields $\A\in C^{2+\a}(\overline{\O},\Bbb R^3)$.
Then from \eqref{eq-curlH}
$\curl\H=-\lam^{-1}f^2\A\in C^{2+\a}(\overline{\O},\mathbb R^3)$. Going back to \eqref{3.4}
we see that $f\in C^{4+\a}(\overline{\O})$. So \eqref{B.13} holds and  (c.2) is true when $n=1$.
The case where $n\geq 2$ can be proved by iteration.
\end{proof}

\section{Proof of Theorem \ref{Thm-reg-fA}}
\label{AppendixC}

\begin{proof} {\it Step 1.} Let $(f,\A)$ be a weak solutions of \eqref{eqfA}. Then $(f,\A)$ satisfies \eqref{3.4}-\eqref{3.5}, \eqref{eqB.3} and \eqref{est-B.4} hold, $M=\|\A\|_{L^\infty(\O)}<\infty$ and $0<c\leq f\leq 1$. As in Lemma \ref{LemB.2} we have  $f\in C^{1+\a}(\overline{\O})$ for all $\a\in (0,1)$,  and  $\A\in H^1(\O,\Bbb R^3)$.
 Let $\H=\lam\,\curl\A$. Then \eqref{eq-curlH} holds for a.e. $x\in\O$, so $\curl\H=-\lam^{-1}f^2\A\in H^1(\O,\Bbb R^3)\cap L^\infty(\O,\Bbb R^3)$.
From Lemma \ref{Lem3.5}~(iii) we have $\H\in H^1_{\loc}(\Bbb R^3,\Bbb R^3)$, so $\H^-=\H^+=\H$ on $\p\O$ in $H^{1/2}(\p\O,\Bbb R^3)$.
Since $\div\H=0$, from \eqref{eq-B.16} we find that, for $\k\geq\max\{1,\lam\}$,
$$\int_\O(|\H|^2+|D\H|^2)dx=
\int_\O(|\H|^2+\lam^{-2}f^4|\A|^2)dx
\leq\max\{1,\lam^{-2}\}\lam\int_{\p\O}|\A_T\times\H_T| dS.
$$
Using this and $\|\H_T\|_{L^1(\p\O)}\leq C(\O)\|\H\|_{H^1(\O)}$, we get
\begin{equation}\label{C.1}
\aligned
&\|\H\|_{H^1(\O)}\leq
C(\O)(\lam+\lam^{-1})M,\\
&\|\H\|_{L^2(\O)}\leq C(\O)(1+\lam)M.
\endaligned
\end{equation}
Now \eqref{est-3.13} is proved.

We use  \eqref{eqB.3} and \eqref{eq-divA} to control $\|\div\A\|_{L^2(\O)}$, use the second inequality in \eqref{C.1} to control
$\|\curl\A\|_{L^2(\O)}$, and then use \eqref{1.5} and apply \eqref{dcg1} to find that, for $\k\geq \{1,\lam\}$,
\begin{equation}\label{C.2}
\aligned
\|\A\|_{H^1(\O)} \leq& C(\O)\{\|2f^{-1}\nabla f\cdot\A\|_{L^2(\O)}+\lam^{-1}\|\H\|_{L^2(\O)}\}\\
\leq& C(\O)(c^{-1}\k+1+\lam)\lam^{-1}M.
\endaligned
\end{equation}
From \eqref{eq-curlH} we get the first inequality in \eqref{B.7}, and
\begin{equation*}\aligned
\|\curl\H\|_{H^1(\O)}
\leq\lam^{-1}\bigl\{\|f\|_{L^\infty(\O)}^2\|D\A\|_{L^2(\O)}
+2\|f\|_{L^\infty(\O)}\|\nabla
f\|_{L^2(\O)}\|\A\|_{L^\infty(\O)}+\|f^2\A\|_{L^2(\O)}\bigr\}.
\endaligned
\end{equation*}
From this, \eqref{eqB.3} and \eqref{C.2},
\begin{equation}\label{C.3}
\aligned
\|\curl\H\|_{H^1(\O)} \leq
C(\O)c^{-1}\lam^{-2}\k M.
\endaligned
\end{equation}

{\it Step 2}. Assume $\p\O$ is of class $C^3$.
As in Lemma \ref{LemB.2} (b), we have
$f\in H^3(\O)$. From \eqref{1.5}, \eqref{eq-divA}, the first inequality in \eqref{C.1},  \eqref{C.2}, and \eqref{dcg1}, we have $\A\in H^2(\O,\Bbb R^3)$, and
\eq\label{C.4}
\aligned
\|\A\|_{H^2(\O)}\leq& C(\O)\{\lam^{-1}\|\H\|_{H^1(\O)}+\|2f^{-1}\nabla f\cdot\A\|_{H^1(\O)}+\|\A\|_{L^2(\O)}\}\\
\leq& C(\O,M)c^{-4}\lam^{-2}\k^2M.
\endaligned
\end{equation}
Then $\curl\H=-\lam^{-1}f^2\A\in H^2(\O,\Bbb R^3)$.
From \eqref{eq-curlH}, \eqref{est-B.24} and \eqref{C.4} we have
$$
\aligned
\|\curl\H\|_{H^2(\O)}\leq
C(\O,M)c^{-4}\lam^{-3}\k^2M.
\endaligned
$$
If $\H_T\in T\!H^{3/2}(\p\O,\Bbb R^3)$, then $\H\in H^2(\O,\mathbb
R^3)$,  by  \eqref{C.1}, \eqref{C.3}, and \eqref{dcg2},  we have
\eq\label{C.6}
\|\H\|_{H^2(\O)}\leq
C(\O)c^{-1}\lam^{-2}\|\A\|_{C^0(\overline{\O})}\k+C(\O)\|\H_T\|_{H^{3/2}(\p\O)}.
\eeq

{\it Step 3.} Assume $\p\O$ is of class $C^{3+\b}$, $0<\b<1/2$.
As in the proof of Lemma \ref{LemB.2}
we write $\lam\A=\B+\nabla\phi$ on $\O$, where $\B$ satisfies \eqref{eq-BH} and $\phi\in \dot H^1(\O)$ satisfies \eqref{eq-phi}.
Then $\B\in H^2(\O,\Bbb R^3)$ and   \eqref{B-Holder} holds.
From \eqref{eq-phi2} $\phi\in C^{2+\b}(\overline{\O})$ and  \eqref{est-B.23} holds.  Using \eqref{1.5}, as in Lemma \ref{LemB.2} (c) we have $\A\in C^{1+\b}(\overline{\O},\mathbb R^3)$. Then $\phi\in C^{3+\b}(\overline{\O})$ and $\curl\H\in C^{1+\b}(\overline{\O},\mathbb R^3)$.
From \eqref{eqfA}
$f\in C^{3+\b}(\overline{\O})$ and \eqref{est-B.24} holds.
The other estimates of Theorem \ref{Thm-reg-fA} can be derived as in the proof of Lemma \ref{LemB.2}.
\end{proof}

\section{Proof of Proposition \ref{Prop4.7} (i)}\label{AppendixD}

We prove \eqref{est-4.16}. In the 2-dimensional case such estimates have been obtained in \cite{BCM}. We treat the 3-dimensional case. Let $(f_\k,\A_\k)$ denote the solution of \eqref{dec-eqfA-1},
and $\A_\infty$ denote the solution of \eqref{eqA-Omega}-\eqref{cond-1.8} with $\mH_T$ replaced by $\mB_T$:
\eq\label{eqD.1}
\left\{\aligned
-&\lam^2\curl^2\A_\infty=f_\infty^2\A_\infty\q&\text{\rm in }\O,\\
&(\lam\,\curl\A_\infty)^-_T=\mB_T\q&\text{\rm on }\p\O,
\endaligned\right.
\eeq
where $f_\infty(x)=(1-|\A_\infty(x)|^2)^{1/2}$. We write the equation in this form for our late convenience.   From \cite[Theorem 1]{BaP}, $\A_\infty\in C^{2+\a}(\overline{\O},\mathbb R^3)$,
so $f_\infty\in C^{2+\a}(\overline{\O})$. Recall that $\A_\k$ satisfies \eqref{1.5}. From \eqref{cond-3.27} and  \cite[Lemma 2.5]{BaP} we know $\A_\infty$ also satisfies \eqref{1.5}.
Since $\O$ is simply-connected and $\nu\cdot\curl\mB_T=0$ on $\p\O$, so $\mB_T$ has a curl-free extension $\mB\in \mathcal
B^{2+\a}(\overline{\O},\mathbb R^3)$, see Lemma \ref{Lem-extension}.
Write
$$\H_\k=\lam\,\curl\A_\k,\q \mH_\k=\H_\k-\mB,\q
\H_\infty=\lam\,\curl\A_\infty,\q
\mH_\infty=\H_\infty-\mB.
$$
Then $\mH_\k$ and $\mathcal
H_\infty$ satisfy respectively the following systems
\eq\label{eqD.2}
\left\{\aligned
&\lam^2\curl(f_\k^{-2}\curl\mH_\k)+\mH_\k=-\mB\q&\text{\rm in }\O,\\
&\mH_{\k T}=\0\q&\text{\rm on }\p\O,
\endaligned\right.
\eeq
and
\eq\label{eqD.3}
\left\{\aligned
&\lam^2\curl(f_\infty^{-2}\curl\mH_\infty)+\mathcal
H_\infty=-\mB\q&\text{\rm in }\O,\\
& \mH_{\infty
T}=\0\q&\text{\rm on }\p\O.
\endaligned\right.
\eeq
As in general ${\p
f_\infty\over\p\nu}\not\equiv 0$ on $\p\O$, we follow the idea
in \cite{BCM} and approximate $f_\infty$ by a function $\hat f$ which
satisfies the following condition
\begin{equation}\label{cond-D.4}
{\p \hat f\over\p\nu}=0\q\text{\rm on }\p\O.
\end{equation}
In the following we denote
$$N(\hat f)=\max\{\|\hat f\|_{L^\infty(\O)},1\}.
$$

\begin{Lem} Assume the conditions in Proposition \ref{Prop4.7}
and $\k\geq \max\{1,\lam\}$. Let
$\hat f$ be a function satisfying \eqref{cond-D.4} such that $(\hat
f,\A_\infty)\in \mathbb K_\delta(\O)$. Define
\begin{equation}\label{D.5}
\bar f_\k=f_\k-\hat f,\q \tilde f_\k=f_\k-f_\infty,\q \bar \A_\k=\A_\k-\A_\infty,\q \bar\mH_\k=\lam\,\curl\bar \A_\k.
\end{equation}
We have
\begin{equation}\label{est-D.6}
\aligned
&{\lam^2\over\k^2}\|\nabla\bar f_\k\|_{L^2(\O)}^2+\|\bar f_\k\|_{L^2(\O)}^2+\lam^2\|\curl\bar \A_\k\|_{L^2(\O)}^2+\|\bar \A_\k\|_{L^2(\O)}^2\\
\leq & C_1\Bigl\{{\lam^4\over\k^4}\|\Delta\hat f\|_{L^2(\O)}^2
+\|\hat f-f_\infty\|_{L^2(\O)}^2 \Bigr\},
\endaligned
\end{equation}
and
\begin{equation}\label{est-D.7}
\aligned \|\bar {\mH}_\k\|_{H^1(\O)}\leq &C_2\bigl\{
\lam^{-1}\|\tilde f_\k\|_{L^2(\O)}\|\mH_\k+\mathcal
B\|_{L^\infty(\O)}\\
&\qq+\|\nabla \tilde f_\k\|_{L^2(\O)}+\|\tilde f_\k\|_{L^2(\O)}\|\nabla
f_\infty\|_{L^\infty(\O)}\bigr\},\\
\|\bar {\mH}_\k\|_{H^2(\O)}\leq
&C_2\|f_\infty\|_{C^1(\overline{\O})}\bigl\{
\lam^{-1}\|\tilde f_\k\|_{L^2(\O)}\|\mH_\k+\mathcal
B\|_{L^\infty(\O)}\\
&\qq\qq\qq+\|\nabla \tilde f_\k\|_{L^2(\O)}+\|\tilde f_\k\|_{L^2(\O)}\|\nabla
f_\infty\|_{L^\infty(\O)}\bigr\},
\endaligned
\end{equation}
where $C_1=C_1(\O,\delta, N(\hat f))$ and $C_2=C_2(\O,\delta).$
\end{Lem}

\begin{proof} {\it Step 1}. We follow the ideas in \cite{BCM} to prove \eqref{est-D.6}. Let $G(f,\A)$ be the function defined in \eqref{funct-E-Omega}.
Since $(f_\k,\A_\k)$ satisfies \eqref{dec-eqfA-1} and
$G'_f(f_\infty,\A_\infty)=0$, we have
\begin{equation}\label{eqD.8}
\left\{\aligned
-&{\lam^2\over\k^2}\Delta \bar f_\k+{1\over
2}[G'_f(f_\k,\A_\k)-G'_f(\hat f,\A_\infty)]\\
&\qq +{1\over 2}[G'_f(\hat
f,\A_\infty)-G'_f(f_\infty,\A_\infty)]={\lam^2\over\k^2}\Delta\hat
f\q&\text{in }\O,\\
&\lam^2\curl^2\bar \A_\k+{1\over 2}[G_{\A}'(f_\k,\A_\k)-G_{\A}'(\hat
f,\A_\infty)]\\
&\qq +{1\over 2}[G_{\A}'(\hat f,\A_\infty)-G_{\A}'(f_\infty,\A_\infty)]
=\0\q&\text{in }\O,\\
&{\p \bar f_\k\over\p\nu}=0,\q (\curl\bar
\A_\k)_T=\0\q&\text{on }\p\O.
\endaligned
\right.
\end{equation}
So we get
$$\aligned
&{\lam^2\over\k^2}\|\nabla\bar f_\k\|_{L^2(\O)}^2
+\lam^2\|\curl\bar \A_\k\|_{L^2(\O)}^2\\
&+{1\over
2}\int_\O\Big\{[G'_f(f_\k,\A_\k)-G'_f(\hat f,\A_\infty)]\bar f_\k+[G_{\A}'(f_\k,\A_\k)-G_{\A}'(\hat
f,\A_\infty)]\cdot\bar\A_\k\Big\}dx\\
=&\int_\O\Big\{{\lam^2\over\k^2}(\Delta\hat
f)\bar f_\k + {1\over 2}[G'_f(f_\infty,\A_\infty)-G'_f(\hat f,\A_\infty)]\bar f_\k+{1\over 2}[G_{\A}'(f_\infty,\A_\infty)-G_{\A}'(\hat f,\A_\infty)]\cdot\bar\A_\k\Big\}dx.
\endaligned
$$
Applying Remark \ref{Rem4.1}~(b) to  $(f_0,\A_0)=(\hat f,\A_\infty)$ and $(f_1,\A_1)=(f_\k,\A_\k)$ we get
$$\aligned
&\int_\O\Big\{[G'_f(f_\k,\A_\k)-G'_f(\hat f,\A_\infty)]\bar f_\k+[G_{\A}'(f_\k,\A_\k)-G_{\A}'(\hat
f,\A_\infty)]\cdot\bar\A_\k\Big\}dx\\
\geq& C(\delta)(\|\bar f_\k\|_{L^2(\O)}^2+\|\bar\A_\k\|_{L^2(\O)}^2).
\endaligned
$$
Using the facts $|\A_\infty(x)|\leq 1/\sqrt{3}$ and $|f_\infty+\hat
f|\leq 2N(\hat f)$, we get
$$\aligned
&\int_\O[G'_f(f_\infty,\A_\infty)-G'_f(\hat f,\A_\infty)]\bar f_\k\,dx\leq C(\var)\|\hat f-f_\infty\|_{L^2(\O)}^2+\var\|\bar f_\k\|_{L^2(\O)}^2,\\
&\int_\O[G_{\A}'(f_\infty,\A_\infty)-G_{\A}'(\hat f,\A_\infty)]\cdot\bar\A_\k dx
\leq C(\var)\|\hat f-f_\infty\|_{L^2(\O)}^2+\var\|\bar\A_\k\|_{L^2(\O)}^2.
\endaligned
$$
By choosing $\var$ suitably small in these two inequalities, and summarizing the above computations, we get \eqref{est-D.6}.

{\it Step 2}. From \eqref{eqD.2} and \eqref{eqD.3} we have
\begin{equation}\label{eqD.9}
\left\{\aligned
&\lam^2\curl(f_\infty^{-2}\curl\bar{\mH}_\k)+\bar{\mathcal
H}_\k=\lam^2\mF_\k\q&\text{\rm in }\O,\\
& \bar{\mH}_{\k
T}=\0\q&\text{\rm on }\p\O,
\endaligned\right.
\end{equation}
where
$$\mF_\k=\curl[(f_\infty^{-2}-f_\k^{-2})\curl\mH_\k].
$$
We claim that
\begin{equation}\label{est-D.10}
\aligned
\|\mF_\k\|_{L^p(\O)}\leq&
3\lam^{-2}\|\tilde f_\k\|_{L^p(\O)}\|\mathcal
H_\k+\mB\|_{L^\infty(\O)}\\
&+2^{-3/2}9\lam^{-1}\{\|\nabla \tilde f_\k\|_{L^p(\O)}+\|\tilde f_\k\|_{L^p(\O)}\|\nabla
f_\infty\|_{L^\infty(\O)}\}.
\endaligned
\end{equation}
To verify, we compute
$$\aligned
\curl^2\mH_\k=&-\lam^{-2}f_\k^2(\mH_\k+\mathcal
B)+2f_\k^{-1}\nabla
f_\k\times\curl\mH_\k,\\
\mF_\k =&\lam^{-2}f_\infty^{-2}(f_\infty^2-f_\k^2)(\mathcal
H_\k+\mB)+2f_\infty^{-3}f_\k^{-1}(f_\infty\nabla
f_\k-f_\k\nabla f_\infty)\times\curl\mH_\k.
\endaligned
$$
Since $\curl\mB=\0$, we have $\curl\mH_\k=\curl\H_\k$. Since $(f_\k,\H_\k)\in\mathbb U_\delta(\O)$
and $1/\sqrt{3}\leq f_\k\leq 1$ (see Proposition \ref{Prop4.3}), we have
$$|\curl\mH_\k|=|\curl\H_\k|=|\lam^{-1}f_\k^2\A_\k|\leq \lam^{-1}f_\k^2\leq
\lam^{-1}.
$$
We also have $(f_\infty,\H_\infty)\in\mathbb U_\delta(\O)$ and
$\sqrt{2/3}\leq f_\infty\leq 1$. So \eqref{est-D.10} holds.

Now we apply Lemma \ref{Lem-linear-curl} (i) to \eqref{eqD.9} with
$a=\lam^2f_\infty^{-2}$. Since
$$\aligned
&m=\min\Big\{1,\lam^2\min_{x\in\overline{\O}}f_\infty^{-2}(x)\Big\}\geq
\lam^2,\q
\|a\|_{C^1(\overline{\O})}\leq 2\Big({3\over
2}\Big)^{3/2}\lam^2\|f_\infty\|_{C^1(\overline{\O})},
\endaligned
$$
from \eqref{L-est1} we have, for $0<\lam\leq\k$,
$$\aligned
& \|\bar{\mH}_\k\|_{H^1(\O)}\leq C(\O)\lam\|\mathcal
F_\k\|_{L^2(\O)},\q \|\bar{\mH}_\k\|_{H^2(\O)}\leq
C(\O)\|f_\infty\|_{C^1(\overline{\O})}\lam\|\mF_\k\|_{L^2(\O)}.
\endaligned
$$
Combining these with \eqref{est-D.10} we get \eqref{est-D.7}. \end{proof}

\begin{Lem}\label{LemD.2}
Assume the conditions of Proposition \ref{Prop4.7} and $\k\geq\max\{1,\lam\}$. We can
choose the function $\hat f$ such that the $\bar f_\k$, $\bar\A_\k$ and $\bar{\mH}_\k$ defined in
\eqref{D.5} satisfy the following inequalities:
\begin{equation}\label{est-D.11}
\aligned
&\|\bar f_\k\|_{L^2(\O)}\leq C\k^{-3/2},\qqq\qq\;
\|\nabla\bar f_\k\|_{L^2(\O)}\leq C(1+\lam^{-1})\k^{-1/2},\\
&\|\Delta\bar f_\k\|_{L^2(\O)}\leq C(1+\lam^{-2})\k^{1/2},\qq \|D^2\bar
f_\k\|_{L^2(\O)}\leq C(1+\lam^{-2})\k^{1/2},
\endaligned
\end{equation}
\begin{equation}\label{est-D.12}
\aligned
&\|\bar \A_\k\|_{L^2(\O)}\leq C\k^{-3/2},\\
\endaligned
\end{equation}
\begin{equation}\label{est-D.13}
\aligned
\|\bar{\mH}_\k\|_{L^2(\O)}\leq &C\k^{-3/2},\qqq
\|\bar{\mH}_\k\|_{H^1(\O)}\leq Cb_1\lam^{-3}\k^{-1/2},\\
\|\bar{\mH}_\k\|_{H^2(\O)}\leq &Cb_2\lam^{-3}\k^{-1/2},
\endaligned
\end{equation}
where
$$C=C(\O,\delta,N(\hat f),f_\infty),\q
b_1=\lam^2+ \|\mB\|_{C^0(\overline{\O})} +\lam^2\k^{-1}\|\nabla
f_\infty\|_{L^\infty(\O)},\q b_2=b_1\|f_\infty\|_{C^1(\overline{\O})}.
$$
\end{Lem}
\begin{proof} Choose the function $\hat f$ as in
\cite[Lemma 4.5]{BCM}:
$$
\hat f(x)=f_\infty(x)-\chi(\k
d(x))d(x){\p f_\infty\over\p\nu}(y_x),
$$
where $d(x)=\text{dist}(x,\p\O)$, $y_x\in\p\O$ being such that
$|y_x-x|=d(x)$, and $\chi(t)$ is a smooth and non-increasing
function in $t$ such that $\chi(t)=1$ for $0\leq t\leq 1$ and
$\chi(t)=0$ for $t>2$. Then, as $\k\to\infty$ it holds that
\begin{equation}\label{est-D.14}
\aligned &\|\hat f-f_\infty\|_{L^2(\O)}\leq C_1\k^{-3/2},\qq\q
\|\nabla\hat f-\nabla f_\infty\|_{L^2(\O)}\leq C_2\k^{-1/2},\\
&\|\Delta \hat f-\Delta f_\infty\|_{L^2(\O)}\leq C_3\k^{1/2},\qq
\|D^2\hat f-D^2f_\infty\|_{L^2(\O)}\leq C_4\k^{1/2},\\
&\|\hat f-f_\infty\|_{C^\a(\overline{\O})}\leq C_5\k^{-(1-\a)},\\
\endaligned
\end{equation}
where $C_1$ and $C_2$ depend on $\|f_\infty\|_{C^1(\overline{\O})}$, $C_3$
depends on $\|f_\infty\|_{C^2(\overline{\O})}$, $C_4$ and $C_5$ depend on
$\|f_\infty\|_{C^3(\overline{\O})}$.\footnote{Note that $\A_\infty$ and $f_\infty$ depend on $\lam$. The analysis in \cite{BaP} shows
that if $\mB_T\in C^{2+\a}(\p\O,\Bbb R^3)$, then
$\|\A_\infty\|_{C^{2+\a}(\overline{\O})}\leq C\lam^{-2-\a}$. Hence
$\|D^{n+\a}f_\infty\|_{C^0(\O)}\leq C\lam^{-n-\a}$ for $n=0,1,2$.}
From \eqref{est-D.6} and \eqref{est-D.14} we find
$$
{\lam^2\over\k^2}\|\nabla\bar f_\k\|_{L^2(\O)}^2+\|\bar
f_\k\|_{L^2(\O)}^2+\lam^2\|\curl\bar \A_\k\|_{L^2(\O)}^2+\|\bar
\A_\k\|_{L^2(\O)}^2\leq C\k^{-3},
$$
where $C$ depends on $\O$, $\delta$, $N(\hat f)$ and $f_\infty$. From this inequality
we get the first two inequalities in \eqref{est-D.11},
\eqref{est-D.12}, and the first inequality in \eqref{est-D.13}.

Since $\tilde f_\k=\bar f_\k+(\hat f-f_\infty)$, using the first two inequalities in \eqref{est-D.11}, and \eqref{est-D.14}, we find
\begin{equation}\label{est-D.15}
\aligned
\|\tilde f_\k\|_{L^2(\O)}\leq C\k^{-3/2},\q
\|\nabla \tilde f_\k\|_{L^2(\O)}\leq C(1+\lam^{-1})\k^{-1/2},
\endaligned
\end{equation}
where $C=C(\O,N(\hat f),\delta,f_\infty)$.

From the first equation in \eqref{eqD.8} we have
\begin{equation}\label{est-D.16}
{\lam^2\over\k^2}\|\Delta\bar f_\k\|_{L^2(\O)}\leq\|(|f_\k|^2+|\A_\k|^2-1)f_\k\|_{L^2(\O)}+{\lam^2\over\k^2}\|\Delta
\hat f\|_{L^2(\O)}.
\end{equation}
To estimate the first term in the right side of \eqref{est-D.16} we note that
$(|f_\infty|^2+|\A_\infty|^2-1)f_\infty=0$,
$|\A_\k|\leq f_\k\leq 1$ and $|\A_\infty|\leq f_\infty\leq 1$,
so
$$\aligned
|(|f_\k|^2+|\A_\k|^2-1)f_\k|=&|(|f_\k|^2+|\A_\k|^2-1)f_\k-(|f_\infty|^2+|\A_\infty|^2-1)f_\infty|\\
\leq& 3|\tilde f_\k|+2|\bar
\A_\k|.
\endaligned
$$
To estimate the second term, we use the third inequality in \eqref{est-D.14} to get, for $\k\geq 1$,
$$\|\Delta\hat f\|_{L^2(\O)}\leq C_3\k^{1/2}+\|\Delta f_\infty\|_{L^2(\O)}\leq (C_3+\|f_\infty\|_{C^2(\overline\O)})\k^{1/2}.
$$
Combining the above computations, and using \eqref{est-D.12}, \eqref{est-D.16},  the fourth inequality in \eqref{est-D.14} and the first inequality in \eqref{est-D.15}, we find
\begin{equation*}
\|\Delta \bar f_\k\|_{L^2(\O)}\leq \|\Delta \hat
f\|_{L^2(\O)}+C\lam^{-2}\k^{1/2}\leq C(1+\lam^{-2})\k^{1/2},
\end{equation*}
where $C$ depends on $\O$, $\delta$, $N(\hat f)$ and $f_\infty$. Therefore the third inequality in \eqref{est-D.11} is true. From this
and since ${\p\bar f_\k\over\p\nu}=0$ on $\p\O$, we have
$$
\|D^2\bar f_\k\|_{L^2(\O)}\leq C(\O)\{\|\Delta\bar f_\k\|_{L^2(\O)}+\|\bar f_\k\|_{L^2(\O)}\}\leq
C(1+\lam^{-2})\k^{1/2},
$$
where $C$ depends on $\O$, $\delta$, $N(\hat f)$ and $f_\infty$.
So the fourth inequality in \eqref{est-D.11} is true.

Now we apply Lemma \ref{Lem-linear-curl} (i)  with $a=\lam^2f_\k^{-2}$ to \eqref{eqD.2}.
For $0<\lam\leq\k$, since $f_\k$ satisfies \eqref{est-3.17},  we have
$$\aligned
&m=\min\{1,\lam^2\min_{x\in\overline{\O}}f_\k^{-2}(x)\}\geq \lam^2,\q
\|a\|_{C^1(\overline{\O})}\leq
2\cdot3^{3/2}\lam^2\|f_\k\|_{C^1(\overline{\O})}\leq C(\O)\lam\k.
\endaligned
$$
So we use \eqref{L-est1} to get
\begin{equation}\label{est-D.17}
\aligned
\|\mH_\k\|_{L^\infty(\O)}\leq C(\O)\|\mathcal
H_\k\|_{H^2(\O)}\leq
C(\O)m_a^{-3/2}\|a\|_{C^1(\overline{\O})}\|\mB\|_{L^2(\O)}
\leq  C(\O)\lam^{-2}\k\|\mB\|_{L^\infty(\O)}.
\endaligned
\end{equation}
Plugging \eqref{est-D.15} into the first inequality in \eqref{est-D.7} and using \eqref{est-D.17} we find
\begin{equation*}\aligned
\|\bar{\mH}_\k\|_{H^1(\O)}
\leq &C\lam^{-1}\k^{-1/2}\bigl\{1+
C(\O)\lam^{-2}\|\mB\|_{L^\infty(\O)} +\k^{-1}\|\nabla
f_\infty\|_{L^\infty(\O)}\bigr\},
\endaligned
\end{equation*}
where $C$ depends on $\O$, $\delta$, $N(\hat f)$ and $f_\infty$.
So we get the second inequality in
\eqref{est-D.13}.

Similarly we get the third inequality in \eqref{est-D.13}. \end{proof}

\begin{Lem}\label{LemD.3}
Assume the conditions of Proposition \ref{Prop4.7} and $\k\geq\max\{1,\lam\}$.
We  have
\begin{equation}\label{est-D.18}
\aligned
\|\bar\A_\k\|_{H^1(\O)}\leq& Cb_3(1+\lam^{-1})\k^{-1/2},\\
\|\bar\A_\k\|_{H^2(\O)}\leq& Cb_4\lam^{-3}(1+\lam^{-2})\k^{1/2},
\endaligned
\end{equation}
where $C=C(\O,\delta,N(\hat f),f_\infty)$ and
$$\aligned
 b_3=&\lam\k^{-1}\|\nabla f_\infty\|_{L^\infty(\O)}+1,\\
b_4=&\lam^3(3+ \k^{-2}\|\nabla f_\infty\|_{L^\infty(\O)})
+\lam^4\k^{-1} (2\|\nabla f_\infty\|_{L^\infty(\O)}
+\|D\A_\infty\|_{L^\infty(\O)})\\
&+\lam^5\k^{-2}\{1+\|\nabla f_\infty\|_{L^\infty(\O)}(\|\nabla
f_\infty\|_{L^\infty(\O)}+\|D\A_\infty\|_{L^\infty(\O)})\}+\k^{-1}\|\mathcal
B\|_{C^0(\overline{\O})}.
\endaligned
$$
\end{Lem}
\begin{proof} Recall that
$\nu\cdot\bar\A_\k=0$ on $\p\O$. If $0<\lam\leq\k$, we use  the last equality in \eqref{D.5},  \eqref{est-D.12}, the first
inequality in \eqref{est-D.13} and \eqref{dcg1} to get
\begin{equation}\label{est-D.19}
\aligned
\|\bar\A_\k\|_{H^1(\O)}\leq
C\{\|\div\bar\A_\k\|_{L^2(\O)}+\lam^{-1}\k^{-3/2}\},
\endaligned
\end{equation}
where $C=C(\O,\delta,M,f_\infty)$.
From the second equation of \eqref{dec-eqfA-1} we have
$\div(f_k^2\A_\k)=0$, and from the first equation of \eqref{eqD.1} we
have $\div(f_\infty^2\A_\infty)=0$. Hence
\begin{equation}\label{eqD.20}
\aligned
\div\bar\A_\k={2\over f_\k f_\infty}\tilde f_\k\nabla
f_\infty\cdot\A_\infty-{2\over
f_\k}(\nabla f_\infty\cdot\bar\A_\k+\nabla \tilde f_k\cdot\A_\k).
\endaligned
\end{equation}
Since $1/\sqrt{3}\leq f_\k\leq 1$, $1/\sqrt{3}\leq f_\infty\leq 1$
and $|\A_\infty|\leq f_\infty$, we use \eqref{est-D.12} and \eqref{est-D.15}
to find
$$\aligned
\|\div\bar\A_\k\|_{L^2(\O)}\leq& 2\sqrt{3}\|\nabla
f_\infty\|_{L^\infty(\O)}\{\|\tilde f_\k\|_{L^2(\O)}
+\|\bar\A_\k\|_{L^2(\O)}\}+2\sqrt{3}\|\nabla\tilde f_\k\|_{L^2(\O)}\\
\leq &C\|\nabla
f_\infty\|_{L^\infty(\O)}\k^{-3/2}+C(1+\lam^{-1})\k^{-1/2}.
\endaligned
$$
From this and \eqref{est-D.19}, the first inequality in \eqref{est-D.18} is
true for all $\k\geq\max\{1,\lam\}$.

Using Lemma 2.1 (i), \eqref{est-D.12}, and the second inequality in
\eqref{est-D.13} we have
 \begin{equation}\label{est-D.21}
\aligned \|\bar\A_\k\|_{H^2(\O)}\leq
C\{\|\div\bar\A_\k\|_{H^1(\O)}+b_1\lam^{-4}\k^{-1/2}+\k^{-3/2}\}.
\endaligned
\end{equation}
We use \eqref{eqD.20} to compute $\p_j\div\bar\A_\k$, then use \eqref{est-D.12}, \eqref{est-D.15}, the first inequality in \eqref{est-D.18}, and use the facts $1/\sqrt{3}\leq f_\k,\; f_\infty\leq 1$, $|\A_\k|\leq f_\k$,
$|\A_\infty|\leq f_\infty$, $|\nabla f_\k(x)|\leq C\lam^{-1}\k$,
$|D^2f_\k(x)|\leq C\lam^{-2}\k^2$ to get
$$
\|\p_j\div\bar\A_j\|_{L^2(\O)} \leq C\b_1(1+\lam^{-2})\k^{1/2}.
$$
Plugging this back to \eqref{est-D.21} we get
$$\|\bar\A_\k\|_{H^2(\O)}\leq C\b_2(1+\lam^{-2})\lam^{-2}\k^{1/2},
$$
where
$$\aligned
\b_1=&1+b_3+\lam\k^{-1}(\|\nabla
f_\infty\|_{L^\infty(\O)}+\|D\A_\infty\|_{L^\infty(\O)})\\
&+\lam^2\k^{-2}\|\nabla f_\infty\|_{L^\infty(\O)}(\|\nabla
f_\infty\|_{L^\infty(\O)}+\|D\A_\infty\|_{L^\infty(\O)}),\\
\b_2=&\b_1\lam^2+b_1\k^{-1}+\lam^4\k^{-2}.
\endaligned
$$
When $\k\geq \max\{1,\lam\}$ we have $\b_2\leq b_4\lam^{-1}$. So the second
inequality in \eqref{est-D.18} holds.
\end{proof}

\begin{proof}[\bf Proof of part (i) of Proposition \ref{Prop4.7}]\
$f_\infty$ depends only on $\O$, $\lam$ and $\mB_T$, and $N(\hat f)$ can be constructed to depend only on $\O$ and $f_\infty$. Hence the estimate \eqref{est-4.16} follows from Lemma \ref{LemD.2}, Lemma \ref{LemD.3} and \eqref{est-D.15}.
 \end{proof}

\section{Additional Remarks}\label{AppendixE}

\subsection{Remarks on requirement \eqref{cond-5.9}}\

We consider a problem slightly more general than \eqref{cond-5.9}:

\begin{Prob} Find conditions on $g$ such that the following equation has a solution $\mA_T\in T\!C^{1+\a}(\p\O,\Bbb R^3)$:
\eq\label{eq-g}
\nu\cdot\curl\mA_T=g\q\text{\rm on }\p\O.
\eeq
\end{Prob}

We shall see that the solvability of \eqref{eq-g} depends on both the topology of $\O$ and $g$.
Let us denote the connected components of $\p\O$ by $\Gamma_j$, $j=1,\cdots, m+1$, where $m\geq 0$, and
$\Gamma_{m+1}$ denotes the boundary of the infinite connected component of $\O^c$.

\begin{Prop}\label{PropE.2}
Let $\O$ be a bounded domain in $\Bbb R^3$ with a $C^{2+\a}$ boundary, $0<\a<1$.
\begin{itemize}
\item[(i)] Let $g\in C^\a(\p\O)$. Then \eqref{eq-g} has a solution $\mA_T\in T\!C^{1+\a}(\p\O,\Bbb R^3)$ if and only if
\eq\label{cond-g}
\int_{\Gamma_j} g\, dS=0,\q j=1,\cdots, m+1.
\eeq
\item[(ii)] Let $\mH^e\in C^\a(\p\O,\Bbb R^3)$ and let $\phi_0\in C^{2+\a}$ be a solution of \eqref{eqphi-mu} associated with $\mu=0$. Then \eqref{cond-5.9} holds for some $\mA_T\in T\!C^{1+\a}(\p\O,\Bbb R^3)$ if and only if $g=\nu\cdot\mH^e$ satisfies \eqref{cond-g}. In particular, if $\O$ has no holes, then this condition is exactly \eqref{cond-vanish}.
\end{itemize}
\end{Prop}

\begin{proof} We only need to prove (i). Let
$$\Bbb H_2(\O)=\{\w\in L^2(\O,\Bbb R^3):~ \curl\w=\0\;\;\text{and}\;\; \div\w=0\;\;\text{in }\O,\;\; \u_T=\0\;\;\text{on }\p\O\}.
$$
Then $\dim \Bbb H_2(\O)=m$.
Assume \eqref{eq-g} has a solution $\mA_T\in T\!C^{1+\a}(\p\O,\Bbb R^3)$ and
let $\A\in C^{1+\a}(\overline{\O},\Bbb R^3)$ be a divergence-free extension of $\mA_T$, see \cite{P4}. Then
$$\int_{\p\O} g\,dS=\int_{\p\O}\nu\cdot\curl\mA_T dS=\int_{\p\O}\nu\cdot\curl\A\, dS=\int_\O\div\curl\A\, dx=0.
$$
This gives \eqref{cond-g} when $m=0$ (namely when $\O$ has no holes). If $m>0$, then $\Bbb H_2(\O)$ has a basis $\{\nabla q_i\}_{i=1}^m$, where $q_i$ is a harmonic function in $\O$ and $q_i=\delta_{ij}$ on $\Gamma_j$ for $1\leq j\leq m+1$. Then we have
$$\aligned
\int_{\Gamma_j}g\, dS=&\int_{\p\O}q_j\nu\cdot\curl\A dS=\int_{\O}\div(q_j\curl\A)dx
=\int_\O \nabla q_j\cdot\curl\A dx\\
=&\int_\O\div(\A\times\nabla q_j)dx=\int_{\p\O}\nu\cdot[\A\times(\nabla q_j)_T]dS=0.
\endaligned
$$
Hence \eqref{cond-g} is a necessary condition for \eqref{eq-g} to have a solution.

To prove \eqref{cond-g} is also sufficient, let us recall that (see \cite[P.223, Proposition 3]{DaL3})
$$
\curl [H^1_{n0}(\O,\Bbb R^3)]=\mH^\Gamma(\O,\div0),\q \mH(\O,\div0)=\mH^\Gamma(\O,\div0)\oplus\Bbb H_2(\O),
$$
where
$$\mH^\Gamma(\O,\div0)=\{\w\in \mH(\O,\div0):~ \langle 1,\nu\cdot\w\rangle_{H^{1/2}(\Gamma_j),H^{-1/2}(\Gamma_j)}=0\;\;\text{for } j=1,\cdots, m+1\}.
$$
Assume $g\in C^\a(\p\O)$ satisfies \eqref{cond-g}. Let $\psi$ be a harmonic function in $\O$ satisfying
${\p\psi\over\p\nu}=g$ on $\p\O$.
Then
$$\nabla\psi\in C^\a(\overline{\O},\Bbb R^3)\cap \mH^\Gamma(\O,\div0)=C^\a(\overline{\O},\Bbb R^3)\cap \curl [H^1_{n0}(\O,\Bbb R^3)].
$$
So there exists $\A\in H^1_{n0}(\O,\div0)$ such that
$\curl\A=\nabla\psi$ in $\O$.
From \eqref{dcg3} we have $\A\in C^{1+\a}(\p\O,\Bbb R^3)$. Let $\mA_T=\A_T$. Then
$\nu\cdot\curl\mA_T={\p\psi\over\p\nu}=g$.
\end{proof}

Assume $\O$ satisfies  $(O)$. Proposition \ref{PropE.2} says that for any $\mB_T\in T\!C^{1+\a}(\p\O,\Bbb R^3)$ satisfying \eqref{cond-5.2}, and for any  $\mH^e\in C^\a(\p\O,\Bbb R^3)$ satisfying \eqref{cond-vanish},  there always exists $\mA_T\in T\!C^{1+\a}(\p\O,\Bbb R^3)$ such that \eqref{cond-5.9} holds. Therefore for $\O$, $\mB_T$ and $\mH^e$ satisfying the conditions in Theorem \ref{Thm5.3}, we can always find $\mA_T$ such that there exists $\A$ satisfying \eqref{dec-eqfA-2}-\eqref{1.4} except the boundary condition $\A_T^+=\mA_T$.

\subsection{Remarks on condition \eqref{cond-6.15}}\

As observed in Proposition \ref{Prop6.5}, condition \eqref{cond-6.15} is necessary for $\mH$ to produce a solution of \eqref{eqA}. Thus it is important to classify  vector fields satisfying \eqref{cond-6.15}.
We start with the following two problems
\eq\label{eq-w}
\left\{\aligned
&\curl\w=\0,\q \div\w=0\q&\text{in }\O^c,\\
&\w_T=\bold v\q&\text{on }\p\O,\\
&\lim_{|x|\to\infty}\w(x)=\0,\q \int_{\p\O}\nu\cdot\w dS=0,
\endaligned\right.
\eeq
and
\eq\label{eq-u}
\left\{\aligned
&\curl\u=\0,\q \div\u=0\q&\text{in }\O^c,\\
&\nu\cdot\u=g\q&\text{on }\p\O,\\
&\lim_{|x|\to\infty}\u(x)=\0.
\endaligned\right.
\eeq

\begin{Lem}\label{LemE.3} Assume $\O$ satisfies $(O)$ with $r\geq 3$ and $0<\a<1$.
\begin{itemize}
\item[(i)] For any $\bold v\in \mB^{2+\a}(\p\O)$, \eqref{eq-w} has a unique solution $\w$, and
\eq\label{est-w}
\w \in C^{2+\a}_{\loc}(\overline{\O^c},\Bbb R^3)\cap C^\a(\overline{\O^c},\Bbb R^3),\qq
\|\w\|_{C^{2+\a}(\overline{\O^c})}\leq C_1(\O,\a)\|\bold v\|_{C^{2+a}(\p\O)}.
\eeq
\item[(ii)] For any $g\in \dot C^{2+\a}(\p\O)$,  \eqref{eq-u} has a unique solution $\u$, and
$$
\u \in C^{2+\a}_{\loc}(\overline{\O^c},\Bbb R^3)\cap C^\a(\overline{\O^c},\Bbb R^3),\qq\|\u\|_{C^{2+\a}(\overline{\O^c})}\leq C_2(\O,\a)\|g\|_{C^{2+a}(\p\O)}.
$$
\end{itemize}
\end{Lem}
\begin{proof} Existence of a unique solution of \eqref{eq-w} follows from  \cite[Theorem 3.3 (b)]{NW}, and  that of \eqref{eq-u} follows from \cite[Theorem 3.2]{NW}.
The $C^{2+\a}$ estimate can be obtained using the integral equations given in the proof of Theorem 4.2 in \cite{NW}.
\end{proof}

The vector fields $\mH$ satisfying \eqref{cond-6.10} and the vector fields $\mW(\bv)$ satisfying \eqref{cond-6.28} are connected by the relation $$\mH=\mH^e+\mW(\bv).
$$

\begin{Def}\label{DefE.4} Assume $\O$ satisfies $(O)$ with $r\geq 3$ and $0<\a<1$, and $\mH^e$ satisfies $(H)$. Let $\bold v\in \mB^{2+\a}(\p\O)$ satisfy \eqref{cond-6.28}. We denote by $\mW(\bold v)$ the unique solution $\w$ of \eqref{eq-w} with boundary date $\bold v$ and define
\eq
\Sigma(\bold v)=\nu\cdot\mW(\bold v)^+,
\eeq
where $\nu$ is the unit outer normal vector to $\p\O$. We also define
\eq
\Gamma(\lam,\bold v)=\nu\cdot[\mS(\lam,(\mH^e_T)^+ +\bold v)]^-,
\eeq
where  $\mS(\lam,(\mH^e_T)^++\bold v)$ is the solution $\H_\O$ of \eqref{eq6.2} with boundary data $(\H_\O)_T^-=(\mH^e_T)^++\bold v$ as defined in Definition \ref{Def6.4}.

With these maps, the comparability condition \eqref{cond-6.15} can be written as
\eq\label{solvability-1}
\Gamma(\lam,\bold v)=\nu\cdot(\mH^e)^+ +\Sigma(\bold v)\q\text{\rm on }\p\O,
\eeq
and the solvability condition \eqref{eq6.33} can be written as
\eq
\Gamma(\lam, \bold v)=\nu\cdot(\mH^e)^+ +{\p\phi_{\bold v,0}\over\p\nu}\q\text{\rm on }
\p\O,
\eeq
where $\phi_{\bold v,0}$ is the solution of \eqref{eq6.29} for this $\bold v$ and for $\mu=0$.
\end{Def}

Under the assumptions in the above definition, from Lemma \ref{LemE.3} we know that $\w=\mW(\bv)$ exists and is unique, hence the operators $\Sigma(\bv)$ is well-defined.  Since $\bold v\in \mB^{2+\a}(\p\O)$ satisfies  \eqref{cond-6.28}, so $\mH=\mH^e+\w$ satisfies \eqref{cond-6.10} in $\O^c$, and $\Sigma(\bv)$ is well-defined. From \cite[Theorem 1]{BaP} we know that the solution $\H_\O$ of \eqref{eq6.2} in $\O$ with boundary condition $(\H_\O)_T=\mH_T$ exists and is unique, so $\Gamma(\lam,\bv)$ is well-defined. If the condition \eqref{solvability-1} holds, then $\mH$ satisfies \eqref{cond-6.15}.

Note that $\Sigma$ maps the tangential component $\w_T^+=\bold v$ of the solution $\w=\mW(\bv)$ of \eqref{eq-w} to the normal component $\nu\cdot\w^+$, and
$\Gamma$ maps the tangential component $(\H_\O)_T^-=(\mH^e_T)^++\bold v$ of a solution $\H_\O$ of \eqref{eq6.2} to its normal component $\nu\cdot\H_\O^-=\nu\cdot[\mS(\lam,(\mH^e_T)^++\bold v)]^-$.
Thus both operators may be called the \emph{Dirichlet-to-Neumann maps}, while $\Sigma$ is a linear operator with respect to an equation in the exterior domain, $\Gamma(\lam,\cdot)$ is a nonlinear operator related to an equation in $\O$. Recall that
$$
\int_{\p\O}\Gamma(\lam,\bold v) dS=0,\qq \int_{\p\O}\Sigma(\bold v) dS=0,
$$
where the first equality comes from the divergence theorem because $\div\H_\O=0$ in $\O$, and the second equality comes from the last condition in \eqref{eq-w}.

\begin{Lem} Assume $\O$ satisfies $(O)$ with $r\geq 3$ and $0<\a<1$. Then the operator
$$\Sigma~:~ TC^{2+\a}(\p\O,\Bbb R^3)\to \dot C^{2+\a}(\p\O)
$$
is a homeomorphism, where
$$
\dot C^{k+\a}(\p\O)=\{g\in C^{k+\a}(\p\O)~:~ \int_{\p\O} g(x)dS=0\}.
$$
\end{Lem}
\begin{proof} Using \eqref{est-w} we have, for any $\bold v\in \mB^{2+\a}(\p\O)$ and $\w=\mW(\bold v)$,
$$\|\Sigma(\bold v)\|_{C^{2+\a}(\p\O)}=\|\nu\cdot\w\|_{C^{2+\a}(\p\O)}\leq\|\w\|_{C^{2+\a}(\overline{\O^c})}\leq C_1\|\bold v\|_{C^{2+\a}(\p\O)}.
$$
Hence $\Sigma$ is a continuous linear operator.
On the other hand, from Lemma \ref{LemE.3} (ii), for any $g\in \dot C^{2+\a}(\p\O,0)$, \eqref{eq-u} has a unique solution $\u$. Denote $\bold v=\u_T$. Then $\u$ is also a solution of \eqref{eq-w} with boundary data $\bold v=\u_T$. Hence $g=\nu\cdot\u=\Sigma(\bold v)$. Thus $\Sigma$ is surjective, and
$$\aligned
\|\bold v\|_{C^{2+\a}(\p\O)}=&\|\u_T\|_{C^{2+\a}(\p\O)}\leq\|\u\|_{C^{2+\a}(\overline{\O^c})}\leq C_2(\O,\a)\|g\|_{C^{2+\a}(\p\O)}
=C_2\|\Sigma(\bold v)\|_{C^{2+\a}(\p\O)}.
\endaligned
$$
\end{proof}

Now we mention that finding $\mW(\bv)$ satisfying \eqref{solvability-1} is not a trivial question. In fact, if $\w=\mW(\bv)$ satisfies \eqref{solvability-1}, then $\w$ is a solution of
\eq\label{Dirichlet-w}
\left\{\aligned
&\curl\w=\0,\q \div\w=0\q&\text{in }\O^c,\\
&\w=\w_0(\bv)\q&\text{on }\p\O,\\
& \lim_{|x|\to\infty}\w(x)=\0,
\endaligned
\right.
\eeq
where $\w_0(\bold v)=\bold v+\Gamma(\lam,\bold v)-\nu\cdot(\mH^e)^+$.
Note that \eqref{Dirichlet-w} is a div-curl system with a boundary condition of prescribing the full trace, which may not be solvable for an arbitrary boundary data $\w_0$.

\section{Proof of Lemma \ref{Lem7.1}}
\label{AppendixF}

{\it Step 1.} We prove (i). Let $(f,\A)\in C^2(\overline{\O})\times \mathbb
C^{3,0}_t(\overline{\O},\overline{\O^c}, \mathbb R^3)$ be a solution of
\eqref{eqfA}-\eqref{1.4} with $f>0$ on $\overline{\O}$, and
$\H=\lam\,\curl\A$. Then $\H\in C^2(\overline{\O},\mathbb R^3)\cap
C^2_{\loc}(\overline{\O^c},\mathbb R^3)$ and $\H_T\in \mathbb
C^0(\p\O,\mathbb R^3)$, so $\H\in\mathbb C^{2,0}_t(\overline{\O},\overline{\O^c},\mathbb R^3)$.
From \eqref{eqfA} and \eqref{1.4} we see that $(f,\H)$ is a solution
of \eqref{eqfH}-\eqref{cond-1.12}. If in addition $\A\in \mathbb
C_t^{3,1}(\overline{\O},\overline{\O^c},\mathbb R^3)$, then $\A_T\in\mathbb
C^1(\p\O,\mathbb R^3)$, and hence
$\nu\cdot\H=\lam\,\nu\cdot\curl(\A_T)\in \mathbb C^0(\p\O)$,  so $\H\in
\mathbb C^{2,0}(\overline{\O},\overline{\O^c},\mathbb R^3)$.

{\it Step 2}. We prove (iia). Assume  $(f,\H)\in C^{2+\a}(\overline{\O})\times
\mathbb C^{3+\a,0}(\overline{\O},\overline{\O^c},\mathbb R^3)$ is a solution of
\eqref{eqfH}-\eqref{cond-1.12} with $f>0$ on $\overline{\O}$ and $0<\a<1$. Applying the Schauder estimate to the Neumann problem for $f$ (see \eqref{eq4.7})
we conclude that $f\in C^{3+\a}(\overline{\O})$.

Since $\O$ is simply-connected and  without holes, and $\H\in C^{3+\a}(\overline{\O},\div0)$,
there exists a unique $\B\in C^{2+\a}_{n0}(\overline{\O},\div0)$ solving \eqref{eq-BH}. Let $\bold
Q=\lam^2f^{-2}\curl\H+\B$. Then $\bold Q\in C^{2+\a}(\overline{\O},\mathbb
R^3)$. From \eqref{eqfH} $\curl \bold Q=\bold 0$, so $\bold Q=\nabla\varphi$ for some $\varphi\in
C^{3+\a}(\overline{\O})$, hence
\begin{equation}\label{F.1}
\bold B-\nabla\varphi+\lam^2f^{-2}\curl\H=\bold 0.
\end{equation}
Let
\begin{equation}\label{F.2}
\A^i=\lam^{-1}(\B-\nabla\varphi).
\end{equation}
Then $\A^i\in C^{2+\a}(\overline{\O},\mathbb R^3)$ and $\A^i_T\in
T\!C^{2+\a}(\p\O,\mathbb R^3)$. From the first two equalities in \eqref{eqfH}
we see that $(f,\A^i)$ satisfies the first two equalities in \eqref{eqfA}.

For the $\H$ given above, we can find $\mB\in
C^{2+\a}_{\loc}(\overline{\O^c},\mathbb R^3)\cap C^{1+\delta}(\overline{\O^c},\mathbb
R^3)$ satisfying
\begin{equation}\label{F.3}
\left\{\aligned
&\lam\,\curl\mB=\H-\mH^e,\q \div\mB=0\q&\text{\rm  in }\O^c,\\
&\mB_T=\A^i_T-\mF^e_T\q&\text{\rm on }\p\O,\\
&\int_{\p\O}\nu\cdot\mB dS=0.
\endaligned\right.
\end{equation}
In fact, using \eqref{cond-7.1} and arguing as in the proof of \cite[Lemma 3.3]{P3} we can verify that
$$
\lam\,\nu\cdot\curl[\nu\times(\A^i-\mF^e)]=-\nu\cdot(\H-\mH^e),
$$
and for any closed and oriented surface $\Sigma\subset\O^c$ it holds that
$$
\int_\Sigma \nu_\Sigma\cdot(\H-\mH^e)dS=0.
$$
Recalling that $\O$ is simply-connected and has no holes, applying \cite[Theorem 3.3]{NW} we see that \eqref{F.3} has a
solution $\mB\in C^{1+\delta}(\overline{\O^c},\mathbb R^3)$ which decays at infinity.
Using local estimate as in the proof of \cite[Lemma 3.3]{P3} we can show that $\mB\in
C^{2+\a}_{\loc}(\overline{\O^c},\mathbb R^3)$.

Define $\tilde\A^o$ on $\overline{\O^c}$ by letting
$\tilde\A^o=\mF^e+\mB.$ Then $\tilde\A^o\in
C^{2+\a}_{\loc}(\overline{\O^c},\mathbb R^3)\cap C^{1+\delta}(\overline{\O^c},\mathbb R^3)$.
Since
$(\A^i)^-
-(\tilde\A^o)^+\in C^{2+\a}(\p\O,\mathbb R^3)$, we can find
$\phi\in C^{3+\a}_{\loc}(\overline{\O^c})$ satisfying
\begin{equation} \label{F.4}
\phi=0\q\text{\rm and}\q {\p\phi\over\p\nu}=\nu\cdot[(\A^i)^--(\tilde\A^o)^+]\q\text{\rm on }\p\O.
\end{equation}
Existence of $\phi$ can be proved by using the trace theorem of $H^2(\O^c)$.
Set $\A^o=\tilde\A^o+\nabla\phi$, and define a vector field $\A$ on $\Bbb R^3$ by letting
$\A=\A^i$ in $\O$ and $\A=\A^o$ in $\O^c$.
Then $\A$ satisfies the first three equations in \eqref{eqfA},
$\A\in \C^{2+\a,0}(\overline{\O},\overline{\O^c},\mathbb R^3)$ and
$[\A]=\0$ on $\p\O$. So $(f,\A)$ is a solution of \eqref{eqfA}-\eqref{1.4}. Since $\A^i\in C^{2+\a}(\overline{\O},\mathbb R^3)$ and
$\A^o=\mF^e+\mB+\nabla\phi\in
C^{2+\a}_{\loc}(\overline{\O^c},\mathbb R^3)$, we have $\A\in
\C^{2+\a,0}(\overline{\O},\overline{\O^c},\mathbb R^3)$.

{\it Step 3}. We prove (iib). Since
$\H\in \mathbb C^{2+\a,0}(\overline{\O},\overline{\O^c},\mathbb R^3)$ and
$\H-\mH^e\in C^2_{\loc}(\overline{\O^c},\mathbb R^3)$, we have
$\mH^e\in C^2_{\loc}(\overline{\O^c},\mathbb R^3)$.
Define $\B$ and $\A^i$ as in step
2, and let $\varphi$ be the function in \eqref{F.1}.  Since $\H\in C^{2+\a}(\overline{\O},\mathbb R^3)$, from \cite[Lemma 2.1 (ii)]{P5} we have $\B\in
C^{3+\a}(\overline{\O},\mathbb R^3)$, and from \eqref{F.1} we see $\varphi\in C^{2+\a}(\overline{\O})$.  Since $\nu\cdot\curl\H=\0$, from \eqref{F.1} we
see that $\varphi$ satisfies
\begin{equation}\label{F.5}
\Delta\varphi=\lam^2\nabla (f^{-2})\cdot\curl\H\q\text{in }\O,\q
{\p\varphi\over\p\nu}=0\q\text{\rm on}\;\;\p\O.
\end{equation}
The right side of the first equation in \eqref{F.5} belongs to $C^{1+\a}(\overline{\O})$, hence
$\varphi\in C^{3+\a}(\overline{\O})$. Then from \eqref{F.2} $\A^i\in
C^{2+\a}(\overline{\O},\mathbb R^3)$, so $\A^i_T\in
T\!C^{2+\a}(\p\O,\Bbb R^3)$. Let $\mB$ be the solution of
\eqref{F.3}. Since $\mH^e\in
C^{2}_{\loc}(\overline{\O^c},\mathbb R^3)$ and $\mF^e_T\in
C^{2+\a}(\p\O,\Bbb R^3)$, as in step 2 we can show $\mB\in C^{2+\a}_{\loc}(\overline{\O^c},\mathbb R^3)\cap
C^{1+\delta}(\overline{\O^c},\mathbb R^3)$. Then $\tilde\A^o=\mathcal
F^e+\mB\in C^{2+\a}_{\loc}(\overline{\O^c},\mathbb R^3)$. So
$(\A^i)^--(\tilde\A^o)^+\in C^{2+\a}(\p\O,\Bbb R^3)$, and we can
find a function $\phi\in C^{3+\a}_{\loc}(\overline{\O^c})$ which satisfies
\eqref{F.4}. Thus $\A^o=\tilde\A^o+\nabla\phi\in
C^{2+\a}_{\loc}(\overline{\O^c},\mathbb R^3)$. Define $\A$ as in Step 2. We see that $\A\in \mathbb C^{2+\a,0}(\overline{\O},\overline{\O^c},\mathbb
R^3)$, and $(f,\A)$ is a solution of problem \eqref{eqfA}-\eqref{1.4}.
From the condition $\nu\cdot\curl\H=0$ on $\p\O$ , the fact $\A\in
C^2(\overline{\O},\mathbb R^3)$, and using the second equation in \eqref{eqfA}, we see that $\nu\cdot\A=0$ on $\p\O$. Thus
$(f,\A)$ is a Meissner solution of problem \eqref{eqfA}-\eqref{1.4}.
\qed

\section{Notations}

$$\aligned
&\mA(\O,\mathbb R^3),\;\; \mA(\O,\mathbb R^3,\lam^{-1}\mH^e) \qqq\qqq\qqq\q &\eqref{sp-AB}\\
&\mB(\O,\mathbb R^3)\qqq\qq\;\;\, &\eqref{sp-AB}\\
&\mB^{k+\a}(\p\O),\;\;\mB^{k+\a}(\overline{\O}) \qq\qq\q\,\,&\eqref{sp-Bkalpa-1}, \eqref{sp-Bkalpa-2}\\
&\C^{k+\b}(\p\O,\Bbb R^3)\qqq\q &\text{Definition } \ref{Def3.7}\\
&\C^{k+\a,m+\b}(\overline{\O},\overline{\O^c},\mathbb R^3),\;\;
\C_t^{k+\a,m+\b}(\overline{\O},\overline{\O^c},\mathbb R^3)\qq  &\eqref{sp-Ckalpha}\\
&\text{Conditions }(F), (H), (H_0), (O)\qqq\; &\text{Subsection }\ref{Subsection2.3}\\
&H^1_{0,\loc}(\O^c)\qqq\qq\q &\text{beginning of Section }\ref{Section5}\\
&\mH(\O,\div),\q \mH(\O,\div0)\qqq\qq\;\;\;&\text{section 2}\\
&\mH(\O,\curl),\q \mH(\O,\curl0)\qqq\qq\;\;\,&\text{section 2}\\
&\mH_{\loc}(D,\text{\rm div}),\;\; \mH_{\loc}(D,\text{\rm curl})\qqq\qq\;&\eqref{sp-Hloc}\\
&\mH(\O,\mathbb R^3),\;\; \mH(\O,\mathbb R^3,\lam^{-1}\mH^e)\qqq\;\;&\eqref{sp-HHU}\\
&\mK(\O),\;\; \overline{\mK}(\O),\;\; \mK_\delta(\O),\;\; \mK^1_\delta(\O)\qqq\q\;\,&\eqref{sp-KK}\\
&\mathbb K(\O),\;\; \overline{\mathbb K}(\O),\;\; \mathbb K_\delta(\O)\qq\, &\eqref{sp-XKUU}\\
&T\!H^s(\O,\Bbb R^3),\;\; T\!C^{k+\a}(\p\O,\Bbb R^3)\qq&\eqref{TanS}\\
&\mathbb U(\O),\;\;\overline{\mathbb U}(\O),\;\;\mathbb U_\delta(\O)\qq\q &\eqref{sp-XKUU}\\
&\mU(\O,\mathbb R^3) \qqq\qqq &\eqref{sp-HHU}\\
&\mV(\O)\qqq\qqq\qq&\eqref{sp-WWVD}\\
&\mW(\O),\;\; \mW_{t0}(\O)\qqq\qqq\q&\eqref{sp-WWVD}\\
&\mZ(\O^c)\qqq\qqq\,&\eqref{sp-Z}\\
&\lam_{f\A}(\O,\mB_T),\;\; \k_{f\A}(\O,\mB_T,\lam)\qqq\qqq &\text{Theorem \ref{Thm4.9}}\\
&\lam_\H(\O,\mH_T)\qqq&\text{Lemma \ref{Lem6.2}}\\
&\lam_\A(\O,\var_0)\qqq &\eqref{lam-A-var}\\
&\lam_{f\A}(\O,\var_0),\;\; \k_{f\A}(\O,\var_0,\lam)\qqq&\eqref{lam-ka-var}\\
\endaligned
$$

\v0.3in

%\section*{Compliance with Ethical Standards}
%
%
%Funding:  This study was funded by the National Natural Science Foundation of China grants no.  11671143 and no. 11431005.
%
%Conflict of Interest: The author  declares that he has no conflict of interest.
%
%Ethical approval: This article does not contain any studies with human participants or animals performed by the author.
%
%\vspace {0.3cm}

\end{document}